\documentclass{amsart}
\usepackage{amssymb,stmaryrd,float}
\usepackage{arydshln}
\usepackage[shortlabels]{enumitem}

\setlist[enumerate,1]{label={(\roman*)}} 

\newtheorem{thm}{Theorem}[section]
\newtheorem{mainthm}{Theorem}

\newtheorem{lemma}[thm]{Lemma}
\newtheorem{cor}[thm]{Corollary}
\newtheorem{claim}{Claim}[thm]

\newtheorem{prop}[thm]{Proposition}
\newtheorem{theorem}[thm]{Theorem}
 
\theoremstyle{definition}
\newtheorem{defn}[thm]{Definition}
\newtheorem{defi}{Definition}

\newtheorem{conv}[thm]{Convention}

\theoremstyle{remark}
\newtheorem{remark}[thm]{Remark}

\newenvironment{why}[1][Proof]{\proof[#1]\mbox{}}{\endproof}

\DeclareMathOperator\cspec{Cspec}
\DeclareMathOperator\ssup{ssup}
\DeclareMathOperator\tcf{tcf}
\DeclareMathOperator\sep{Sep}
\DeclareMathOperator\ad{AD}
\DeclareMathOperator\ubd{{\sf unbounded}}
\DeclareMathOperator\onto{{\sf onto}}
\DeclareMathOperator\proj{{\sf projection}}
\DeclareMathOperator{\reg}{Reg}
\DeclareMathOperator{\sing}{Sing}

\DeclareMathOperator{\cf}{cf}

\DeclareMathOperator{\im}{Im}
\DeclareMathOperator{\otp}{otp}
\DeclareMathOperator{\acc}{acc}
\DeclareMathOperator{\nacc}{nacc}

\DeclareMathOperator{\Tr}{Tr}
\DeclareMathOperator{\tr}{tr}
\DeclareMathOperator{\p}{P}
\DeclareMathOperator{\PP}{PP}
\DeclareMathOperator{\U}{U}
\DeclareMathOperator{\cov}{cov}
\DeclareMathOperator{\cof}{cof}

\newcommand\self{{\circlearrowleft}}
\newcommand\s{\subseteq}
\newcommand\sq{\sqsubseteq}
\newcommand\br{\blacktriangleright}
\newcommand\last[2]{\eth_{#1,#2}}
\newcommand\stick{{{\ensuremath \mspace{2mu}\mid\mspace{-12mu} {\raise0.6em\hbox{$\bullet$}}}}}
\newcommand\on{\text{ORD}}
\newcommand*\axiomfont[1]{\textsf{\textup{#1}}}

\newcommand\ssh{\axiomfont{SSH}}

\newcommand\ns{\textup{NS}}
\newcommand\sa{\textup{SA}}
\newcommand\bd{\textup{bd}}
\newcommand\tree{\mathbf{T}}
\renewcommand\restriction{\mathbin\upharpoonright}
\renewcommand\mid{\mathrel{|}\allowbreak}

\author{Tanmay Inamdar}
\address{Department of Mathematics, Bar-Ilan University, Ramat-Gan 5290002, Israel.}
\curraddr{Department of Mathematics, Ben-Gurion University of the Negev, P.O.B. 653, Be’er Sheva, 84105 Israel}
\urladdr{https://scholar.google.com/citations?user=RDXg2z8AAAAJ}

\author{Assaf Rinot}
\address{Department of Mathematics, Bar-Ilan University, Ramat-Gan 5290002, Israel.}
\urladdr{https://www.assafrinot.com}

\title[Was Ulam right? II]{Was Ulam right? II:\\ Small width and general ideals} 
\date{Preprint as of December 17, 2023. For the latest version, visit \textsf{http://p.assafrinot.com/53}.}
\subjclass[2010]{Primary 03E02; Secondary 03E35, 03E55}

\begin{document}
\dedicatory{This paper is dedicated to W.~F.~Sierpi\'nski on the occasion of his $140^{th}$ birthday}

\begin{abstract}
We continue our study of Sierpi\'nski-type colourings.
In contrast to the prequel paper, we focus here on colourings for ideals stratified by their completeness degree.
In particular, improving upon Ulam's theorem and its extension by Hajnal,
it is proved that if $\kappa$ is a regular uncountable cardinal that is not weakly compact in $L$,
then there is a universal witness for non-weak-saturation of $\kappa$-complete ideals. Specifically,
there are $\kappa$-many decompositions of $\kappa$ such that, for every $\kappa$-complete ideal $J$ over $\kappa$,
and every $B\in J^+$, one of the decompositions shatters $B$ into $\kappa$-many $J^+$-sets.

A second focus here is the feature of narrowness of colourings, one already present in the theorem of Sierpi\'nski.
This feature ensures that a colouring suitable for an ideal is also suitable for all superideals possessing the requisite completeness degree.
It is proved that unlike successors of regulars, every successor of a singular cardinal admits such a narrow colouring.

\end{abstract}
\maketitle
\section{Introduction}
Throughout this paper, $\kappa$ denotes an infinite cardinal
and $\lambda,\mu,\nu$ denote infinite cardinals $\le\kappa$,
and $\theta$ is a cardinal with $2\le\theta\le\kappa$. Also,
$J^{\bd}[\kappa]$ stands for the ideal of bounded subsets of $\kappa$,
and we let
$$\mathcal J^\kappa_\nu:=\{ J\mid J\text{ is a }\nu\text{-complete ideal over }\kappa\text{ extending }J^{\bd}[\kappa]\}.$$
Some additional conventions are listed in Subsection~\ref{nandc} below.

In the prequel paper \cite{paper47}, we initiated the systematic study of Sierpi\'nski-type colourings.
To exemplify, for an ideal $J$ over $\kappa$, the principle  $\onto(J,\theta)$ asserts the existence of a colouring $c:[\kappa]^2\rightarrow\theta$ with the property that for
every $J$-positive set $B$,
there exists some $\eta<\kappa$ such that $c[\{\eta\}\circledast B]=\theta$.

Sierpi\'nski-type colourings are quite natural,
and they can be used to characterize large cardinals (for example, a regular uncountable cardinal $\kappa$ is almost ineffable iff $\onto(J^{\bd}[\kappa],2)$ fails,
and is weakly compact iff $\onto(J^{\bd}[\kappa],3)$ fails),
but our motivation comes from the problem of finding sufficient conditions for any $J$-positive set of a given ideal $J\in\mathcal J^\kappa_\kappa$ to have a decomposition into $\kappa$-many $J$-positive sets.
Probably the earliest such sufficient condition was discovered by Ulam in \cite{ulam1930masstheorie}.
Ulam proved that any successor cardinal $\kappa=\lambda^+$ admits an \emph{Ulam matrix}, that is,
an array $\langle U_{\eta,\tau}\mid \eta<\lambda, \tau<\kappa\rangle$ of subsets of $\kappa$
such that:
\begin{itemize}
\item For every $\eta<\lambda$, $\langle U_{\eta, \tau}\mid \tau<\kappa\rangle$ consists of pairwise disjoint sets;
\item For every $\tau<\kappa$, $|\kappa\setminus \bigcup_{\eta<\lambda}U_{\eta,\tau}|<\kappa$.
\end{itemize}

In particular, given $J\in\mathcal J^\kappa_\kappa$,
each row $\vec {U_\eta}=\langle U_{\eta, \tau} \mid \tau < \kappa\rangle$ shatters any $B\in J^+$ into $\kappa$-many disjoint sets,
and there must exist an $\eta<\lambda$ such that $\{\tau < \kappa\mid U_{\eta, \tau}\cap B\in J^+\}$ has order-type $\kappa$.
Put differently, there is a $\lambda$-list of candidates such that  any element of $J^+$ is partitioned into $\kappa$-many disjoint pieces,
$\kappa$ of which are in $J^+$, by at least one of the candidates.

Later, in \cite{MR260597}, Hajnal extended the idea of Ulam to accommodate some limit cardinals as well. He showed that for every inaccessible cardinal $\kappa$
admitting a stationary set that does not reflect at regulars, there exists
a \emph{triangular Ulam matrix}, that is, an array $\langle U_{\eta,\tau}\mid \eta<\tau<\kappa\rangle$ of subsets of $\kappa$ such that:
\begin{itemize}
\item For every $\eta<\kappa$, $\langle U_{\eta, \tau}\mid \eta<\tau<\kappa\rangle$ consists of pairwise disjoint sets;
\item For stationarily many $\tau<\kappa$, $|\kappa\setminus(\bigcup_{\eta<\tau}U_{\eta,\tau})|<\kappa$.
\end{itemize}

The existence of a triangular Ulam matrix at $\kappa$ is in fact equivalent to
the existence of a stationary subset of $\kappa$ that does not reflect at regulars,
but our focus here is its primary consequence.
The question we consider is the following:  given an ideal $J\in\mathcal J^\kappa_\kappa$,
are there other means to obtain a $\kappa$-list of candidates such that
any element of $J^+$ is partitioned into $\kappa$-many disjoint pieces, $\kappa$ of which are in $J^+$, by at least one of the candidates?

To streamline the discussion, we rephrase things in the language of colourings:
\begin{defi}[\cite{paper47}]\label{def1} Let $J$ be an ideal over $\kappa$.
\begin{itemize}
\item $\onto^+(J,\theta)$  asserts the existence of a colouring $c:[\kappa]^2\rightarrow\theta$
such that
for every $B\in J^+$, there is an $\eta<\kappa$ such that
$$\{\tau<\theta\mid \{\beta\in B\setminus(\eta+1)\mid c(\eta,\beta)=\tau\}\in J^+\}=\theta;$$
\item  $\ubd^+(J,\theta)$  asserts the existence of a colouring $c:[\kappa]^2\rightarrow\theta$
that is upper-regressive (i.e., $c(\alpha,\beta)<\beta$ for all $\alpha<\beta<\kappa$)
with the property that,
for every $B\in J^+$, there is an $\eta<\kappa$ such that
$$\otp(\{\tau<\theta\mid \{\beta\in B\setminus(\eta+1)\mid c(\eta,\beta)=\tau\}\in J^+\})=\theta.$$
\end{itemize}
For a collection $\mathcal J$ of ideals over $\kappa$, we write $\onto^+(\mathcal J,\theta)$  to assert the existence
of a colouring $c:[\kappa]^2\rightarrow\theta$ simultaneously witnessing $\bigwedge_{J\in\mathcal J}\onto^+(J,\kappa)$.
The same convention applies to $\ubd^+$.
\end{defi}

By the above-mentioned theorems of Ulam and Hajnal, if $\kappa$ is a regular uncountable cardinal and $\ubd^+(\mathcal J^\kappa_\kappa,\kappa)$ fails,
then $\kappa$ is a Mahlo cardinal all of whose stationary subsets reflect at inaccessibles.\footnote{By the work of Mekler and Shelah \cite{MkSh:367}, the consistency strength of this reflection principle is weaker than that of a weakly compact cardinal.}
The first main result of this paper improves it, deriving
that $\kappa$ is a greatly Mahlo cardinal that is weakly compact in $L$,
and for every sequence $\langle S_i\mid i<\kappa\rangle$ of stationary subsets of $\kappa$,
there exists an inaccessible $\alpha<\kappa$ such that $S_i\cap\alpha$ is stationary in $\alpha$ for every $i<\alpha$.
The result is obtained by drawing a connection between the theory of walks on ordinals and the problem of decomposing a positive set of an ideal in $\mathcal J^\kappa_\kappa$ into the maximal number of positive pieces.
It reads as follows.
\begin{mainthm}\label{thmb}
For a regular uncountable cardinal $\kappa$, the following are equivalent:
\begin{enumerate}[(1)]
\item $\ubd^+(J^{\bd}[\kappa],\kappa)$ holds;
\item $\ubd^+(\mathcal J^\kappa_\kappa,\kappa)$ holds;
\item $\kappa$ admits a nontrivial $C$-sequence in the sense of \cite[Definition~6.3.1]{TodWalks}.
\end{enumerate}
\end{mainthm}

Corollary~\ref{cor55} below gives a pumping-up condition, showing that for every infinite regular cardinal $\theta$ such that $\mathfrak b_\theta=\theta^+$, if $\onto^+(J^{\bd}[\theta],\theta)$ holds,
then so does $\onto^+(J^\bd[\theta^+],\theta^+)$.
Corollary~\ref{cmonotone} below uncovers a downward monotonicity feature, showing that
for every $\mathcal J\s \mathcal J^\kappa_\omega$, $\ubd^+(\mathcal J,\allowbreak\theta)$ implies
$\onto^+(\mathcal J,\allowbreak\vartheta)$ for all regular $\vartheta<\theta$.
When put together with Theorem~\ref{thmb}, this shows that if $\theta\le\kappa$ is a pair of infinite regular cardinals such that
$\onto^+(J^{\bd}[\kappa],\theta)$ fails, then either $\theta=\kappa$ or $\kappa$ is greatly Mahlo.

Now, let us address the harder problem of finding \emph{disjoint refinements}:
given an ideal $J\in\mathcal J^\kappa_\kappa$ and a sequence $\vec B=\langle B_\tau\mid \tau<\theta\rangle$ of $J^+$-sets,
is there a pairwise disjoint sequence $\langle A_\tau\mid \tau<\theta\rangle$ of $J^+$-sets such that $A_\tau\s B_\tau$ for all $\tau<\theta$?

While (classical and triangular) Ulam matrices provide a disjoint refinement for any constant sequence $\vec B$,
they fail to handle more general sequences. Here, we show that from the same hypothesis under which Ulam matrices exist,
there does exist a universal refining matrix, as follows.
\begin{mainthm}\label{thmc}
Suppose that $\kappa$ is a regular uncountable cardinal admitting a stationary set that does not reflect at regulars.
Then, for every cardinal $\theta<\kappa$, $\onto^{++}(\mathcal J^\kappa_\kappa,\theta)$ holds.\footnote{Here, one cannot take $\theta=\kappa$, since, by \cite[Proposition~9.11]{paper47}, $\onto^{++}(J^{\bd}[\kappa],\kappa)$ fails.}
That is, there exists a colouring $c:[\kappa]^2\rightarrow\theta$
such that for every $J\in\mathcal J^\kappa_\kappa$ and every sequence
$\vec B=\langle B_\tau\mid \tau<\theta\rangle$ of $J^+$-sets,
there exists an $\eta<\kappa$ such that
$$\{\tau<\theta\mid \{\beta\in B_\tau\setminus(\eta+1)\mid c(\eta,\beta)=\tau\}\in J^+\}=\theta.$$
\end{mainthm}

Finally, let us turn our attention to \emph{narrow} colourings.
For a principle $\mathsf p\in\{\onto^+,\allowbreak\ubd^+\}$,
the instance $\mathsf p(\{\lambda\},\allowbreak J,\theta)$
is obtained by requiring in Definition~\ref{def1} that the ordinal $\eta$ be chosen below $\lambda$.
By Proposition~\ref{prop26} below, if $\mathsf p(\{\lambda\},J^{\bd}[\kappa],\theta)$ holds for a regular cardinal $\kappa$,
then $\theta\le\kappa\le 2^\lambda$.
More importantly, by Proposition~\ref{prop614}  below, if $\theta\le\lambda<\kappa$, then any colouring witnessing $\mathsf p(\{\lambda\},J^{\bd}[\kappa],\allowbreak \theta)$
moreover witnesses $\mathsf p(\{\lambda\}, \mathcal J^\kappa_{\lambda^+},\theta)$.
The following provides a sample of the results obtained here on narrow colourings.

\begin{mainthm}\label{thmd} Assume $\theta<\lambda<\cf(\kappa)\le2^\lambda$.
\begin{enumerate}[(1)]
\item If $\lambda$ is singular, then $\ubd^+(\{\lambda\},J^{\bd}[\kappa],\lambda)$ holds for $\kappa=\lambda^+$;
\item If $\lambda$ is regular, then $\ubd^+(\{\lambda\},J^{\bd}[\kappa],\lambda)$ holds for $\kappa\in\{\mathfrak b_\lambda,\mathfrak d_\lambda\}$;
\item If $\lambda$ is singular, then $\onto^+(\{\lambda\},J^{\bd}[\kappa],\theta)$ holds for $\kappa\in\PP(\lambda)$;
\item If $\lambda=\lambda^\theta$ or if $\lambda$ is a strong limit, then $\onto^+(\{\lambda\},J^{\bd}[\kappa],\theta)$ holds;
\item If $\lambda=\theta^+$ and $\mathfrak b_\lambda=\mathfrak d_\lambda=\kappa$, then $\ubd^+(\{\lambda\},J^{\bd}[\kappa],\theta)$ holds;
\item If $\lambda>\theta^{++}$, then $\onto^+(\{\lambda\},J^{\bd}[\kappa],\theta^+)$ holds for $\kappa=\lambda^+$.
\end{enumerate}
\end{mainthm}

\subsection{Organization of this paper}
In Section~\ref{basics}, we recall the colouring principles from \cite[\S2]{paper47},
identify some inconsistent instances, and demonstrate the utility of narrow colourings.

In Section~\ref{subnormalrevisited}, we bring in walks on ordinals in order to improve the results of \cite[\S3]{paper47}.
The equivalence between Clauses (1) and (3) of Theorem~\ref{thmb} will be established there.

In Section~\ref{pumings}, we shall draw some implications between the various instances of our colouring principles.
In particular we prove pumping-up theorems for these principles using the concept of \emph{projections}, as well as establish a monotonicity result between them.
The equivalence between Clauses (1) and (2) of Theorem~\ref{thmb} will be established there.

In Section~\ref{partitions}, we compare our colouring principles with other well-studied principles such as $\kappa\nrightarrow[\kappa]^2_\theta$ and $\U(\kappa,\mu,\theta,\chi)$,
improving upon results from \cite{paper47} that were limited to \emph{subnormal} ideals.
It is proved that $\onto^{++}([\kappa]^\kappa,\mathcal J^\kappa_\kappa,\theta)$ holds for $\kappa=\theta^+$ with $\theta$ regular,
that $\ubd^{++}([\kappa]^\kappa,\mathcal J^\kappa_\kappa,\theta)$ holds for $\kappa=\theta^+$ with $\theta$ singular,
and that $\onto^+(\mathcal J^\kappa_\kappa,\omega)$ holds for every regular cardinal $\kappa\ge\mathfrak d$ that is not weakly compact.
Additional instances are shown to hold in the presence of Shelah's Strong Hypothesis (\ssh).
The proof of Theorem~\ref{thmc} will be found there.

In Section~\ref{sectionscales}, we obtain narrow colourings using scales for regular and singular cardinals.
Some of the clauses of Theorem~\ref{thmd} are proved there.

In Section~\ref{adsection}, we obtain more narrow colourings, this time using independent families, almost-disjoint families,
and trees. The remaining clauses of Theorem~\ref{thmd} are proven there.
We also show that in contrast with the monotonicity result of Section~\ref{pumings},
narrow colourings are not downwards nor upwards monotone.

In the Appendix, we provide a concise index for many of the results of this paper.
\subsection{Notation and conventions}\label{nandc}
The dual filter of an ideal $J$ over $\kappa$ is denoted by $J^*:=\{ \kappa\setminus X\mid X\in J\}$,
and the collection of $J$-positive sets is denoted by $J^+:=\mathcal P(\kappa)\setminus J$.
Let $\reg(\kappa)$ denote the collection of all infinite regular cardinals below $\kappa$.
Let $E^\kappa_\theta:=\{\alpha < \kappa \mid \cf(\alpha) = \theta\}$,
and define $E^\kappa_{\le \theta}$, $E^\kappa_{<\theta}$, $E^\kappa_{\ge \theta}$, $E^\kappa_{>\theta}$,  $E^\kappa_{\neq\theta}$ analogously.
For a set of ordinals $A$, we write $\ssup(A) := \sup\{\alpha + 1 \mid \alpha \in A\}$, $\acc^+(A) := \{\alpha < \ssup(A) \mid \sup(A \cap \alpha) = \alpha > 0\}$,
$\acc(A) := A \cap \acc^+(A)$, and $\nacc(A) := A \setminus \acc(A)$.
For a stationary $S\s \kappa$, we write
$\Tr(S):= \{\alpha \in E^\kappa_{>\omega}\mid  S\cap \alpha\text{ is stationary in }\alpha\}$.
A (\emph{$\xi$-bounded}) \emph{$C$-sequence} over $S$ is a sequence $\vec C=\langle C_\beta\mid\beta\in S\rangle$
such that, for every $\beta\in S$, $C_\beta$ is a closed subset of $\beta$ with $\sup(C_\beta)=\sup(\beta)$ (and $\otp(C_\beta)\le\xi$).
For $A,B$ sets of ordinals, we denote $A\circledast B:=\{(\alpha,\beta)\in A\times B\mid \alpha<\beta\}$
and we identify $[B]^2$ with $B\circledast B$.
In particular, we interpret the domain of a colouring $c:[\kappa]^2\rightarrow\theta$ as a collection of ordered pairs.
In scenarios in which we are given an unordered pair $p=\{\alpha,\beta\}$,
we shall write $c(\{\alpha,\beta\})$ for $c(\min(p),\max(p))$. We also agree to interpret $c(\{\alpha,\beta\})$ as $0$, whenever $\alpha=\beta$.
For $\theta\neq 2$, $[\kappa]^\theta$ stands for the collection of all subsets of $\kappa$ of size $\theta$.

The definitions of cardinal invariants such as $\mathfrak t,\mathfrak b,\mathfrak d$ and their higher generalizations may be found in \cite{MR1355135,ShSj:643,MR2768685}.
The definitions of square principles such as $\square(\kappa,{<}\mu)$ may be found in \cite[\S1]{paper29}.
Our trees conventions follow that of \cite[\S2.1]{paper23}.
Our \emph{walks on ordinals} notation follows that of \cite[\S4.2]{paper34},
and the reader is further referred to \cite{TodWalks} for a comprehensive treatment of the subject.

\section{Colouring principles and the utility of narrowness}\label{basics}
In this section, we recall the colouring principles from  \cite[\S2]{paper47},
identify a few inconsistent instances,
and demonstrate the utility of narrow colourings.

\begin{defn}\label{defonto} For a family $\mathcal A\s\mathcal P(\kappa)$ and
an ideal $J$ over $\kappa$:
\begin{itemize}
\item $\onto^{++}(\mathcal A,J,\theta)$ asserts
the existence of a colouring $c:[\kappa]^2\rightarrow\theta$ with the property that,
for every $A\in\mathcal A$ and every sequence $\langle B_\tau\mid \tau<\theta\rangle$ of elements of $J^+$,
there is an $\eta\in A$ such that $\{ \beta\in B_\tau\setminus(\eta+1)\mid c(\eta,\beta)=\tau\}\in J^+$ for every $\tau<\theta$;
\item $\onto^+(\mathcal A,J,\theta)$ asserts
the existence of a colouring $c:[\kappa]^2\rightarrow\theta$ with the property that,
for all $A\in\mathcal A$ and $B\in J^+$,
there is an $\eta\in A$ such that, for every $\tau<\theta$, $\{\beta\in B\setminus(\eta+1)\mid c(\eta,\beta)=\tau\}\in J^+$;
\item $\onto(\mathcal A,J,\theta)$ asserts
the existence of a colouring $c:[\kappa]^2\rightarrow\theta$ with the property that,
for all $A\in\mathcal A$ and $B\in J^+$,
there is an $\eta\in A$ such that $$c[\{\eta\}\circledast B]=\theta;$$
\item $\onto^-(\mathcal A,J,\theta)$ asserts
the existence of a colouring $c:[\kappa]^2\rightarrow\theta$ with the property that,
for all $A\in\mathcal A$ and $C\in J^*$,
and for every regressive map $f:C\rightarrow\kappa$,
there are $\eta\in A$ and $\bar\kappa<\kappa$ such that
$$c[\{\eta\}\circledast\{\beta\in C\mid f(\beta)=\bar\kappa\}]=\theta.$$
\end{itemize}
\end{defn}
\begin{conv}\label{colouringconventions}\begin{enumerate}
\item For a principle $\mathsf p\in\{\onto^{++},\onto^+,\onto\}$  and a collection of ideals $\mathcal J$ over $\kappa$, we write $\mathsf p(\mathcal A,\mathcal J,\theta)$
to assert the existence of a colouring simultaneously witnessing $\mathsf p(\mathcal A,J,\theta)$  for all $J\in\mathcal J$;
\item If we omit $\mathcal A$, then we mean that $\mathcal A:=\{\kappa\}$;
\item If we put $\self$ instead of $\mathcal A$, then we mean that the sets $A$ and $B$ in the definition of the principle $\mathsf p$ coincide.
So, for example,
$\onto(\self, J, \theta)$ asserts the existence of a colouring $c:[\kappa]^2\rightarrow \theta$ with the property that for every $B \in J^+$ there is an $\eta \in B$ such that
$c[\{\eta\}\circledast B]=\theta$.
\end{enumerate}
\end{conv}

Recall that a colouring $c:[\kappa]^2 \rightarrow \theta$ is \emph{upper-regressive} if $c(\alpha,\beta)<\beta$ for all $\alpha<\beta<\kappa$.
\begin{defn}\label{defubd} For a family $\mathcal A\s\mathcal P(\kappa)$ and an ideal $J$ over $\kappa$:
\begin{itemize}
\item $\ubd^{++}(\mathcal A,J,\theta)$ asserts
the existence of an upper-regressive colouring $c:[\kappa]^2\rightarrow\theta$ with the property that,
for every $A\in\mathcal A$ and every sequence $\langle B_\tau\mid \tau<\theta\rangle$ of elements of $J^+$,
there is an $\eta\in A$ and an injection $h:\theta\rightarrow\theta$ such that,
for every $\tau<\theta$, $\{ \beta\in B_\tau\setminus(\eta+1)\mid c(\eta,\beta)=h(\tau)\}\in J^+$;
\item $\ubd^+(\mathcal A,J,\theta)$ asserts
the existence of an upper-regressive colouring $c:[\kappa]^2\rightarrow\theta$ with the property that,
for all $A\in\mathcal A$ and $B\in J^+$,
there is an $\eta\in A$ such that
$$\otp(\{\tau<\theta\mid \{\beta\in B\setminus(\eta+1)\mid c(\eta,\beta)=\tau\}\in J^+\})=\theta;$$
\item $\ubd(\mathcal A,J,\theta)$ asserts
the existence of an upper-regressive colouring $c:[\kappa]^2\rightarrow\theta$ with the property that,
for all $A\in\mathcal A$ and $B\in J^+$,
there is an $\eta\in A$ such that
$$\otp(c[\{\eta\}\circledast B])=\theta.$$
\end{itemize}
\end{defn}
The conventions of Convention~\ref{colouringconventions} also apply to the principles of Definition~\ref{defubd}.
\begin{remark}
For $\mathsf{p} \in \{\onto, \onto^+,\ubd, \ubd^+\}$, $\mathsf{p}(J^+, J, \theta)$ implies $\mathsf{p}(\self, J, \theta)$ which implies $\mathsf{p}(J^*, J, \theta)$ which in turn implies $\mathsf{p}(J, \theta)$; in case $J$ extends $[\kappa]^{<\kappa}$ then also $\mathsf{p}([\kappa]^\kappa, J, \theta)$ implies $\mathsf{p}(J^+, J, \theta)$.
\end{remark}

By Ulam's celebrated theorem, $\ubd^+(\{\nu\},J^{\bd}[\kappa],\kappa)$ holds for every successor cardinal $\kappa=\nu^+$.
In Corollary~\ref{ulamext} below, we present a consistent strengthening of this result,
but let us first point out that Ulam's theorem is optimal in the sense that, by the next proposition,
$\ubd(\{\nu\},J^{\bd}[\kappa],\kappa)$ fails whenever $\kappa$ is a regular cardinal greater than $\nu^+$.

\begin{prop}\label{prop26} Suppose that $\ubd(\{\nu\},J^{\bd}[\kappa],\theta)$ holds for $\kappa$ regular. Then:

\begin{enumerate}[(1)]
\item $\theta\le\kappa$, and if $\nu<\cf(\theta)$, then $\theta=\kappa$;
\item $\cf(\theta)\le\nu^+$;
\item $\kappa\le 2^\nu$.
\end{enumerate}
\end{prop}
\begin{proof} Let $c:[\kappa]^2\rightarrow\theta$ be a colouring witnessing $\ubd(\{\nu\},J^{\bd}[\kappa],\theta)$.

(1) As $(J^{\bd}[\kappa])^+$ contains a set of size $\kappa$,
the choice of $c$ implies that $\theta\le\kappa$.
Suppose that $\nu<\cf(\theta)$.
Then, for every $\beta\in\kappa\setminus\nu$,
$\sigma_\beta:=\sup\{ c(\eta,\beta)\mid \eta<\nu\}$ is less than $\theta$.
Now, if $\theta<\kappa$, then $\nu<\theta<\kappa$, and for some $\sigma<\theta$,
$B:=\{\beta\in \kappa\setminus\nu\mid \sigma_\beta=\sigma\}$ is a $J^{\bd}[\kappa]$-positive
set that satisfies $\otp(c[\{\eta\}\circledast B])\le\sigma+1<\theta$ for every $\eta<\nu$.
This is a contradiction.

(2) Suppose that $\cf(\theta)>\nu^+$. Then by Clause~(1), $\kappa=\theta>\nu^+$, so that $S:=E^\kappa_{\nu^+}$ is stationary.
As $c$ is upper-regressive,
for every $\beta\in S$, $\sigma_\beta:=\sup\{ c(\eta,\beta)\mid \eta<\nu\}$ is less than $\beta$.
By Fodor's lemma, find some $\sigma<\kappa$ such that
$B:=\{\beta\in S\mid \sigma_\beta=\sigma\}$ is stationary. In particular, $\otp(c[\{\eta\}\circledast B])\le\sigma+1<\kappa=\theta$ for all $\eta<\nu$.
This is a contradiction.

(3) Let $T$ be a cofinal subset of $\theta$ of order-type $\cf(\theta)$.
By Clause~(2), $|T|\le\nu^+$.
As inferring $\kappa \leq 2^\nu$ from $\kappa\le\nu$ is trivial, we may assume that $\nu< \kappa$.
For every $\beta\in [\nu,\kappa)$,
define a function $g_\beta:\nu\rightarrow T$ via $g_\beta(\eta):=\min(T\setminus c(\eta,\beta))$.
Towards a contradiction, suppose that $\kappa>2^\nu$.
As $|{}^\nu T|=2^\nu$, we may pick a cofinal subset $B\s\kappa\setminus \nu$
on which the map $\beta\mapsto g_\beta$ is constant, and let $g \in {}^\nu\theta$ denote this constant value.
So, for every $\eta< \nu$, $\sup(c[\{\eta\}\circledast B])\le g(\eta)<\theta$.
In particular, for every $\eta< \nu$, $\otp(c[\{\eta\}\circledast B]) \leq g(\eta)+1 < \theta$.
This is a contradiction.
\end{proof}

\begin{prop}\label{prop912} Suppose that $J$ is an ideal over $\kappa$ such that $\kappa\notin J$.

For every $\nu\le\theta$, $\ubd^{++}(\{\nu\},J,\theta)$ fails.
\end{prop}
\begin{proof} Given infinite cardinals $\nu\le\theta$,
we fix a map $f:\theta\rightarrow\nu$ such that $f^{-1}\{\eta\}$ is infinite for every $\eta<\nu$.
Next, given a colouring $c:[\kappa]^2\rightarrow\theta$, for every $\eta<\nu$,
if there exists $\tau<\theta$ for which $\{\beta<\kappa\mid c(\eta,\beta)=\tau\}$ is in $J^+$,
then let $\xi_\eta$ be the least such $\tau$, and then let $B^\eta:=\{\beta<\kappa\mid c(\eta,\beta)=\xi_\eta\}$.
Otherwise, let $\xi_\eta:=0$ and $B^\eta:=\kappa$.
Finally, for every $\tau<\theta$, let $B_\tau:=B^{f(\tau)}$.

Towards a contradiction, suppose that there exists $\eta<\nu$
and an injection $h:\theta\rightarrow\theta$ such that, for every $\tau<\theta$, $\{ \beta\in B_\tau\mid c(\eta,\beta)=h(\tau)\}\in J^+$.
In particular, $B^\eta\in J^+\setminus\{\kappa\}$.
By the choice of $f$, let us fix $\tau_0\neq\tau_1$ in $\theta$ such that $f(\tau_0)=\eta=f(\tau_1)$. For each $i<2$:
$$\{\beta\in B^\eta\mid c(\eta,\beta)=h(\tau_i)\}=\{\beta\in B_{\tau_i}\mid c(\eta,\beta)=h(\tau_i)\}\in J^+,$$
and hence $h(\tau_i)=\xi_\eta$. So $h(\tau_0)=h(\tau_1)$,
contradicting the fact that $h$ is injective.
\end{proof}

\begin{prop} Suppose that $\kappa\in E^{\mathfrak t}_\omega$. Then $\onto(J^{\bd}[\kappa],\omega)$ fails.
\end{prop}
\begin{proof}
Let $c:[\kappa]^2\rightarrow\omega$ be any colouring. Fix an injective map $\pi:\omega\rightarrow\kappa$ whose image is cofinal in $\kappa$.
We shall recursively construct a sequence of infinite subsets of $\omega$, $\langle X_\eta\mid \eta\le\kappa\rangle$, as follows.

$\br$ Let $X_0:=\omega$.

$\br$ For every $\eta<\kappa$ such that $X_\eta$ has already been defined,
pick $X_{\eta+1}\in[X_\eta]^\omega$
such that $\omega\setminus c[\{\eta\}\circledast \pi[X_{\eta+1}]]$ is infinite.

$\br$ For every $\eta\in\acc(\kappa+1)$ such that $\langle X_{\bar\eta}\mid \bar\eta<\eta\rangle$ has already been defined,
since $\eta\le\kappa<\mathfrak t$, we may find some set $X_\eta\in[\omega]^\omega$ such that
$X_\eta\setminus X_{\bar\eta}$ is finite, for all $\bar\eta<\eta$.

This completes the recursive construction. As $X_\kappa\in[\omega]^\omega$,
$B:=\pi[X_\kappa]$ is $J^{\bd}[\kappa]$-positive.
Now, for every $\eta<\kappa$, $B\setminus\pi[X_{\eta+1}]$ is finite,
and hence $c[\{\eta\}\circledast B]\neq\omega$.
\end{proof}

A similar idea shows that $\onto([\omega]^\omega,J^{\bd}[\kappa],2)$ fails for every infinite cardinal $\kappa<\mathfrak t$.
By the main result of \cite{Sh:949}, it is also consistent for $\onto([\lambda]^\lambda,J^{\bd}[\lambda^+],2)$ to fail at a singular cardinal $\lambda$.
In contrast, by Theorem~\ref{ek} below, $\onto(\{\lambda\},J^{\bd}[\lambda^+],2)$ does hold for every infinite cardinal $\lambda$.

\begin{defn}\label{defnnarrow}
A colouring witnessing a principle $\mathsf p(\mathcal A,J,\theta)$ as in Definitions \ref{defonto} and \ref{defubd} is said to be \emph{narrow} iff there is an $A\in\mathcal A$ with $|A|<|\bigcup J|$.
\end{defn}
By an argument from \cite[p.~291]{TodActa}, if $\kappa=\nu^+$ for an infinite regular cardinal $\nu$,
then for every $\mathsf p\in\{\ubd,\onto\}$, if $\mathsf p([\nu]^\nu,J^{\bd}[\kappa],\theta)$ holds, then so does $\mathsf p([\kappa]^\nu,J^{\bd}[\kappa],\theta)$.
The same implication remains valid after replacing the regularity of $\nu$ by $\square^*_\nu$.
The next proposition motivates our interest in narrow colourings.
\begin{prop}\label{prop614}  Suppose that $J\in\mathcal J^\kappa_{\nu^+}$ for an infinite cardinal $\nu<\kappa$,
and let $\mathcal A\s[\kappa]^{\le\nu}$ be nonempty.
Consider the collection $\mathcal J:=\{ I\in\mathcal J^\kappa_{\nu^+}\mid I\supseteq J\}$.
\begin{enumerate}[(1)]
\item For every $\theta\le\nu$, $\ubd(\mathcal A,J, \theta)$ implies $\ubd^+(\mathcal A,\mathcal J, \theta)$;
\item For every $\theta\le\kappa$, $\onto(\mathcal A,J, \theta)$ implies $\onto^+(\mathcal A,\mathcal J, \theta)$;
\item For every $\theta\le\kappa$, $\onto^{++}(\mathcal A,J, \theta)$	implies $\onto^{++}(\mathcal A,\mathcal J, \theta)$.
\end{enumerate}
\end{prop}
\begin{proof} Suppose that we are given a colouring $c:[\kappa]^2\rightarrow\theta$.
For all $B\s\kappa$, $\eta<\kappa$ and $\tau<\theta$, denote $B^{\eta,\tau}:=\{\beta \in B\setminus(\eta+1) \mid c(\eta, \beta) = \tau\} $.

(1) Suppose that $c$ witnesses $\ubd(\mathcal A,J, \theta)$ for a given cardinal $\theta\le\nu$.
We claim that $c$ moreover witnesses $\ubd^+(\mathcal A,\mathcal J, \theta)$.
Suppose not, and fix $A\in\mathcal A$, $I\in \mathcal J$ and $B\in I^+$ such that, for every $\eta\in A$,  the set
$T_\eta:=\{\tau<\theta\mid B^{\eta,\tau}\in I^+\}$ has size less than $\theta$.
As $I$ is $\theta^+$-complete, for every $\eta \in A$, $E_\eta:=\kappa\setminus\bigcup_{\tau\in \theta\setminus T_\eta}B^{\eta,\tau}$ is in $I^*$.
As  $I$ is $|A|^+$-complete, $B':=B\cap\bigcap_{\eta\in A}E_\eta$ is in $I^+$.
In particular, $B'\in J^+$.
So, by the choice of $c$,
there is an $\eta\in A$ such that $c[\{\eta\}\circledast  B']$ has ordertype $\theta$. In particular, we may pick $\tau\in c[\{\eta\}\circledast  B']\setminus T_\eta$.
Fix $\beta \in B'$ above $\eta$ such that $c(\eta, \beta)=\tau$. Then $\beta\in B'\s E_\eta$.
On the other hand, as $\tau\in\theta\setminus T_\eta$, $E_\eta\cap B^{\eta,\tau}=\emptyset$.
This is a contradiction.

(2) Suppose that $c$ witnesses $\onto(\mathcal A,J, \theta)$ for a given cardinal $\theta\le\kappa$.
We claim that $c$ moreover witnesses $\onto^+(\mathcal A,\mathcal J, \theta)$.
Suppose not, and fix $A\in\mathcal A$, $I\in \mathcal J$ and $B\in I^+$ such that, for every $\eta\in A$,
there is $\tau_\eta<\theta$ such that $E_\eta:=\kappa\setminus B^{\eta,\tau_\eta}$ is in $I^*$.
As before $B':=B\cap\bigcap_{\eta\in A}E_\eta$ is in $I^+$,
and hence $B'\in J^+$.
By the choice of $c$,
there is an $\eta\in A$ such that $c[\{\eta\}\circledast  B']=\theta$. In particular, we may pick
$\beta \in B'$ above $\eta$ such that $c(\eta, \beta)=\tau_\eta$. But $\beta\in B'\s E_\eta$,
contradicting the fact that $E_\eta\cap B^{\eta,\tau_\eta}=\emptyset$.

(3) Suppose that $c$ witnesses  $\onto^{++}(\mathcal A,J, \theta)$ for a given cardinal $\theta\le\kappa$.
We claim that $c$ moreover witnesses $\onto^{++}(\mathcal A,\mathcal J, \theta)$.
Suppose not. Fix $A\in\mathcal A$, $I\in \mathcal J$
and a sequence $\langle B_\tau\mid \tau<\theta\rangle$ of $I^+$-sets such that, for every $\eta\in A$,
there is $\tau_\eta<\theta$ such that $E_\eta:=\kappa\setminus (B_{\tau_\eta})^{\eta,\tau_\eta}$ is in $I^*$.
As $I$ is $\nu^+$-complete, for every $\tau<\theta$, $B_\tau':=B_\tau\cap\bigcap_{\eta\in A}E_\eta$ is in $I^+$, and hence $J^+$ as well.
By the choice of $c$,
there is an $\eta\in A$ such that $(B'_\tau)^{\eta,\tau}$ is $J$-positive for all $\tau<\theta$.
In particular, we may pick $\beta\in (B_{\tau_\eta}')^{\eta,\tau_\eta}$.
Then $\beta\in E_\eta\cap (B_{\tau_\eta})^{\eta,\tau_\eta}$. This is a contradiction.
\end{proof}

In \cite[p.~12]{sierpinski1934hypothese}, Sierpi\'nski proved a theorem asserting that ``$\mathsf{H}\rightarrow\mathsf{P_3}$''.
A modern interpretation of this theorem reads as follows.

\begin{cor}[Sierpi\'nski]\label{ulamext} Assuming $\kappa=\nu^+=2^\nu$,
$\onto^+([\kappa]^\nu,\mathcal J^\kappa_{\kappa},\kappa)$ holds.
\end{cor}
\begin{proof} By \cite[Lemma~8.3(1)]{paper47}, $\kappa=\nu^+=2^\nu$ implies $\onto([\kappa]^\nu,\allowbreak J^{\bd}[\kappa],\kappa)$,
and then, by Proposition~\ref{prop614}(2), $\onto^+([\kappa]^\nu,\mathcal J^\kappa_\kappa,\kappa)$ follows.
\end{proof}
\begin{remark} A simple elaboration on the proof of \cite[Lemma~8.3(2)]{paper47} shows that $\stick(\kappa)$ implies $\onto^+([\kappa]^\kappa,\mathcal J^\kappa_{\kappa},\kappa)$ for every infinite successor cardinal $\kappa$.
\end{remark}

\section{Subnormal ideals revisited}\label{subnormalrevisited}

By \cite[Lemma~4.4]{paper47}, $\ubd(\ns_\kappa,\kappa)$ implies $\ubd^+(\ns_\kappa,\kappa)$,
and it remained open whether likewise $\ubd(J^{\bd}[\kappa],\kappa)$ implies $\ubd^+(J^{\bd}[\kappa],\kappa)$.
In fact, in the original version of \cite{paper47}, we claimed to have also proved the second implication,
but the referee generously provided a counterexample to that proof. Here we establish the implication we were after.

We discuss first the difference between the two statements. Evidently, for any upper-regressive colouring $c:[\kappa]^2\rightarrow\kappa$ and any fixed $\eta<\kappa$,
the induced one-dimensional map $c_\eta:\kappa\setminus(\eta+1)\rightarrow\kappa$ defined via $c_\eta(\beta):=c(\eta,\beta)$ is regressive.
So, to put our finger on the difference between the two tasks, note that in proving the first implication, given a stationary $B\s\kappa$,
if one can find an $\eta$ such that $c_\eta$ attains $\kappa$ many colours over $B$,
then Fodor's lemma will readily ensure that some of these colours will be repeated stationarily often;
on the other hand, in trying to prove the second implication, one may face a witness $c$ to $\ubd(J^{\bd}[\kappa],\kappa)$ and an unbounded  $B\s\kappa$
such that $c_\eta$ is \emph{injective} over $B$ for all $\eta<\kappa$.
It follows that to compensate for the lack of Fodor's feature, there is a need to give more sincere weight to the fact that the colourings under considerations are two-dimensional.
We do this by incorporating walks on ordinals.

In this section, we study the behavior of our colouring principles over the class $\mathcal S^\kappa_\nu$ of all $\nu$-complete subnormal ideals over $\kappa$ extending $J^{\bd}[\kappa]$.
\begin{defn}[{\cite[\S2]{paper47}}]\label{defsubnormal} An ideal $J$ over $\kappa$ is said to be \emph{subnormal} if for every sequence $\langle E_\eta\mid \eta<\kappa\rangle$ of sets from $J^*$,
the following two hold:
\begin{enumerate}
\item  for every $B\in J^+$, there exists $B'\s B$ in $J^+$ such that, for every $(\eta,\beta)\in [B']^2$,
$\beta\in E_\eta$;
\item  for all $A,B\in J^+$, there exist $A'\s A$ and $B'\s B$ with $A',B'\in J^+$ such that, for every $(\eta,\beta)\in A'\circledast B'$,
$\beta\in E_\eta$.
\end{enumerate}
\end{defn}
As examples, note that $J^\bd[\kappa]$ is always a subnormal ideal, and in case $\kappa$ is an uncountable regular cardinal, then every normal ideal is subnormal as well.

The main result of this section reads as follows:

\begin{cor}\label{cor8a} Suppose that $\kappa$ is a regular uncountable cardinal.
\begin{enumerate}[(1)]
\item  For $\theta\in[3,\kappa)$, $\ubd( J^{\bd}[\kappa],\theta)$ implies $\ubd^+(J^{\bd}[\kappa],\theta)$;
\item  For $\theta<\kappa$, $\ubd(\self,J^{\bd}[\kappa],\theta)$ implies $\ubd^+(\self, J^{\bd}[\kappa],\theta)$;
\item  For $\theta=\kappa$, $\ubd( J^{\bd}[\kappa],\theta)$ implies $\ubd^+(\self, J^{\bd}[\kappa],\theta)$.
\end{enumerate}
\end{cor}
\begin{proof}
(1) The case that $\theta$ is an infinite cardinal less than $\kappa$ is covered by \cite[Proposition~2.26(3)]{paper47}.
By Theorem~\ref{ek} below, $\onto^{++}(J^{\bd}[\kappa],\theta)$ holds for every finite $\theta$, provided that $\kappa$ is not strongly inaccessible.
So suppose that $\kappa$ is strongly inaccessible and $2<\theta<\omega$.
By \cite[Corollary~10.8]{paper47}, $\ubd( J^{\bd}[\kappa],\theta)$  implies that $\kappa$ is not weakly compact,
and then the fact that $\kappa$ is strongly inaccessible together with \cite[Theorem~10.2]{paper47}
imply that moreover $\onto^+(J^{\bd}[\kappa],\omega)$ holds.

(2) By Lemma~\ref{lemma43b} below.

(3) By Corollary~\ref{cor28} below.
\end{proof}

In this section, $\kappa$ denotes a regular uncountable cardinal.
The utility of subnormal ideals is demonstrated in the following series of lemmas.
Loosely speaking, the first lemma shows that an upper-regressive colouring that is $({\ge}2)$-to-$1$ over positive sets of an ideal $J\in\mathcal S^\kappa_\kappa$ will witness $\ubd^+$ over all $I\in\mathcal S^\kappa_\kappa$ extending $J$.

\begin{lemma}\label{lemma113} Suppose that $J\in\mathcal S^\kappa_\kappa$.

For $\mathcal A\s\mathcal P(\kappa)$ and a colouring $c:[\kappa]^2\rightarrow\theta$, the following are equivalent:
\begin{enumerate}[(1)]
\item $c$ witnesses $\ubd^+(\mathcal A,J,\theta)$;
\item for all $A\in\mathcal A$ and $B\in J^+$,
there is an $\eta\in A$ such that
$$\otp(\{\tau < \theta\mid |\{\beta \in B\setminus(\eta+1) \mid c(\eta, \beta) = \tau\}|\ge2\})=\theta;$$
\item $c$ witnesses $\ubd^+(\mathcal A,I,\theta)$ for every $I\in\mathcal S^\kappa_\kappa$ extending $J$.
\end{enumerate}
\end{lemma}
\begin{proof} Only the implication $(2)\implies(3)$ requires an argument,
so suppose that (2) holds.
For all $B\s\kappa$, $\eta<\nu$ and $\tau<\theta$, denote $B^{\eta,\tau}:=\{\beta \in B\setminus(\eta+1) \mid c(\eta, \beta) = \tau\} $.
Towards a contradiction, suppose that $I$ is a $\kappa$-complete subnormal ideal over $\kappa$ extending $J$
and yet for some $A\in\mathcal A$, $B \in I^+$, for every $\eta\in A$,
$T_\eta:=\{ \tau<\theta\mid B^{\eta,\tau}\in I^+\}$ has order-type less than $\theta$.
As $I$ is $\kappa$-complete, for every $\alpha<\kappa$, $E_\alpha:=\kappa\setminus\bigcup_{\eta\in A\cap\alpha}\bigcup_{\tau\in \theta\cap \alpha\setminus T_\eta}B^{\eta,\tau}$ is in $I^*$.
As $I$ is subnormal, we may fix $B' \s B$ in $I^+$ such that, for every $(\alpha, \beta)\in [B']^2$,
$\beta\in E_\alpha$.
As $B'\in I^+$, in particular, $B'\in J^+$. Now pick some $\eta\in A$ such that the following set has order-type $\theta$:
$$T:=\{\tau < \theta\mid |\{\beta \in B'\setminus(\eta+1) \mid c(\eta, \beta) = \tau\}|\ge2\}.$$
Fix $\tau \in T\setminus T_\eta$, and find $(\alpha,\beta)\in[B'\setminus(\eta+1)]^2$ such that $c(\eta,\alpha)=\tau=c(\eta,\beta)$.
As $c$ is upper-regressive, $\tau=c(\eta, \alpha) < \alpha$.
Altogether, $\eta\in A\cap\alpha$ and $\tau\in\theta\cap\alpha\setminus T_\eta$.
As $\beta\in E_\alpha$, it follows that $\beta\in \kappa\setminus B^{\eta,\tau}$,
contradicting the fact that $c(\eta, \beta) = \tau$.
\end{proof}

For a cardinal $\theta<\kappa$, we get a free upgrade from $\ubd$ to $\ubd^+$, while settling for $\theta^+$-completeness, as follows.

\begin{lemma}\label{lemma114} Suppose that $J\in\mathcal S^\kappa_{\theta^+}$ with $\theta<\kappa$.

For a colouring $c:[\kappa]^2\rightarrow\theta$, the following are equivalent:
\begin{enumerate}[(1)]
\item $c$ witnesses  $\ubd(J^+,J,\theta)$;
\item $c$ witnesses $\ubd^{+}(I^+,I, \theta)$ for every $I\in\mathcal S^\kappa_{\theta^+}$ extending $J$.
\end{enumerate}
\end{lemma}
\begin{proof} Only the forward implication requires an argument.
For any $B\s\kappa$ and for ordinals $\eta<\kappa$ and $\tau<\theta$,
denote $B^{\eta,\tau}:=\{\beta \in B\setminus(\eta+1) \mid c(\eta, \beta) = \tau\} $.

Suppose that $c:[\kappa]^2\rightarrow\theta$ witnesses $\ubd(J^+,\allowbreak J,\theta)$.
Fix $I\in\mathcal S^\kappa_{\theta^+}$ extending $J$,
and we shall show that $c$ witnesses $\ubd^+(I^+, I, \theta)$.
To this end, fix $A, B \in I^+$.
Towards a contradiction suppose that for every $\eta \in A$ there is some $T_\eta \in [\theta]^{<\theta}$ such that,
for every $\tau \in \theta\setminus T_\eta$, $B^{\eta, \tau} \in I$.
As $I$ is $\theta^+$-complete, for every $\eta \in A$, $E_\eta:= \kappa \setminus (\bigcup_{\tau \in \theta \setminus T_\eta} B^{\eta, \tau})$ is in $I^*$.
As $I$ is subnormal, let us fix $A' \s A$ and $B'\s B$, both in $I^+$, such that $\beta\in E_\eta$ for all $(\eta,\beta)\in A' \circledast B'$.
As $A',B'$ are both in particular in $J^+$, let us now fix $\eta\in A'$ such that $\otp(c[\{\eta\}\circledast B'])=\theta$.
In particular, we may pick $\beta\in B'$ above $\eta$ such that $c(\eta, \beta) \notin T_\eta$. So, $\beta\notin E_\eta$, contradicting the fact that $(\eta,\beta) \in A' \circledast B'$.
\end{proof}

The proofs of Lemmas \ref{lemma113} and \ref{lemma114} make it clear that the following analogous results hold, as well.
\begin{lemma}\label{lemma43} For $J\in\mathcal S^\kappa_\kappa$,
a colouring $c:[\kappa]^2\rightarrow\theta$
witnesses $\ubd^+(\self,J,\theta)$
iff it witnesses $\ubd^+(\self,I,\theta)$ for every $I\in\mathcal S^\kappa_\kappa$ extending $J$.\qed
\end{lemma}

\begin{lemma}\label{lemma43b} Suppose that $J\in\mathcal S^\kappa_{\theta^+}$ with $\theta<\kappa$.

For a colouring $c:[\kappa]^2\rightarrow\theta$, the following are equivalent:
\begin{enumerate}[(1)]
\item $c$ witnesses  $\ubd(\self,J,\theta)$;
\item $c$ witnesses $\ubd^{+}(\self,I, \theta)$ for every $I\in\mathcal S^\kappa_{\theta^+}$ extending $J$.\qed
\end{enumerate}
\end{lemma}

\begin{lemma} Suppose that $\theta$ is infinite.
Consider $\mathcal J=\{ I\in\mathcal S^\kappa_\kappa\mid I\supseteq J\}$
for a given $J\in\mathcal S^\kappa_\kappa$.
Then:
\begin{enumerate}[(1)]
\item $\onto(J,\theta)$ implies $\onto^+(\mathcal J,\theta)$;
\item $\onto([\kappa]^\kappa,J,\theta)$ implies $\onto^+([\kappa]^\kappa,\mathcal J,\theta)$.
\end{enumerate}
\end{lemma}
\begin{proof} Let $c:[\kappa]^2\rightarrow\theta$ be any colouring. Fix a $2$-to-$1$ map $\pi:\theta\rightarrow\theta$.
Set $d:=\pi\circ c$.
For all $B\s\kappa$, $\eta<\kappa$, and $\tau<\theta$, denote $B^{\eta,\tau}:=\{\beta \in B\setminus(\eta+1) \mid d(\eta, \beta) = \tau\} $.

(1) 		Suppose that $c$ witnesses $\onto(J, \theta)$.
Let $I\in \mathcal J$.
Towards a contradiction, suppose that $B\in I^+$ is such that for every $\eta<\kappa$, there is a $\tau_\eta<\theta$
such that $E_\eta:=\kappa\setminus B^{\eta,\tau_\eta}$ is in $I^*$.
As $I$ is subnormal and $\kappa$-complete, we can find $B'\s B$ in $I^+$ such that, for all $(\alpha,\beta)\in[B']^2$,
$\beta\in\bigcap_{\eta<\alpha} E_\eta$.
As $B'$ is in particular in $J^+$, we may fix an $\eta<\kappa$ such that $c[\{\eta\}\circledast  B']=\theta$.
We may pick $(\alpha,\beta)\in[B']^2$ above $\eta$ such that $d(\eta, \alpha)=\tau_\eta=d(\eta,\beta)$. As $\beta \in B'$ and $\beta > \alpha>\eta$, we have that $\beta \in E_\eta$,
contradicting the fact that $E_\eta\cap B^{\eta,\tau_\eta}=\emptyset$.

(2) The proof is similar.
\end{proof}
\begin{remark} In the special case of $\mathcal A=\self$, in the above proof, the $\kappa$-completeness of $I$ won't play any role. Namely,
for an infinite $\theta$ and $J\in\mathcal S^\kappa_\omega$,
$\onto(\self,J,\theta)$ implies $\onto^+(\self,\mathcal J,\theta)$
for $\mathcal J:=\{ I\in\mathcal S^\kappa_\omega\mid I\supseteq J\}$.
\end{remark}

\begin{defn}[{\cite[\S3]{paper47}}] Let $S\s\kappa$. A $C$-sequence $\vec C=\langle C_\beta\mid\beta\in S\rangle$
is \emph{strongly amenable in $\kappa$} if for every club $D$ in $\kappa$,
the set $\{ \beta\in S \mid D\cap\beta\subseteq C_\beta\}$ is bounded in $\kappa$.
Let $\sa_\kappa:=\{S\s\kappa\mid S\text{ carries a }C\text{-sequence strongly amenable in }\kappa\}$.
\end{defn}

Recall that for a set of ordinals $S$, $J^{\bd}[S]$ stands for the ideal of bounded subsets of $S$.
By \cite[Lemma~3.4]{paper47}, $\sa_\kappa\cap[\kappa]^\kappa=\{ S\in[\kappa]^\kappa\mid \ubd(J^{\bd}[S],\kappa)\text{ holds}\}$.
In this section, we shall be interested in the stronger principle $\ubd^+$.
To appreciate the difference, note that  even if $\ubd(J^{\bd}[\kappa],\kappa)$ fails,
for the set $S$ of all successor ordinals below $\kappa$,
it is easy to construct a colouring $c$ witnessing $\ubd([\kappa]^\kappa,J^{\bd}[S],\kappa)$
that does not satisfy the proposition of Lemma~\ref{lemma113}(2) with $\theta:=\kappa$,
whereas, $\ubd^+$ enjoys the following equivalency:

\begin{prop} $\ubd^+(J^{\bd}[S],\theta)$ holds for some $S\in[\kappa]^\kappa$
iff it holds for all $S\in[\kappa]^\kappa$.
\end{prop}
\begin{proof} Given $c:[\kappa]^2\rightarrow\theta$ witnessing $\ubd^+(J^{\bd}[S],\theta)$ with $S\in[\kappa]^\kappa$,
let $\pi:\kappa\rightarrow S$ denote the inverse collapse.
Then pick an upper-regressive colouring $d:[\kappa]^2\rightarrow\theta$ such that $d(\eta,\beta):=c(\eta,\pi(\beta))$, provided that the latter is less than $\beta$.
It is not hard to verify that $d$ witnesses $\ubd^+(J^{\bd}[\kappa],\theta)$.
\end{proof}

\begin{defn}[{\cite[\S4]{paper35}}]\label{defnchi} \hfill
\begin{itemize}
\item Given a $C$-sequence $\vec{C} = \langle C_\beta \mid \beta \in S \rangle$ over a subset $S$ of $\kappa$,
$\chi(\vec{C})$ stands for the least cardinal $\chi \leq \kappa$ such that there exist
$\Delta \in [\kappa]^\kappa$ and $b:\kappa\rightarrow[S]^\chi$ with
$\Delta\cap\alpha\s\bigcup_{\beta\in b(\alpha)}C_\beta$
for every $\alpha<\kappa$.
\item $\cspec(\kappa):= \{\chi(\vec C) \mid \vec C$ is a $C$-sequence over $\kappa\}\setminus \omega$.
\item If $\kappa$ is weakly compact, then $\chi(\kappa):=0$.  Otherwise, $\chi(\kappa):=\max(\{1\}\cup\cspec(\kappa))$.
\end{itemize}
\end{defn}
\begin{remark}
The third bullet is actually a \emph{claim} established as \cite[Theorem~4.7]{paper35}, where the original definition of $\chi(\kappa)$ is slightly different \cite[Definition~1.6]{paper35}.
By \cite[Lemma~2.12]{paper35}, if $\chi(\kappa)\le 1$, then
$\kappa$ is a greatly Mahlo cardinal that is weakly compact in $L$,
and for every sequence $\langle S_i\mid i<\kappa\rangle$ of stationary subsets of $\kappa$,
there exists an inaccessible $\alpha<\kappa$ such that $S_i\cap\alpha$ is stationary in $\alpha$ for every $i<\alpha$.
\end{remark}

\begin{lemma}\label{strongstrongamen} Suppose that $\vec C=\langle C_\gamma \mid \gamma \in S\rangle$ is a $C$-sequence over a subset $S\s\kappa$.
Then the following are equivalent:
\begin{enumerate}[(1)]
\item $\vec C$ is strongly amenable in $\kappa$;
\item $\chi(\vec C)>1$;
\item For every club $D \s \kappa$ there are club many $\delta< \kappa$ such that $\sup((D \cap \delta)\setminus C_\gamma) = \delta$ for every $\gamma \in S$.
\end{enumerate}
\end{lemma}
\begin{proof} $(1)\implies(2)$: Suppose that $\vec C$ is strongly amenable in $\kappa$,
and yet, $\chi(\vec C)=1$.
Fix a set $\Delta\in[\kappa]^\kappa$ and a map $b:\kappa\rightarrow S$ such that $\Delta\cap\alpha\s C_{b(\alpha)}$ for every $\alpha<\kappa$.
For each $\epsilon<\kappa$, let $\beta_\epsilon:=b(\epsilon+1)$.
Now, consider the club $D:=\{\delta\in \acc^+(\Delta)\mid \forall\epsilon<\delta~(\beta_\epsilon<\delta)\}$.
The next claim yields the desired contradiction.
\begin{claim} For every $\epsilon\in D$, $D\cap\beta_\epsilon\s C_{\beta_\epsilon}$.

So, $\{\beta\in S\mid D\cap\beta\s C_\beta\}$ covers $\{ \beta_\epsilon\mid \epsilon\in D\}$,
which is a cofinal subset of $\kappa$.
\end{claim}
\begin{why} Let $\epsilon\in D$. As $\Delta\cap(\epsilon+1)\s C_{\beta_\epsilon}$ and $C_{\beta_\epsilon}$ is closed,
$\acc^+(\Delta)\cap(\epsilon+1)\s C_{\beta_\epsilon}$. In particular, $D\cap(\epsilon+1)\s C_{\beta_\epsilon}$.
Thus, it suffices to prove that $D\cap[\epsilon+1,\beta_\epsilon)=\emptyset$.
Let $\delta$ be any element of $D$ above $\epsilon$,
then, by the definition of $D$, $\delta>\beta_\epsilon$, as sought.
\end{why}

$(2)\implies(3)$: Suppose that $\chi(\vec C)>1$, and yet,
we are given a club $D \s \kappa$ and a stationary set $T\s\kappa$ such that, for every $\delta\in T$ there are an $\epsilon_\delta < \delta$ and a $\gamma_\delta \in S$ such that $\sup((D \cap \delta)\setminus C_{\gamma_\delta}) = \epsilon_\delta$.
As the map $\delta \mapsto \epsilon_\delta$ is regressive, for some stationary subset $T' \s  T$ and $\epsilon<\kappa$ we have that $\delta \in T'$ implies $\epsilon_\delta = \epsilon$.
Set $\Delta:= D\setminus(\epsilon+1)$ and define a function $b:\kappa\rightarrow S$ via $b(\alpha):=\gamma_{\min(T'\setminus\alpha)}$. Then, for every $\alpha<\kappa$, $\Delta\cap\alpha\s C_{b(\alpha)}$,
contradicting the fact that $\chi(\vec C)>1$.

\medskip

$(3)\implies(1)$:  This is immediate.
\end{proof}

Recall that $\kappa\nrightarrow[\kappa]^2_\theta$ asserts the existence of a colouring $c:[\kappa]^2\rightarrow\theta$ such that, for every $B\in[\kappa]^\kappa$, $c``[B]^2=\theta$.

\begin{lemma}\label{claim411} Suppose that $c$ witnesses $\kappa\nrightarrow[\kappa]^2_\theta$.
Let $B\in[\kappa]^\kappa$. Then there exists $\epsilon<\kappa$ such that, for all $\beta\in B\setminus\epsilon$ and $\tau<\min\{\epsilon,\theta\}$,
there exists $\eta\in B\cap\epsilon$ such that $c(\eta,\beta)=\tau$.
\end{lemma}
\begin{proof} Suppose not.
For each $\epsilon<\kappa$, pick $\beta_\epsilon\in B\setminus\epsilon$ and $\tau_\epsilon<\min\{\epsilon,\theta\}$ such that,
for no $\eta\in B\cap\epsilon$, $c(\eta,\beta_\epsilon)=\tau_\epsilon$.
Fix $\tau$ for which $E:=\{\epsilon<\kappa\mid \tau_\epsilon=\tau\}$ is stationary in $\kappa$.
Then, fix a sparse enough cofinal subset $B'$ of $B$ such that for every $(\eta,\beta)\in[B']^2$ there exists $\epsilon\in E$ such that $\eta<\epsilon\le\beta=\beta_\epsilon$.
Finally, as $c$ witnesses $\kappa\nrightarrow[\kappa]^2_\theta$, we may find $(\eta,\beta)\in[B']^2$ such that $c(\eta,\beta)=\tau$.
Fix $\epsilon\in E$ such that $\eta<\epsilon\le\beta=\beta_\epsilon$. Then $c(\eta,\beta_\epsilon)=\tau_\epsilon$,
contradicting the fact that $\eta\in B\cap\epsilon$.
\end{proof}

\begin{theorem}\label{strongamenfirstrepair}
Suppose that $\kappa \in \sa_\kappa$. Then $\ubd^+(\self,J^\bd[\kappa],\kappa)$ holds.
\end{theorem}
\begin{proof}
Let $\vec C= \langle C_\gamma \mid \gamma< \kappa\rangle$ witness that $\kappa \in \sa_\kappa$.
We shall conduct walks on ordinals along $\vec C$.
By Clause~(2) of Lemma~\ref{strongstrongamen} together with \cite[Theorem~8.1.11]{TodWalks},
we may also fix a colouring $o:[\kappa]^2\rightarrow\omega$ witnessing $\kappa\nrightarrow[\kappa]^2_\omega$.

Now, define an upper-regressive colouring $c:[\kappa]^2\rightarrow\kappa$, as follows. Given $\eta<\beta<\kappa$,
let $c(\eta,\beta):=\Tr(\eta,\beta)(o(\eta,\beta)+1)$ which is a well-defined ordinal in the interval $[\eta,\beta)$.

Towards a contradiction, suppose that $c$ fails to witness $\ubd^+(\self,J^\bd[\kappa],\kappa)$.
In this case, we may pick a set $B\in (J^\bd[\kappa])^+$ such that for every $\eta\in B$ there is a $\sigma_\eta < \kappa$ such that for every $\tau \in \kappa \setminus \sigma_\eta$, 
the set $\{\beta \in B \setminus (\eta+1) \mid c(\eta, \beta) = \tau\}$ is in $J^\bd[\kappa]$.
Let $D$ be the collection of all $\delta\in\acc(\kappa)$ such that all of the following hold:
\begin{enumerate}
\item for every $\eta < \delta$, $\sigma_\eta < \delta$;
\item for every $\eta < \delta$ and $\tau \in \delta \setminus\sigma_\eta$, $\sup\{\beta \in B \mid c(\eta, \beta) = \tau\} < \delta$;
\item  for all $\beta\in B\setminus\delta$ and $n< \omega$, $\sup\{\eta \in B \cap \delta \mid o(\eta,\beta) = n\}= \delta$.
\end{enumerate}
\begin{claim}\label{strongamenfirstrepairclaim}
$D$ is a club in $\kappa$.
\end{claim}
\begin{why} We can restrict our attention to Clause (iii) as the set of $\delta\in \acc(\kappa)$ which satisfy Clauses (i) and (ii) is clearly a club.

For every $i<\kappa$, let $\epsilon_i$ be given by Lemma~\ref{claim411} when fed with the set $B\setminus i$, using $\theta:=\omega$.
Clearly, any $\delta<\kappa$ above $\omega$ which forms a closure point of the map $i\mapsto\epsilon_i$
satisfies the requirement of Clause~(iii), and the set of closure points of this map is a club in $\kappa$.
\end{why}

By Lemma~\ref{strongstrongamen}, we may now let $\delta\in\acc(D)$ be such that $\sup((D \cap \delta)\setminus C_\gamma)= \delta$ for every $\gamma< \kappa$.
Fix $\beta \in B$ above $\delta$ and then let $\gamma:=\last{\delta}{\beta}$ in the sense of \cite[Definition~2.10]{paper44},
so that $\delta\le \gamma\le\beta$ and $\sup(C_\gamma\cap\delta)=\delta$, the latter holding since $\delta$ is a limit ordinal.
By \cite[Lemma~2.11]{paper44},   $\Lambda:=\lambda(\gamma,\beta)$ is less than $\delta$.
Set $n:=\rho_2(\gamma,\beta)$.
Then, for every ordinal $\eta$ with $\Lambda<\eta<\gamma$, $\Tr(\eta,\beta)(n)=\gamma$.
By the choice of $\delta$, we may now fix an $\alpha \in (D\cap\delta)\setminus C_\gamma$ above $\Lambda$.
As $\alpha \in D\setminus(\Lambda+1)$, the set $\hat A:= \{\eta\in B\cap(\Lambda, \alpha) \mid o(\eta,\beta)= n\}$ is cofinal in $\alpha$.

Let $\eta \in \hat A$. As $\Lambda< \eta< \alpha<\delta\le \gamma\le\beta$, it is the case that $\Tr(\eta,\beta)(n)=\gamma$ and hence $c(\eta, \beta)= \Tr(\eta,\beta)(n+1)=\min (C_\gamma \setminus \eta)$.
In particular, since $\sup(C_\gamma\cap\delta)=\delta$, for every $\eta \in \hat A$, $c(\eta, \beta) < \delta$.
So, as $\delta\in D$ and $\beta>\delta$, it follows that for every $\eta \in \hat A$, $\eta\leq c(\eta, \beta) \le \sigma_\eta$. Since $\alpha\in D$, for any $\eta\in \hat A$, $\sigma_\eta< \alpha$.
Altogether, $\{ c(\eta,\beta)\mid \eta\in \hat A\}$ is a cofinal subset of $\alpha$, consisting of elements of the set $C_\gamma$ which is closed below $\delta$, contradicting the fact that $\alpha\in\delta\setminus C_\gamma$.
\end{proof}

Recall that $\kappa\nrightarrow[\kappa; \kappa]^2_\theta$ asserts the existence of a colouring $c:[\kappa]^2\rightarrow\theta$ such that, for all $A,B\in[\kappa]^\kappa$, $c[A \circledast B]=\theta$.

\begin{lemma}\label{claim422}  Suppose that $c$ witnesses $\kappa\nrightarrow[\kappa;\kappa]^2_\theta$.
Let $A\in[\kappa]^\kappa$. Then there exists $\epsilon<\kappa$ such that, for all $\beta\in \kappa\setminus\epsilon$ and $\tau<\min\{\epsilon,\theta\}$,
there exists $\eta\in A\cap\epsilon$ such that $c(\eta,\beta)=\tau$.
\end{lemma}
\begin{proof}  Suppose not.
For each $\epsilon<\kappa$, pick $\beta_\epsilon\in \kappa\setminus\epsilon$ and $\tau_\epsilon<\min\{\epsilon,\theta\}$ such that,
for no $\eta\in A\cap\epsilon$, $c(\eta,\beta_\epsilon)=\tau_\epsilon$.
Fix $\tau<\theta$ for which $E:=\{ \epsilon<\kappa\mid \tau_\epsilon=\tau\}$ is stationary in $\kappa$.
Define three strictly increasing maps $f,g,h:\kappa\rightarrow \kappa$ as follows.
Let:
\begin{itemize}
\item $f(0):=\min(A)$;
\item $g(0):=\min(E\setminus(f(0)+1))$;
\item $h(0):=\beta_{g(0)}$.
\end{itemize}
Now, for every $i<\kappa$ such that $f\restriction i,g\restriction i,h\restriction i$ have already been defined,
let:
\begin{itemize}
\item $f(i):=\min(A\setminus\ssup(\im(h\restriction i)))$;
\item $g(i):=\min(E\setminus(f(i)+1))$;
\item $h(i):=\beta_{g(i)}$.
\end{itemize}

Note that for every $i<j<\kappa$, $f(i)<g(i)\le h(i)<f(j)$.
Set $A':=\im(f)$ and $B':=\im(h)$. By the choice of $c$, we may now pick $(\eta,\beta)\in A'\circledast B'$ such that $c(\eta,\beta)=\tau$.
Pick $i,j$ such that $\eta=f(i)$ and $\beta=h(j)$.
As $\eta<\beta$, it must be the case that $i<j$.
Set $\epsilon:=g(j)$. Then $\eta=f(i)<g(i)\le g(j)=\epsilon\le\beta_\epsilon=h(j)=\beta$.
So, $c(\eta,\beta_\epsilon)=\tau_\epsilon$ contradicting the fact that $\eta\in A\cap\epsilon$.
\end{proof}

\begin{theorem}	\label{strongamensecondrepair}
Suppose that $\kappa \in \sa_\kappa$.

If $\kappa \nrightarrow [\kappa; \kappa]^2_\omega$ holds, then so does $\ubd^+([\kappa]^\kappa,J^\bd[\kappa],\kappa)$.
\end{theorem}
\begin{proof}  We define an upper-regressive colouring $c:[\kappa]^2\rightarrow\kappa$ as in the proof of Theorem~\ref{strongamenfirstrepair} except that this time we assume
that the auxiliary colouring $o:[\kappa]^2\rightarrow\omega$ moreover witnesses $\kappa\nrightarrow[\kappa;\kappa]^2_\omega$.

Towards a contradiction, suppose that $c$ fails to witness $\ubd^+([\kappa]^\kappa, J^\bd[\kappa],\allowbreak\kappa)$.
In this case, we may pick sets $A,B\in[\kappa]^\kappa$ such that for every $\eta\in A$ there is a $\sigma_\eta < \kappa$ such that for every $\tau \in \kappa \setminus \sigma_\eta$, 
the set $\{\beta \in B \setminus (\eta+1) \mid c(\eta, \beta) = \tau\}$ is in $J^\bd[\kappa]$.

Let $D$ be the collection of all $\delta \in \acc(\kappa)$ such that all of the following hold:
\begin{enumerate}
\item for every $\eta < \delta$, $\sigma_\eta < \delta$, and
\item for every $\eta < \delta$ and $\tau \in \delta \setminus\sigma_\eta$, $\sup\{\beta \in B \mid c(\eta, \beta) = \tau\} < \delta$, and
\item for every $\beta \in B \setminus \delta$ and every $n< \omega$, $\sup\{\eta \in A \cap \delta \mid o(\eta, \beta) = n\}= \delta$.
\end{enumerate}

A proof similar to that of Claim~\ref{strongamenfirstrepairclaim} except that we use Lemma~\ref{claim422} instead of Lemma~\ref{claim411},
establishes that $D$ is a club in $\kappa$.
Now, the rest of the proof is exactly the same as that of Theorem~\ref{strongamenfirstrepair} except that (following the notation of Theorem~\ref{strongamenfirstrepair}), 
we define $\hat A:= \{\eta\in (\Lambda, \alpha)\cap A \mid o(\eta, \beta)= n\}$.
\end{proof}

The following yields the equivalency $(1)\iff(3)$ of Theorem~\ref{thmb}.
\begin{cor}\label{cor28} The following are equivalent:
\begin{enumerate}[(1)]
\item $\ubd(J^\bd[\kappa],\kappa)$ holds;
\item $\kappa\in\sa_\kappa$;
\item $\chi(\kappa)>1$;
\item $\ubd^+(\self,\mathcal S^\kappa_\kappa,\kappa)$ holds.
\end{enumerate}
\end{cor}
\begin{proof} $(1)\iff(2)$:  By \cite[Lemma~3.4]{paper47}.

$(2)\iff(3)$: By Lemma~\ref{strongstrongamen}.

$(2)\implies(4)$: Suppose that $\kappa\in\sa_\kappa$. Then, by Theorem~\ref{strongamenfirstrepair},
we may fix an upper-regressive colouring $c:[\kappa]^2\rightarrow\kappa$ witnessing $\ubd^+(\self,J^{\bd}[\kappa],\kappa)$.
Then, by Lemma~\ref{lemma43}, $c$ witnesses  $\ubd^+(\self,\allowbreak\mathcal S^\kappa_\kappa,\kappa)$.

$(4)\implies(1)$: This is clear.
\end{proof}

\begin{thm}\label{thm814}
Suppose that $\kappa=\kappa^{\aleph_0}$. For every colouring $c:[\kappa]^2\rightarrow2$,
there exists a corresponding colouring $d:[\kappa]^2\rightarrow\omega$ satisfying the following.
For every $J\in\mathcal S^\kappa_{\omega_1}$
such that $c[A\circledast B]=2$ for all $A,B\in J^+$, $d$ witnesses $\onto^{++}(J,\omega)$.
\end{thm}
\begin{proof}
As $\kappa^{\aleph_0}=\kappa$, we may fix an enumeration $\langle x_\eta \mid \eta< \kappa\rangle$ of all the elements in ${}^\omega \kappa$.
Now given a colouring $c:[\kappa]^2\rightarrow2$,
derive a corresponding colouring $d: [\kappa]^2 \rightarrow \omega$ by letting $d(\eta, \beta)$ be the least $n$ such that $c(x_\eta(n), \beta) =1$ if such an $n$ exists, and if not, $d(\eta, \beta) := 0$.

To see that $d$ is as sought,
suppose that we are given a countably-complete subnormal ideal $J$ over $\kappa$ extending $J^{\bd}[\kappa]$ such that $c[A\circledast B]=2$ for all $A,B\in J^+$.
\begin{claim} \label{wcontoplusclaim}
Suppose that $\langle B_m \mid m< \omega\rangle$ is a sequence of sets in $J^+$.
Then there exists an $\eta < \kappa$ such that, for all $m< \omega$ and  $i < 2$,
the set	$\{\beta \in B_m \mid c(\eta, \beta) = i\}$
is in $J^+$.
\end{claim}
\begin{why}
Suppose not. Then for all $\eta < \kappa$ there are $m_\eta < \omega$, $i_\eta< 2$ and $E_\eta \in J^*$ such that $i_\eta \notin c[\{\eta\} \circledast (B_{m_\eta} \cap E_\eta)]$.
As $J$ is $\omega_1$-complete, we can then find $m^* < \omega$ and $i^* < 2$ for which  $A:=\{\eta<\kappa\mid m_\eta = m^*\text{ and }i_\eta = i^*\}$ is in $J^+$.
Using the subnormality of $J$, we can then find two subsets $A' \s A$ and $B' \s B_{m^*}$ in $J^+$ such that for every $(\eta, \beta) \in A' \circledast B'$, $\beta \in E_{\eta}$. In particular, $i^* \notin c[A' \circledast B']$. This contradicts the hypothesis on $J^+$.
\end{why}
Now suppose that we are given a sequence $\vec B = \langle B_m \mid m< \omega\rangle$ of sets in $J^+$.
For $n<\omega$, $x:n\rightarrow\kappa$, $y:n\rightarrow 2$ and $m< \omega$ denote:
$$\vec  B_{x,y}[m]:=\{\beta\in B_m\mid \forall i<n( c(x(i),\beta)=y(i))\},$$
and also
$$\tree(\vec B):=\{ (x, y) \in{}^n\kappa \times {}^n2\mid n<\omega \ \&\ \forall m < \omega\,[\vec B_{x,y}[m]\in J^+]\}.$$
Then it is clear that $\tree(\vec B)$ is a subset of ${}^{<\omega}\kappa\times {}^{<\omega}2$ consisting of pairs of tuples of the same length and closed under initial segments.
That is, it is a subtree of $\bigcup_{n<\omega}{}^{n}\kappa\times {}^{n}2$.
Since $\vec B$ consists of elements of $J^+$ which extend $J^\bd[\kappa]$, it is clear that $(\emptyset,\emptyset)\in \tree(\vec B)$,
so in particular $\tree(\vec B)$ is nonempty.

\begin{claim}\label{wcontoclaim}
Let $n<\omega$ and let $x\in{}^n\kappa$ and $y \in{}^n2$ be such that $(x,y) \in \tree(\vec B)$.
Then there exists $\eta<\kappa$ such that, for all $i<2$, $(x{}^\smallfrown\langle\eta\rangle, y{}^\smallfrown\langle i\rangle)\in \tree(\vec B)$.
\end{claim}
\begin{why} Since $(x,y) \in \tree(\vec B)$, for each $m< \omega$, the set $\vec B_{x,y}[m]$ is in $J^+$.
So we can apply Claim~\ref{wcontoplusclaim} to the sequence $\langle\vec B_{x,y}[m] \mid m< \omega\rangle$ to obtain an $\eta < \kappa$ such that for all $m< \omega$ and all $i< 2$,
the set $\{\beta \in \vec B_{x,y}[m] \mid c(\eta, \beta) = i\}$ is in $J^+$.
It follows that $(x{}^\smallfrown\langle\eta\rangle, y{}^\smallfrown\langle i\rangle)\in \tree(\vec B)$ for all $i< 2$.
\end{why}
Using this claim, we can recursively construct an  $x\in{}^\omega\kappa$ such that for every $n<\omega$,
the constant function $y:n\rightarrow\{0\}$ satisfies
$(x\restriction (n+1),y{}^\smallfrown\langle1\rangle)\in \tree(\vec B)$.
Note that since $\tree(\vec B)$ is closed under initial segments this implies that for every $n<\omega$ and the constant function $y:n\rightarrow\{0\}$, it is also the case that $(x\restriction n,y)\in \tree(\vec B)$.
Pick $\eta < \kappa$ such that $x = x_\eta$.
Then, for every $n< \omega$ the set $\vec B_{{x_\eta \restriction(n+1), y_n{}^\smallfrown\langle1\rangle}}[n]$ is in $J^+$.
For every $\beta$ in this set, $c(x_\eta(n) , \beta) =1$ and for every $i< n$ it is the case that $c(x_\eta(i) , \beta) =0$.
So for every $\beta \in \vec B_{{x_\eta \restriction (n+1), y_n{}^\smallfrown\langle1\rangle}}[n]$ above $\eta$, $d(\eta,\beta) = n$. As $B_{{x_\eta \restriction (n+1), y_n{}^\smallfrown\langle1\rangle}}[n]\s B_n$ we are done.
\end{proof}

\begin{cor} Suppose that $\kappa=\kappa^{\aleph_0}$.

For every $J\in\mathcal S^\kappa_{\omega_1}$,
if $\onto(J^+,J,2)$ holds, then so does $\onto^{++}(J,\omega)$.\qed
\end{cor}

\section{Pumping-up results}\label{pumings}
In this section we establish various implications between our colouring principles. An important role is played by the \emph{projections} that will be introduced in Definition~\ref{projdefn}.
In Theorem~\ref{uparrowlemma} and Theorem~\ref{thm42} we obtain colourings that have stronger partition properties than we have heretofore seen. The motivation for partitioning positive sets of $\kappa$-complete ideals $J$ that weakly project to $J^{\bd}[\kappa]$ in the sense of the upcoming theorems comes from \cite[\S 4.1]{paper46}.

Clause (3) of the following theorem (using $\nu=\kappa=\theta$) yields the equivalency $(1)\iff(2)$ of Theorem~\ref{thmb}.
\begin{thm}\label{uparrowlemma} Suppose that $\kappa$ is regular, $\nu\le\kappa$, and $\mathcal A\s [\kappa]^{\le\nu}$.
\begin{enumerate}[(1)]
\item Suppose that $\onto(\mathcal A,J^{\bd}[\kappa],\theta)$ holds, with $\theta$ infinite.
Set $\sigma:=\min\{\kappa,\nu^+\}$.
Then $\onto^+(\mathcal A,\mathcal J^\kappa_\sigma,\theta)$ holds.
Furthermore, there exists a colouring $c:[\kappa]^2\rightarrow\theta$
such that for every $\sigma$-complete ideal $J$ and every map $\psi:\bigcup J\rightarrow\kappa$ satisfying $\sup(\psi[B])=\kappa$ for all $B\in J^+$,
the following holds.
For all $A\in\mathcal A$ and $B\in J^+$, there exists an $\eta\in A$ such that
$$\{ \tau<\theta\mid \{ \beta\in B\mid \eta<\psi(\beta)\ \&\ c(\eta,\psi(\beta))=\tau\}\in J^+\}=\theta.$$
\item Suppose that $\ubd^+(\mathcal A,J^{\bd}[\kappa],\theta)$ holds. Set $\sigma:=\min\{\kappa,\max\{\nu,\theta\}^+\}$.
Then $\ubd^+(\mathcal A,\mathcal J^\kappa_\sigma,\theta)$ holds.
Furthermore, there exists a colouring $c:[\kappa]^2\rightarrow\theta$
such that for every $\sigma$-complete ideal $J$ and every map $\psi:\bigcup J\rightarrow\kappa$ satisfying $\sup(\psi[B])=\kappa$ for all $B\in J^+$,
the following holds.
For all $A\in\mathcal A$ and $B\in J^+$, there exists an $\eta\in A$ such that
$$\otp(\{ \tau<\theta\mid \{ \beta\in B\mid \eta<\psi(\beta)\ \&\ c(\eta,\psi(\beta))=\tau\}\in J^+\})=\theta.$$
\item Suppose that $\ubd(\{\nu\},J^{\bd}[\kappa],\theta)$ holds, with $\theta>2$. Set $\sigma:=\min\{\kappa,\allowbreak\max\{\nu,\theta\}^+\}$.
Then $\ubd^+(\{\nu\},\mathcal J^\kappa_\sigma,\theta)$ holds.
Furthermore, there exists a colouring $c:[\kappa]^2\rightarrow\theta$
such that for every $\sigma$-complete ideal $J$ and every map $\psi:\bigcup J\rightarrow\kappa$ satisfying $\sup(\psi[B])=\kappa$ for all $B\in J^+$,
the following holds.
For every $B\in J^+$, there exists an $\eta<\nu$ such that
$$\otp(\{ \tau<\theta\mid \{ \beta\in B\mid \eta<\psi(\beta)\ \&\ c(\eta,\psi(\beta))=\tau\}\in J^+\})=\theta.$$
\end{enumerate}
\end{thm}
\begin{proof} (1) Let $d:[\kappa]^2\rightarrow\theta$ be a colouring witnessing $\onto(\mathcal A,J^{\bd}[\kappa],\theta)$.
As $\theta$ is infinite, we may fix a $2$-to-$1$ map $\pi:\theta\rightarrow\theta$.
We claim that $c:=\pi\circ d$ is as sought.
To this end, let $J$ be a $\sigma$-complete ideal admitting a map $\psi:\bigcup J\rightarrow\kappa$ such that $\sup(\psi[B])=\kappa$ for every $B\in J^+$.
Denote $E:=\bigcup J$.
For all $B\s E$, $\eta<\kappa$, and $\tau<\theta$, denote
$$B^{\eta,\tau, \psi}:=\{\beta \in B\mid \eta<\psi(\beta)\ \&\ c(\eta, \psi(\beta)) = \tau\}.$$
Towards a contradiction, suppose that we are given $A\in\mathcal A$ and $B\in J^+$ such that for every $\eta\in A$,
for some $\tau_\eta<\theta$, $E_\eta:=E\setminus B^{\eta,\tau_\eta,\psi}$ is in $J^*$.

$\br$ Suppose first that $\nu<\kappa$.
As $J$ is $\nu^+$-complete, $B':=\psi[B\cap\bigcap_{\eta\in A} E_\eta]$ is cofinal in $\kappa$.
So by the choice of $d$, we may fix an $\eta\in A$ such that $d[\{\eta\}\circledast  B']=\theta$.
In particular, $c[\{\eta\}\circledast  B']=\theta$,
and we may fix a $\beta'\in B'$ above $\eta$ such that $c(\eta,\beta')=\tau_\eta$.
Pick $\beta\in B\cap E_\eta$ such that $\beta'=\psi(\beta)$.
Then $\eta<\beta'=\psi(\beta)$ and $c(\eta,\psi(\beta))=c(\eta,\beta')=\tau_\eta$,
so $\beta$ belongs to $B^{\eta,\tau_\eta,\psi}\cap E_\eta$. This is a contradiction.

$\br$ Suppose that $\nu=\kappa$.
As $J$ is $\kappa$-complete, for every $\epsilon<\kappa$, $\psi[B\cap\bigcap_{\eta\in A\cap\epsilon} E_\eta]$ is cofinal in $\kappa$.
Define a strictly increasing function $f:\kappa\rightarrow\kappa$ by recursion, as follows.
For every $j<\kappa$ such that $f\restriction j$ has already been defined,
let $\epsilon_j:=\sup(\im(f\restriction j))+1$, and then let
$f(j):=\min(\psi[B\cap\bigcap_{\eta\in A\cap \epsilon_j} E_{\eta}]\setminus \epsilon_{j})$.
Since $f$ is strictly increasing, $B':=\im(f)$ is cofinal in $\kappa$,
so by the choice of $d$, we may fix an $\eta\in A$ such that $d[\{\eta\}\circledast  B']=\theta$.
By the definition of $c$, we may now pick $(\alpha',\beta')\in[B']^2$ above $\eta$ such that $c(\eta,\alpha')=\tau_\eta=c(\eta,\beta')$.
Let $(i,j)\in[\kappa]^2$ be the unique pair to satisfy $\alpha'=f(i)$ and $\beta'=f(j)$.
Then $\eta\in A\cap\alpha'$ and $\alpha'=f(i)<\epsilon_j\le f(j)=\beta'$,
so that $\beta'\in\psi[B\cap E_\eta]$.
Pick $\beta\in B\cap E_\eta$ such that $\beta'=\psi(\beta)$.
Then $\eta<\beta'=\psi(\beta)$ and $c(\eta,\psi(\beta))=c(\eta,\beta')=\tau_\eta$,
so $\beta$ belongs to $B^{\eta,\tau_\eta,\psi}\cap E_\eta$. This is a contradiction.

(2) Fix any colouring $c$ witnessing $\ubd^+(\mathcal A,J^{\bd}[\kappa],\theta)$.
Let $J$ be a $\sigma$-complete ideal admitting a map $\psi:\bigcup J\rightarrow\kappa$ such that $\sup(\psi[B])=\kappa$ for every $B\in J^+$.
Denote $E:=\bigcup J$.
We shall use the notation $B^{\eta,\tau, \psi}$ coined earlier.
Towards a contradiction, suppose that we are given $A\in\mathcal A$ and $B \in J^+$ such that, for every $\eta\in A$,
$T_\eta:=\{ \tau<\theta\mid B^{\eta,\tau,\psi}\in J^+\}$ has order-type less than $\theta$.

$\br$ Suppose first that $\max\{\nu,\theta\}<\sigma$.
As $J$ is $\sigma$-complete,
it follows that the following set is in $J^*$:
$$E':=E\setminus\bigcup_{\eta\in A}\bigcup_{\tau\in \theta\setminus T_\eta}B^{\eta,\tau,\psi}.$$
In particular, $B':=\psi[B\cap E']$ is cofinal in $\kappa$.
Now, as $c$ in particular witnesses $\ubd(\mathcal A,J^{\bd}[\kappa],\theta)$,
we may find an $\eta\in A$ such that $T:=c[\{\eta\}\circledast B']$ has order-type $\theta$.
Fix $\tau \in T\setminus T_\eta$, and then find $\beta'\in B'$ above $\eta$ such that $c(\eta,\beta')=\tau$.
Pick $\beta\in B\cap E'$ such that $\beta'=\psi(\beta)$.
Then $\eta<\beta'=\psi(\beta)$ and $c(\eta,\psi(\beta))=c(\eta,\beta')=\tau$,
so $\beta$ belongs to $B^{\eta,\tau,\psi}\cap E'$, contradicting the fact that $\tau\in\theta\setminus T_\eta$.

$\br$ Suppose that $\max\{\nu,\theta\}=\sigma$. In particular, $\sigma=\kappa$.
As $J$ is $\kappa$-complete, it follows that for every $\epsilon<\kappa$, the following set is in $J^*$:
$$E_\epsilon:=E\setminus\bigcup_{\eta\in A\cap\epsilon}\bigcup_{\tau\in \theta\cap (\epsilon\setminus T_\eta)}B^{\eta,\tau,\psi}.$$
In particular, for every $\epsilon<\kappa$, $\psi[B\cap E_\epsilon]$ is cofinal in $\kappa$.
Recursively, define a function $f:\kappa\rightarrow\kappa$ as follows.
For every $j<\kappa$ such that $f\restriction j$ has already been defined,
let $\epsilon_j:=\sup(\im(f\restriction j))+1$, and then let
$f(j):=\min(\psi[B\cap E_{\epsilon_j}]\setminus \epsilon_{j})$.
Since $B':=\im(f)$ is cofinal in $\kappa$,
the choice of $c$ yields an $\eta\in A$ such that the following set has order-type $\theta$:
$$T:=\{\tau < \theta\mid \sup\{\beta' \in B'\setminus(\eta+1) \mid c(\eta, \beta') = \tau\}=\kappa\}.$$
Fix $\tau \in T\setminus T_\eta$, and find $(\alpha',\beta')\in[B'\setminus(\eta+1)]^2$ such that $c(\eta,\alpha')=\tau=c(\eta,\beta')$.
As $c$ is upper-regressive, $\tau=c(\eta, \alpha') < \alpha'$.
Altogether, $\eta\in A\cap\alpha'$ and $\tau\in\theta\cap(\alpha'\setminus T_\eta)$.

Let $(i,j)\in[\kappa]^2$ be the unique pair to satisfy $\alpha'=f(i)$ and $\beta'=f(j)$.
Then $\alpha'<\epsilon_j\le \beta'$, so that
$\beta'\in\psi[B\cap E_{\epsilon_j}]\s\psi[B\setminus B^{\eta,\tau,\psi}]$.
Pick $\beta\in B\setminus B^{\eta,\tau,\psi}$ such that $\psi(\beta)=\beta'$.
Then $c(\eta,\psi(\beta))=c(\eta,\beta')=\tau$,
contradicting the fact that $\beta\notin B^{\eta,\tau,\psi}$.

(3) If $\max\{\nu,\theta\}<\sigma$, then the conclusion follows from the above proof of Clause~(2).
Thus, suppose that $\max\{\nu,\theta\}=\sigma=\kappa$. There are three options here:

$\br$ If $\nu=\kappa$, then this follows from Clause~(2), using Corollary~\ref{cor8a}(1) if $\theta< \kappa$, and Corollary~\ref{cor8a}(3) if $\theta = \kappa$.

$\br$ If $\theta\le\nu<\kappa$, then this follows from Clause~(2), using Proposition~\ref{prop614}(1).

$\br$ If $\nu<\theta=\kappa$, then by Proposition~\ref{prop26}(2), $(\nu,\kappa,\theta)=(\nu,\nu^+,\nu^+)$.
By \cite[Fact~5.1]{paper47}, that is, Ulam's matrices, $\ubd^+(\{\nu\},J^{\bd}[\nu^+],\nu^+)$ provably holds.
Now, appeal to Clause~(2).

This completes the proof.
\end{proof}
\begin{remark} In \cite[\S3]{kun78}, Kunen constructed a model with a $\kappa$-saturated ideal $J\in\mathcal J^\kappa_\kappa$
at some strongly inaccessible cardinal $\kappa$ which is not weakly compact. Footnote~7 from \cite[p.~180]{TodWalks} points out an observation by Donder and K{\"o}nig that $\chi(\kappa)\le1$ (or that ``there is no nontrivial $C$-sequence on $\kappa$'' in the language of \cite[Definition~6.3.1]{TodWalks}) in this model. Our results show that this is nothing specific to the model constructed by Kunen with its nice additional features. Indeed, by Clause~(3) of the preceding, if $\kappa$ carries a $\kappa$-saturated ideal $J\in\mathcal J^\kappa_\kappa$, then $\ubd^+(J^{\bd}[\kappa],\kappa)$ fails, and hence $\chi(\kappa)\le1$ by Corollary~\ref{cor28}.
\end{remark}

\begin{thm}\label{thm42} Suppose that $\kappa$ is regular.
\begin{enumerate}[(1)]
\item Suppose that $\onto(\self,J^{\bd}[\kappa],\theta)$ holds, with $\theta$ infinite.
Then $\onto^+(\self,\mathcal J^\kappa_\kappa,\theta)$ holds.
Furthermore, there exists a colouring $c:[\kappa]^2\rightarrow\theta$
such that for every $\kappa$-complete ideal $J$ and every map $\psi:\bigcup J\rightarrow\kappa$ satisfying $\sup(\psi[B])=\kappa$ for all $B\in J^+$,
the following holds.
For every $B\in J^+$, there exists an $\eta\in B$ such that
$$\{ \tau<\theta\mid \{ \beta\in B\mid \psi(\eta)<\psi(\beta)\ \&\ c(\psi(\eta),\psi(\beta))=\tau\}\in J^+\}=\theta.$$
\item Suppose that $\ubd(\self,J^{\bd}[\kappa],\theta)$ holds.
Then $\ubd^+(\self,\mathcal J^\kappa_\kappa,\theta)$ holds.
Furthermore, there exists a colouring $c:[\kappa]^2\rightarrow\theta$
such that for every $\kappa$-complete ideal $J$ and every map $\psi:\bigcup J\rightarrow\kappa$ satisfying $\sup(\psi[B])=\kappa$ for all $B\in J^+$,
the following holds.
For every $B\in J^+$, there exists an $\eta\in B$ such that
$$\otp(\{ \tau<\theta\mid \{ \beta\in B\mid \psi(\eta)<\psi(\beta)\ \&\ c(\psi(\eta),\psi(\beta))=\tau\}\in J^+\})=\theta.$$
\end{enumerate}
\end{thm}
\begin{proof} The proof is similar to that of Theorem~\ref{uparrowlemma} and is left to the reader.
\end{proof}

\begin{lemma}\label{lemma39} Let $\theta< \kappa$, $I\in\mathcal J^\kappa_{\theta^+}$ and $J\in\mathcal J^\kappa_\omega$.
\begin{enumerate}[(1)]
\item Every witness to $\onto^+(I^+,J,\theta)$ witnesses $\onto^{++}(I^+,J,\theta)$;
\item Every witness to $\ubd^+(I^+,J,\theta)$ witnesses $\ubd^{++}(I^+,J,\theta)$.
\end{enumerate}
\end{lemma}
\begin{proof}
Let $c:[\kappa]^2\rightarrow\theta$ be any colouring. For a subset $B\s\kappa$,
denote:
\begin{itemize}
\item $B^{\eta,\tau}:=\{\beta \in B\setminus(\eta+1) \mid c(\eta, \beta) = \tau\}$,
\item  $O(B):=\{\eta<\kappa\mid \forall\tau<\theta\,(B^{\eta,\tau}\in J^+)\}$, and
\item  $U(B):=\{\eta<\kappa\mid \otp(\{\tau<\theta\mid B^{\eta,\tau}\in J^+\})=\theta\}$.
\end{itemize}

(1) Suppose that $c$ witnesses $\onto^+(I^+,J, \theta)$.
Then, for all $A\in I^+$ and $B\in J^+$, $A\cap O(B)$ is nonempty.
It follows that for every $B\in J^+$, $O(B)\in I^*$.
Since $I$ is $\theta^+$-complete, given $A\in I^+$ and a sequence $\langle B_\tau \mid \tau< \theta\rangle$ of sets in $J^+$,
we may pick $\eta\in A\cap\bigcap_{\tau<\theta}O(B_\tau)$.
In particular, for every $\tau<\theta$, $B_\tau^{\eta,\tau}\in J^+$.

(2) Suppose that $c$ witnesses $\ubd^+(I^+,J, \theta)$.
Then, for every $B\in J^+$, $U(B)\in I^*$.
Since $I$ is $\theta^+$-complete, given $A\in I^+$ and a sequence $\langle B_\tau \mid \tau< \theta\rangle$ of sets in $J^+$,
we may pick $\eta\in A\cap\bigcap_{\tau<\theta}U(B_\tau)$.
Now, it is easy to find an injection $h:\theta\rightarrow\theta$ such that, for every $\tau<\theta$,
$B_\tau^{\eta,h(\tau)}\in J^+$.
\end{proof}

\begin{cor}\label{tripleupgrade} For $\kappa$ a regular uncountable cardinal and $\theta<\kappa$,
$\ubd([\kappa]^\kappa,\allowbreak J^{\bd}[\kappa],\theta)$ implies $\ubd^{++}([\kappa]^\kappa,\mathcal J^\kappa_\kappa,\theta)$.
\end{cor}
\begin{proof} By Lemma~\ref{lemma114}, $\ubd([\kappa]^\kappa,J^{\bd}[\kappa],\theta)$ implies $\ubd^+([\kappa]^\kappa,J^{\bd}[\kappa],\theta)$.
By Theorem~\ref{uparrowlemma}(2), the latter implies $\ubd^{+}([\kappa]^\kappa,\mathcal J^\kappa_{\kappa},\theta)$ holds.
Now, appeal to Lemma~\ref{lemma39}(2).
\end{proof}
\begin{lemma}\label{increasecolours}  Suppose that $\ubd^+(\mathcal J,\theta^+)$ holds
for a nonempty $\mathcal J\s\mathcal J^\kappa_{\omega}$.

Each of the following  hypotheses imply that $\onto^+(\mathcal J,\theta^+)$ holds:
\begin{enumerate}[(1)]
\item $\onto^+(\{\theta\},\mathcal J,\theta)$ holds;
\item $\onto(\{\theta\},\mathcal J,\theta)$ 	holds and $\mathcal J\s \mathcal J^\kappa_{\theta^+}$;
\item $\onto^+(I^+,\mathcal J,\theta)$ 	holds for some $I\in\mathcal J^\kappa_{\theta^{++}}$;
\item $\kappa\nrightarrow[\kappa;\kappa]^2_\theta$ holds, $\theta^+<\kappa$ and $\mathcal J\s\mathcal J^\kappa_\kappa$.
\end{enumerate}
\end{lemma}
\begin{proof} Fix a colouring $d:[\kappa]^2\rightarrow\theta^+$ witnessing $\ubd^+(\mathcal J,\theta^+)$.
For all $B\s\kappa$ and $\eta,\alpha<\kappa$, denote $B^{\eta,\alpha}:=\{\beta\in B\setminus(\eta+1)\mid d(\eta,\beta)=\alpha\}$.
For every $\alpha<\theta^+$, fix a surjection $e_\alpha:\theta\rightarrow\alpha+1$.
Also, fix a bijection $\pi:\kappa\leftrightarrow\kappa\times\kappa$.

(1) Suppose that $c:[\kappa]^2\rightarrow\theta$ is a colouring witnessing $\onto^+(\{\theta\},\mathcal J,\theta)$.
Define $f:[\kappa]^2\rightarrow\theta^+$ as follows.
Given $\eta<\beta<\kappa$, let $(\eta_0,\eta_1):=\pi(\eta)$ and then set $f(\eta,\beta):=e_{d(\{\eta_1,\beta\})}(c(\{\eta_0,\beta\}))$.

To see this works, let $J\in\mathcal J$ and $B\in J^+$.
By the choice of $d$, we may fix an $\eta_1<\kappa$ such that $A_{\eta_1}(B):=\{\alpha<\theta^+\mid B^{\eta_1,\alpha}\in J^+\}$ has size $\theta^+$.
For each $\alpha\in A_{\eta_1}(B)$, as $B^{\eta_1,\alpha}$ is in $J^+$,
we may find $\eta^\alpha<\theta$ such that, for every $i<\theta$,
$$\{ \beta\in B^{\eta_1,\alpha}\setminus(\eta^\alpha+1)\mid c(\eta^\alpha,\beta)=i\}\in J^+.$$
Pick $\eta_0<\theta$ for which $A:=\{\alpha\in A_{\eta_1}(B)\mid \eta^\alpha=\eta_0\}$ has size $\theta^+$.
Fix $\eta<\kappa$ such that $\pi(\eta)=(\eta_0,\eta_1)$.
Given any colour $\tau<\theta^+$, pick $\alpha\in A\setminus\tau$.
Then pick $i<\theta$ such that $e_\alpha(i)=\tau$.
Then pick $\beta\in B^{\eta_1,\alpha}$ above $\max\{\eta_0,\eta_1,\eta\}$ such that $c(\eta_0,\beta)=i$.
Then $f(\eta,\beta)=e_{d(\eta_1,\beta)}(c(\eta_0,\beta))=e_\alpha(i)=\tau$, as sought.

(2) By Clause~(1) together with Proposition~\ref{prop614}(2).

(3) Suppose that $I\in\mathcal J^\kappa_{\theta^{++}}$, and we are given a colouring $c:[\kappa]^2\rightarrow\theta$
such that, for all $J\in\mathcal J$ and $B\in J^+$, the set $\{\eta<\kappa\mid \forall i<\theta\,\{ \beta\in B\setminus(\eta+1)\mid c(\eta,\beta)=i\}\in J^+\}$ is in $I^*$.
Define $f:[\kappa]^2\rightarrow\theta^+$ as follows.
Given $\eta<\beta<\kappa$, let $(\eta_0,\eta_1):=\pi(\eta)$ and then set $f(\eta,\beta):=e_{d(\{\eta_1,\beta\})}(c(\{\eta_0,\beta\}))$.

To see this works, let $J\in\mathcal J$ and $B\in J^+$.
By the choice of $d$, we may fix an $\eta_1<\kappa$ such that $A_{\eta_1}(B):=\{\alpha<\theta^+\mid B^{\eta_1,\alpha}\in J^+\}$ has size $\theta^+$.
For each $\alpha\in A_{\eta_1}(B)$, $B^{\eta_1,\alpha}$ is in $J^+$. So, since $I$ is $\theta^{++}$-complete,
we may find $\eta_0<\kappa$ such that, for every $\alpha\in A_{\eta_1}(B)$,
for every $i<\theta$,
$$\{ \beta\in B^{\eta_1,\alpha}\setminus(\eta_0+1)\mid c(\eta_0,\beta)=i\}\in J^+.$$
Fix $\eta<\kappa$ such that $\pi(\eta)=(\eta_0,\eta_1)$.
Given any colour $\tau<\theta^+$, pick $\alpha\in A_{\eta_1}(B)\setminus\tau$.
Then pick $i<\theta$ such that $e_\alpha(i)=\tau$.
Then pick $\beta\in B^{\eta_1,\alpha}$ above $\max\{\eta_0,\eta_1,\eta\}$ such that $c(\eta_0,\beta)=i$.
Then $f(\eta,\beta)=e_{d(\eta_1,\beta)}(c(\eta_0,\beta))=e_\alpha(i)=\tau$, as sought.

(4) By Clause~(3) together with Theorem~\ref{prop46}(2) below.
\end{proof}

We now present the definition of projections.
Unlike the colouring principles of the previous section that asserts something about sets of the form $\{\eta\}\circledast B$,
projections have to do with sets of the form $\{\eta\}\times B$.

\begin{defn}\label{projdefn} $\proj(\nu,\kappa,\theta,{<}\mu)$ asserts the existence of a map $p:\nu\times\kappa\rightarrow\theta$
with the property that for every family $\mathcal B\s[\kappa]^\kappa$ of size less than $\mu$, there exists an $\eta<\nu$ such that
$p[\{\eta\}\times B]=\theta$ for all $B\in\mathcal B$.

We write $\proj(\nu,\kappa,\theta,\mu)$ for $\proj(\nu,\kappa,\theta,{<}\mu^+)$.
\end{defn}
\begin{remark}\label{rmk313}
Note that $\onto(\{\nu\},[\kappa]^{<\kappa},\theta)$ implies $\proj(\nu,\kappa,\theta,1)$,
and that $\onto^{++}(\{\nu\},\allowbreak[\kappa]^{<\kappa},\theta)$ implies $\proj(\nu,\kappa,\theta,\theta)$.
\end{remark}
We encourage the reader to determine the monotonicity properties of the above principle. Among them, let us mention the following whose proof is obvious and which we will use later.
\begin{prop}\label{tukey}
$\proj(\nu,\cf(\kappa),\theta,{<}\mu)$ implies $\proj(\nu,\kappa,\theta,{<}\mu)$. \qed
\end{prop}
The next proposition demonstrates that for every triple of cardinals $\theta\le\mu\le\kappa$,
there exists a cardinal $\nu$ such that $\proj(\nu,\kappa,\theta,\mu)$ holds.
By the preceding remark, more optimal values for $\nu$ may be calculated using the results of Section~\ref{adsection}.

\begin{prop}\label{bruteforce} Assuming $\theta\le\mu\le\kappa$, $\proj(\kappa^\mu,\kappa,\theta,\mu)$ holds.
\end{prop}
\begin{proof} Denote $\nu:=\kappa^\mu$.
Let $\langle f_\eta\mid\eta<\nu\rangle$ enumerate all injections from $\mu$ to $\kappa$.
Fix a surjection $s:\mu\rightarrow\theta$ such that the preimage of any singleton has size $\mu$.
Define a function $p:\nu\times\kappa\rightarrow\theta$ as follows.
For every $(\eta,\beta)\in\nu\times\kappa$, if there exists $\xi<\mu$ such that $f_\eta(\xi)=\beta$,
then let $p(\eta,\beta):=s(\xi)$ for this unique $\xi$. Otherwise, let $p(\eta,\beta):=0$.
To see this works, let $\langle B_i\mid i<\mu\rangle$ be  a given list of elements of $[\kappa]^\kappa$.
Find an injection $g:\theta\times\mu\rightarrow\kappa$ such that $g(\tau,i)\in B_i$ for all $(\tau,i)\in\theta\times\mu$.
Then pick an injection $h:\im(g)\rightarrow\mu$ such that $s(h(g(\tau,i)))=\tau$ for all $(\tau,i)\in\theta\times\mu$,
and finally find $\eta<\nu$ such that $h\circ f_\eta$ is the identity map.
We claim that $\eta$ is as sought.

Let $i<\mu$. To see that $p[\{\eta\}\times B_i]=\theta$, let $\tau<\theta$.
By the definition of $g$, $\beta:=g(\tau,i)$ is in $B_i$.
Set $\xi:=h(\beta)$. By the definition of $h$, $s(\xi)=\tau$.
Since $h$ is injective, $\xi=h(\beta)$ and $h(f_\eta(\xi))=\xi$, it is the case that $f_\eta(\xi)=\beta$.
So $p(\eta,\beta)=s(\xi)=\tau$, as sought.
\end{proof}
Recall that $\mathcal C(\kappa,\theta)$ denotes the least size of a family $\mathcal X\s[\kappa]^\theta$
with the property that for every club $C$ in $\kappa$, there is $X\in\mathcal X$ with $X\s C$.

\begin{lemma}\label{bsicproj}  $\proj(\mathcal C(\cf(\kappa),\theta),\kappa,\theta,{<}\mu)$ holds in any of the following cases:
\begin{enumerate}[(1)]
\item $\theta<\cf(\kappa)=\mu$;
\item $\theta = \cf(\kappa)$ and $\mu=\cf(\kappa)^+$.
\end{enumerate}
\end{lemma}
\begin{proof}
By Proposition~\ref{tukey}, it suffices to prove the following two:
\begin{itemize}
\item $\proj(\mathcal C(\cf(\kappa),\theta),\cf(\kappa),\allowbreak\theta,{<}\cf(\kappa))$
holds whenever $\theta<\cf(\kappa)$, and
\item $\proj(\mathcal C(\cf(\kappa),\cf(\kappa)),\cf(\kappa),\cf(\kappa),\cf(\kappa))$ holds.
\end{itemize}
Thus, for notational simplicity we may assume that $\kappa=\cf(\kappa)$.
By Proposition~\ref{bruteforce}, we may also assume that $\theta$ is infinite.

The proofs of both cases are similar and there is some overlap in the case analysis so we prove them together.
Let $\theta\le\kappa$. Fix a surjection $\pi:\theta\rightarrow\theta$ such that,
for every $\tau<\theta$, $\{ \sigma<\theta\mid \pi(\sigma+1)=\tau\}$ is cofinal in $\theta$.
Denote $\nu:=\mathcal C(\kappa,\theta)$.
Fix a sequence $\langle X_\eta\mid \eta<\nu\rangle$ of subsets of $\kappa$,
each of order-type $\theta$, such that, for every club $C$ in $\kappa$,
for some $\eta<\nu$, $X_\eta\s C$.
Then, define a map $p:\nu\times\kappa\rightarrow\theta$ via $p(\eta,\beta):=\pi(\otp(X_\eta\cap\beta))$.
We will show that this colouring works for both cases. So, let $\mathcal B\s[\kappa]^\kappa$ be given of the appropriate size. Our analysis splits into three:

$\br$ If $\kappa=\theta=\omega$, then fix an injective enumeration $\langle B_i\mid i<\omega\rangle$ of $\mathcal B$.
Then, fix a strictly increasing map $f:\omega\rightarrow\omega$ such that for every $n<\omega$,
for every $i<n$ there exists $\beta_{i,n}\in B_i$ with $f(n)<\beta_{i,n}<f(n+1)$.
Trivially, $C:=\im(f)$ is a club in $\kappa$,
so we may fix an $\eta<\nu$ such that $X_\eta\s C$.
Let $i<\omega$, and we shall show that $p[\{\eta\}\times B_i]=\theta$.
To this  end, let $\tau<\theta$ be any prescribed colour.
Find $\sigma<\theta$ above $f(i)$ such that $\pi(\sigma+1)=\tau$.
Let $\xi$ be the unique element of $X_\eta$ to satisfy $\otp(X_\eta\cap\xi)=\sigma$.
Let $n$ be the unique integer such that $\xi=f(n)$. As $f(n)=\xi\ge\sigma>f(i)$, we infer that $n>i$.
So $\beta_{i,n}$ is an element of $B_i$ lying in between $\xi$ and $\min(X_\eta\setminus(\xi+1))$.
Consequently, $\otp(X_\eta\cap\beta_{i,n})=\otp(X_\eta\cap(\xi+1))=\sigma+1$,
and hence $p(\eta,\beta_{i,n})=\pi(\otp(X_\eta\cap\beta_{i,n}))=\pi(\sigma+1)=\tau$, as sought.

$\br$ If $\kappa=\theta>\omega$,
then since $|\mathcal B|\le\kappa$, we may fix a club $C$ in $\kappa$ such that, for every $B\in\mathcal B$,
$C\setminus \acc^+(B)$ is bounded in $\kappa$. Fix $\eta<\nu$ such that $X_\eta\s C$.
Let $B\in\mathcal B$, and we shall show that $p[\{\eta\}\times B]=\theta$.
To this end, let $\tau<\theta$ be any prescribed colour.
As $C\setminus \acc^+(B)$ is bounded in $\kappa$,
$\epsilon:=\ssup(X_\eta\setminus\acc^+(B))$ is less than $\kappa$.
Find $\sigma\in[\epsilon,\kappa)$ such that $\pi(\sigma+1)=\tau$.
Let $\xi$ be the unique element of $X_\eta$ to satisfy $\otp(X_\eta\cap\xi)=\sigma$.
Evidently, $\xi\ge\sigma\ge\epsilon$.
Let $\beta:=\min(B\setminus(\xi+1))$.
As $X_\eta\setminus\epsilon\s \acc^+(B)$, $\otp(X_\eta\cap\beta)=\otp(X_\eta\cap(\xi+1))=\sigma+1$,
and hence $p(\eta,\beta)=\pi(\otp(X_\eta\cap\beta))=\pi(\sigma+1)=\tau$.

$\br$ Otherwise, so that $\kappa>\theta\ge\omega$.
In this case $|\mathcal B|<\kappa$, so that $C:=\bigcap\{\acc^+(B)\mid B\in\mathcal B\}$ is a club in $\kappa$.
Fix $\eta<\nu$ such that $X_\eta\s C$.
By now it should be clear that $p[\{\eta\}\times B]=\theta$ for every $B\in\mathcal B$.
\end{proof}

\begin{cor}\label{regularpairs} Suppose that $\theta$ is an infinite cardinal less than $\kappa$.
\begin{enumerate}[(1)]
\item If $\theta\in\reg(\kappa)$, then $\proj(\kappa,\kappa,\theta,1)$ holds;
\item If $\theta\in\reg(\cf(\kappa))$, then $\proj(\kappa,\kappa,\theta,\theta)$ holds;
\item If $([\theta]^{\cf(\theta)},{\s})$ has cofinality or density $\theta^+$, then $\proj(\kappa,\kappa,\theta,\theta)$ holds.
\end{enumerate}
\end{cor}
\begin{proof} There are five cases to consider:
\begin{enumerate}
\item  If $\cf(\kappa)>\theta^+$ for $\theta$ regular,
then Shelah's club-guessing theorem implies that $\mathcal C(\cf(\kappa),\theta)=\cf(\kappa)$.
So, in this case, Case~(1) of Lemma~\ref{bsicproj} states that
$\proj(\cf(\kappa),\kappa,\theta,{<}\cf(\kappa))$ holds.
In particular, $\proj(\kappa,\kappa,\theta,\theta)$ holds.

\item  If $\cf(\kappa)=\theta^+$ for $\theta$ singular satisfying $\cf([\theta]^{\cf(\theta)},{\supseteq})=\theta^+$,
then by appealing to \cite[Lemma~3.1]{paper38} with $\nu:=\cf(\kappa)$,
we infer that $\mathcal C(\cf(\kappa),\theta)=\cf(\kappa)$.
The rest of the proof is now identical to that of Clause~(i).

\item  If $\cf(\kappa)=\theta^+$ for $\theta$ singular satisfying $\cf([\theta]^{\cf(\theta)},{\s})=\theta^+$,
then by a theorem of Todor\v{c}evi\'c (see \cite[Proposition~2.5]{paper_s02}), $\cf(\kappa)\nrightarrow[\cf(\kappa)]^2_{\theta}$ holds.
So, by the main result of \cite{paper13}, $\cf(\kappa)\nrightarrow[\cf(\kappa);\cf(\kappa)]^2_{\theta}$ holds.
Then, by \cite[Proposition~6.6]{paper47},
$\onto^{++}(J^{\bd}[\cf(\kappa)],\allowbreak\theta)$ holds,
so by Remark~\ref{rmk313},  $\proj(\cf(\kappa),\cf(\kappa),\theta,\theta)$ holds.

\item  If $\cf(\kappa)=\theta^+$ for $\theta$ regular,
then by \cite[Corollary~7.3]{paper47}, $\onto^{++}(J^{\bd}[\cf(\kappa)],\allowbreak\theta)$ holds,
so by Remark~\ref{rmk313},  $\proj(\cf(\kappa),\cf(\kappa),\theta,\theta)$ holds.

\item  If $\cf(\kappa)<\theta^+<\kappa$,
then set $\varkappa:=\theta^{++}$
and fix a $\varkappa$-bounded $C$-sequence $\langle C_\delta\mid \delta\in E^{\kappa}_{\varkappa}\rangle$.
Using Case~(i), fix a map $p:\varkappa\times\varkappa\rightarrow\theta$ witnessing $\proj(\varkappa,\varkappa,\theta,1)$.
Pick a map $q:\kappa\times\kappa\rightarrow\theta$ such that for all $\delta\in E^\kappa_{\varkappa}$, $i<\varkappa$ and $\beta<\delta$,
$q(\delta+i,\beta)=p(i,\otp(C_\delta\cap\beta))$. To see this works, let $B\in[\kappa]^\kappa$.
Find the least ordinal $\delta$ such that $\otp(B\cap\delta)=\varkappa$.
Then $\delta\in E^\kappa_\varkappa$ and $A:=\{ \otp(C_\delta\cap\beta)\mid \beta\in B\cap\delta\}$ is in $[\varkappa]^\varkappa$.
By the choice of $p$, fix $i<\varkappa$ such that $p[\{i\}\times A]=\theta$.
Then $q[\{\delta+i\}\times B]=\theta$.\qedhere
\end{enumerate}
\end{proof}

It is not hard to see that for every regular uncountable cardinal $\theta$,
$\mathcal C(\theta,\theta)=\mathfrak d_\theta$.
It thus follows from Case~(2) of Lemma~\ref{bsicproj} that $\proj(\mathfrak d_\theta,\theta,\theta,\theta)$ holds for every regular uncountable cardinal $\theta$.
The next result improves this, covering the case $\theta=\omega$ and showing that the $4^{\text{th}}$ parameter can consistently be bigger than $\theta$.

\begin{lemma}\label{d1}
$\proj(\mathfrak d_\theta,\theta,\theta,{<}\mathfrak b_\theta)$ holds for every infinite regular cardinal $\theta$.
\end{lemma}
\begin{proof} Suppose that $\theta$ is an infinite regular cardinal. We follow the proof of Lemma~\ref{bsicproj}.
Fix a surjection $\pi:\theta\rightarrow\theta$ such that,
for every $\tau<\theta$, $\{ \sigma<\theta\mid \pi(\sigma+1)=\tau\}$ is cofinal in $\theta$.
Denote $\nu:=\mathfrak d_\theta$.
Fix a sequence $\langle f_\eta\mid\eta<\nu\rangle$ that is cofinal in $({}^\theta\theta,{<^*})$.
For each $\eta<\nu$, set $X_\eta:=\im(g_\eta)$, where $g_\eta:\theta\rightarrow\theta$ is some strictly increasing function
satisfying $g_\eta(\sigma+1) > f_\eta(g_\eta(\sigma))$ for all $\sigma<\theta$.
Finally, define a map $p:\nu\times\kappa\rightarrow\theta$ via $p(\eta,\beta):=\pi(\otp(X_\eta\cap\beta))$.

To see this works, fix $\mathcal B\s [\theta]^\theta$ with $0<|\mathcal B|<\mathfrak b_\theta$.
For each $B\in\mathcal B$, define a function $f^B:\theta\rightarrow B$ via $$f^B(\xi):=\min(B\setminus(\xi+1)).$$
As $|\mathcal B|<\mathfrak b_\theta$, we may pick an $\eta<\nu$ such that, for every $B\in\mathcal B$,
$f^B<^*f_\eta$. We claim that $p[\{\eta\}\times B]=\theta$ for all $B\in\mathcal B$.
To this end, fix $B\in\mathcal B$ and a prescribed colour $\tau<\theta$.
Fix $\epsilon<\theta$ such that $f^B(\xi)<f_\eta(\xi)$ whenever $\epsilon<\xi<\theta$.
Find $\sigma<\theta$ above $\epsilon$ such that $\pi(\sigma+1)=\tau$.
Evidently, $$\epsilon<\sigma\le g_\eta(\sigma)<f^B(g_\eta(\sigma))<f_\eta(g_\eta(\sigma))<g_\eta(\sigma+1).$$
So $\beta:=f^B(g_\eta(\sigma))$ is an element of $B$ satisfying $g_\eta(\sigma)<\beta<g_\eta(\sigma+1)$,
and hence $p(\eta,\beta)=\pi(\otp(X_\eta\cap\beta))=\pi(\sigma+1)=\tau$, as sought.
\end{proof}
\begin{lemma}\label{lemma15}
Suppose that:
\begin{itemize}
\item $\theta\le\varkappa\le\kappa$;
\item $\mu\le\nu\le\kappa$;
\item $\mathcal J\s\mathcal J^\kappa_\omega$;
\item $\proj(\nu,\varkappa,\theta,\theta)$ holds.
\end{itemize}

\begin{enumerate}[(1)]
\item If $\ubd^{++}(\{\nu\},\mathcal J,\varkappa)$ holds, then so does $\onto^{++}(\{\nu\},\mathcal J,\theta)$;
\item If $\ubd^{+}(I,\mathcal J,\varkappa)$ holds with $I\in\mathcal J^\nu_{\theta^+}$, then so does $\onto^{++}(\{\nu\},\allowbreak\mathcal J,\theta)$;
\item If $\ubd^+(\{\mu\},\mathcal J,\varkappa)$ holds, $\mathcal J\s\mathcal J^\kappa_{\varkappa^+}$,
$\varkappa\in\reg(\kappa)$, $\cov(\mu,\varkappa,\theta^+,2)\le\nu$,
then $\onto^{++}(\{\nu\},\mathcal J,\theta)$ holds.
\end{enumerate}
\end{lemma}
\begin{proof} Let $p:\nu\times\varkappa\rightarrow\theta$ be a function witnessing $\proj(\nu,\varkappa,\theta,\theta)$.
Fix a bijection $\pi:\nu\leftrightarrow\nu\times\nu$.

(1) Let $c:[\kappa]^2\rightarrow\varkappa$ be a colouring witnessing $\ubd^{++}(\{\nu\},\mathcal J,\varkappa)$.
Then, pick any colouring $d:[\kappa]^2\rightarrow\theta$ such that
for all $\eta<\beta<\kappa$, if $\eta<\nu$ and $\pi(\eta)=(\eta',i)$, then $d(\eta,\beta)=p(i,c(\{\eta',\beta\}))$.

To see that $d$ is as sought, let $\langle B_\tau\mid\tau<\theta\rangle$ be a sequence of sets in $J^+$, for a given $J\in\mathcal J$.
By the choice of $c$, we may pick an $\eta'<\nu$ such that, for every $\tau<\theta$, the following set has order-type $\varkappa$:
$$X_\tau:=\{\xi<\varkappa\mid \{ \beta\in B_\tau\setminus(\eta'+1)\mid c(\eta',\beta)=\xi\}\in J^+\}.$$

Let $i<\nu$ be such that $p[\{i\}\times X_\tau]=\theta$ for all $\tau<\theta$.
Let $\eta<\nu$ be such that $\pi(\eta)=(\eta',i)$.
\begin{claim}\label{claim3201} Let $\tau<\theta$.
Then $\{ \beta\in B_\tau\setminus(\eta+1)\mid d(\eta,\beta)=\tau\}$ is in $J^+$.
\end{claim}
\begin{why}
Fix $\xi\in X_\tau$ such that $p(i,\xi)=\tau$.
As $\xi\in X_\tau$,
the set	$B':=\{ \beta\in B_\tau\setminus(\eta'+1)\mid c(\eta',\beta)=\xi\}$ is in $J^+$.
As $J$ extends $J^{\bd}[\kappa]$, so is $B'\setminus(\max\{\eta',\eta\}+1)$.
Now, for every $\beta\in B'\setminus(\max\{\eta',\eta\}+1)$,
$$d(\eta,\beta)=p(i,c(\eta',\beta))=p(i,\xi)=\tau,$$
as sought.
\end{why}

(2) Let $c:[\kappa]^2\rightarrow\varkappa$ be a colouring witnessing $\ubd^{++}(I,\mathcal J,\varkappa)$ for a given $I\in\mathcal J^\nu_{\theta^+}$.
As $I$ is $\theta^+$-complete, by the same proof of Lemma~\ref{lemma39},
for every sequence $\langle B_\tau\mid\tau<\theta\rangle$ of $J$-positive sets for some $J\in\mathcal J$,
there exists an $\eta'<\nu$ such that, for every $\tau<\theta$, the following set has order-type $\varkappa$:
$$X_\tau:=\{\xi<\varkappa\mid \{ \beta\in B_\tau\setminus(\eta'+1)\mid c(\eta',\beta)=\xi\}\in J^+\}.$$
Thus, a proof nearly identical to that of Clause~(1) establishes that $\onto^{++}(\{\nu\},\allowbreak\mathcal J,\theta)$ holds.

(3) Let $c:[\kappa]^2\rightarrow\varkappa$ be a colouring witnessing $\ubd^{+}(\{\mu\},\mathcal J,\varkappa)$
for a collection $\mathcal J\s\mathcal J^\kappa_{\varkappa^+}$.
Assuming $\cov(\mu,\varkappa,\theta^+,2)\le\nu$,
we may fix a sequence $\langle x_\alpha\mid\alpha<\nu\rangle$ of elements $[\mu]^{<\varkappa}$
such that, for every $y\in[\mu]^\theta$, there exists $\alpha<\nu$ such that $y\s x_\alpha$.
Assuming that $\varkappa$ is regular, we may define a function $f:\nu\times\kappa\rightarrow\varkappa$ via:
$$f(\alpha,\beta):=\sup\{0,c(\eta,\beta)\mid \eta\in x_\alpha\cap\beta\}.$$
Then, pick any colouring $d:[\kappa]^2\rightarrow\theta$ such that for all $\eta<\beta<\kappa$,
if $\eta<\nu$ and $\pi(\eta)=(\alpha,i)$, then $d(\eta,\beta)=p(i,f(\alpha,\beta))$.

To see that $d$ is as sought, let $\langle B_\tau\mid\tau<\theta\rangle$ be a sequence of sets in $J^+$,
for a given $J\in\mathcal J$.
By the choice of $c$, for each $\tau<\theta$, we may fix an $\eta_\tau<\mu$ for which
$$\otp(\{\sigma<\varkappa\mid \{\beta\in B_\tau\setminus(\eta_\tau+1)\mid c(\eta_\tau,\beta)=\sigma\}\in J^+\})=\varkappa.$$
Find $\alpha<\nu$ such that $x_\alpha\supseteq\{ \eta_\tau\mid\tau<\theta\}$.
\begin{claim} Let $\tau<\theta$ and $\epsilon<\varkappa$. There exists $\xi\in[\epsilon,\varkappa)$ for which
$$\{\beta\in B_\tau\setminus(\alpha+1)\mid f(\alpha,\beta)=\xi\}\in J^+.$$
\end{claim}
\begin{why} By the choice of $\eta_\tau$ and as $J$ extends $J^{\bd}[\kappa]$,
we may fix $\sigma\in[\epsilon,\varkappa)$ for which
$B':=\{\beta\in B_\tau\setminus(\max\{\alpha,\eta_\tau\}+1)\mid c(\eta_\tau,\beta)=\sigma\}$ is in $J^+$.
Evidently, $\{ \beta\in B_\tau\setminus(\alpha+1)\mid f(\alpha,\beta)\ge\epsilon\}$ covers $B'$.
As $J$ is $\varkappa^+$-complete, there must exist some $\xi\in[\epsilon,\varkappa)$ as sought.
\end{why}

By the preceding claim, for each $\tau<\theta$, the following set is cofinal in $\varkappa$:
$$X_\tau:=\{\xi<\varkappa\mid \{\beta\in B_\tau\setminus(\alpha+1)\mid f(\alpha,\beta)=\xi\}\in J^+\}.$$
Fix $i<\nu$ such that $p[\{i\}\times X_\tau]=\theta$ for all $\tau<\theta$.
Let $\eta<\nu$ be such that $\pi(\eta)=(\alpha,i)$.
Then a verification similar to the proof of Claim~\ref{claim3201}
implies that, for every $\tau<\theta$,
$\{\beta\in B_\tau\setminus(\eta+1)\mid d(\eta,\beta)=\tau\}$ is in $J^+$.
\end{proof}

\begin{cor}\label{cor415} Suppose that $\ubd^+([\kappa]^\kappa,\mathcal J,\varkappa)$ holds for a given $\mathcal J\s \mathcal J^\kappa_\omega$ and $\varkappa\le \kappa$.
For every $\theta\in\reg(\cf(\varkappa))$, $\onto^{++}(\mathcal J,\theta)$ holds.
\end{cor}
\begin{proof}
Let $c$ be a colouring witnessing $\ubd^+([\kappa]^\kappa,\mathcal J,\varkappa)$.
Let $\theta\in\reg(\cf(\varkappa))$.
By Corollary~\ref{regularpairs}(2), $\proj(\varkappa,\varkappa,\theta,\theta)$ holds.
So, by Lemma~\ref{lemma15}(2), $\onto^{++}(\mathcal J,\allowbreak\theta)$ holds.
\end{proof}

\begin{cor}\label{lemma16}
Suppose that $\ubd^+(\{\mu\},\mathcal J,\theta)$ holds for a given collection $\mathcal J\s\mathcal J^\kappa_{\theta^+}$
and $\aleph_0\le\theta\le\mu\le\kappa$.
Then $\onto^{++}(\{\mu\},\mathcal J,n)$ holds  for every positive integer $n$.\qed
\end{cor}

A proof similar to that of Lemma~\ref{lemma15} establishes:
\begin{lemma} \label{ubdtoonto}
Suppose that $\theta\le\varkappa\le\kappa$ and $\mu\le\nu\le\kappa$ are cardinals, $\mathcal J\s \mathcal J^\kappa_\omega$,
and $\proj(\nu,\varkappa,\theta,1)$ holds.
\begin{enumerate}[(1)]
\item If $\ubd(\{\mu\},\mathcal J,\varkappa)$ holds and either $\mu<\kappa$ or $\mathcal J$ consists of subnormal ideals, then $\onto(\{\nu\},\mathcal J,\theta)$ holds;
\item If $\ubd^{+}(\{\nu\},\mathcal J,\varkappa)$ holds, then so does $\onto^{+}(\{\nu\},\mathcal J,\theta)$.\qed
\end{enumerate}
\end{lemma}

By putting together the tools established so far, we obtain the following nontrivial monotonicity result.

\begin{cor}[monotonicity]\label{cmonotone} Suppose that $\ubd^+(\mathcal J,\varkappa)$ holds for a given $\mathcal J\s \mathcal J^\kappa_\omega$ and $\varkappa\le \kappa$.
\begin{enumerate}[(1)]
\item For every $\theta<\varkappa$, $\ubd^+(\mathcal J,\theta)$ holds;
\item For every regular $\theta<\varkappa$, $\onto^+(\mathcal J,\theta)$ holds;
\item For every $\theta\le\varkappa$ such that $2^\theta\le\kappa$, $\onto^+(\mathcal J,\theta)$ holds;
\item For every regular $\theta<\varkappa$ such that $\cf([\kappa]^\theta,{\s})=\kappa$, $\onto^{++}(\mathcal J^\kappa_{\theta^{++}}\cap\mathcal J,\theta)$ holds.
\end{enumerate}
\end{cor}
\begin{proof}
(1) In case that $\theta^+=\varkappa$, $\ubd^+(\mathcal J,\theta)$ holds by the same proof of \cite[Lemma~6.2]{paper47}.
So, we are left with handling the case that $\theta$ is finite or $\theta^+<\varkappa$, both of which follow from the monotonicity of $\onto^+(\ldots)$ in the number of colours together with the next clause.

(2) Let $\theta\in\reg(\varkappa)$.
By Corollary~\ref{regularpairs}(1), $\proj(\varkappa,\varkappa,\theta,1)$ holds.
In particular, $\proj(\nu,\varkappa,\theta,1)$ holds for $\nu:=\kappa$.
So, by Lemma~\ref{ubdtoonto}(2), $\onto^+(\mathcal J,\theta)$ holds, as well.

(3) Given $\theta\le\varkappa$, by Clause~(1), $\ubd^+(\mathcal J,\theta)$ holds.
In addition, by Proposition~\ref{bruteforce}, $\proj(2^\theta,\theta,\theta,1)$ holds.
So, assuming that $2^\theta\le\kappa$, Proposition~\ref{ubdtoonto}(2) implies that $\onto^+(\mathcal J,\theta)$ holds, as well.

(4) Suppose that $\theta\in\reg(\varkappa)$ is such that $\cf([\kappa]^\theta,{\s})=\kappa$.
By Clause~(1), $\ubd^+(\mathcal J,\theta^+)$ holds.

By Corollary~\ref{regularpairs}(2),
$\proj(\theta^+,\theta^+,\theta,\theta)$ holds. In particular,
$\proj(\kappa,\theta^+,\allowbreak\theta,\theta)$ holds.
So by Lemma~\ref{lemma15}(3), using $(\varkappa,\mu,\nu):=(\theta^+,\kappa,\kappa)$, $\onto^{++}(\mathcal J^\kappa_{\theta^{++}}\cap\mathcal J,\theta)$ holds.
\end{proof}

\begin{cor}\label{cor52} Suppose that $\kappa$ is a regular uncountable cardinal.

Then $(1)\implies(2)\implies(3)\implies(4)$:
\begin{enumerate}[(1)]
\item There exists a stationary subset of $\kappa$ that does not reflect at regulars;
\item $\chi(\kappa)>1$ and $\kappa\nrightarrow[\kappa;\kappa]^2_\omega$ holds;
\item $\ubd^{+}([\kappa]^\kappa,\mathcal J^\kappa_\kappa,\kappa)$ holds;
\item $\onto^{++}(\mathcal J^\kappa_\kappa,\theta)$ holds for all $\theta\in\reg(\kappa)$.
\end{enumerate}
\end{cor}
\begin{proof}
$(1)\implies(2)$: If $\kappa=\mu^+$ is a successor cardinal,
then  $\chi(\kappa)\ge\cf(\mu)>1$ and by works of Moore and Shelah (see the introduction to \cite{paper14}) $\kappa\nrightarrow[\kappa;\kappa]^2_\omega$ holds.

If $\kappa$ is an inaccessible cardinal admitting a stationary subset of $\kappa$ that does not reflect at regulars,
then $\chi(\kappa)>1$, and by \cite[Lemmas 4.17 and 4.7]{Sh:365}, $\kappa\nrightarrow[\kappa;\kappa]^2_\omega$ holds.

$(2)\implies(3)$: By Lemma~\ref{strongamensecondrepair} and Theorem~\ref{uparrowlemma}(2).

$(3)\implies(4)$: By Corollary~\ref{cor415}.
\end{proof}

\section{Partition relations}\label{partitions}

In this section, we improve upon results from \cite[\S6]{paper47} that were limited to subnormal ideals.
The main two results of this section read as follows:\footnote{The definition of $\U(\ldots)$ may be found in Definition~\ref{defUp} below; $\cspec(\ldots)$ was defined in Definition~\ref{defnchi}.}
\begin{cor} For every pair $\theta<\kappa$ of infinite regular cardinals:
\begin{enumerate}[(1)]
\item $\square(\kappa)$ implies $\kappa\nrightarrow[\kappa;\kappa]^2_\theta$ which in turn implies $\onto^{++}([\kappa]^\kappa,\mathcal J^\kappa_\kappa, \theta)$;
\item $\U(\kappa,2,\theta,3)$ implies $\ubd^{++}([\kappa]^\kappa,\allowbreak \mathcal J^\kappa_\kappa, \theta)$.
\end{enumerate}
\end{cor}
\begin{proof} (1) The first part follows from the main result of \cite{paper18}.
The second part is Theorem~\ref{prop46}(2) below.

(2) By Lemma~\ref{lemma34} below.
\end{proof}

\begin{cor}\label{cor42} Suppose that $\kappa$ is strongly inaccessible.
Then $\square(\kappa,{<}\omega)$ implies
$\sup(\cspec(\kappa))=\kappa$ which in turn implies that
$\onto^{++}(\mathcal J^\kappa_\kappa,\theta)$ holds for all $\theta<\kappa$.
\end{cor}
\begin{proof} The first part follows from \cite[Corollary~5.26]{paper35}.
Next, assuming that $\sup(\cspec(\kappa))=\kappa$,
given a cardinal $\theta<\kappa$, fix $\varkappa\in\cspec(\kappa)$ above $\theta$.
By Proposition~\ref{csingular} below, $\ubd^{++}([\kappa]^\kappa,\mathcal J^\kappa_{\kappa},\varkappa)$ holds.
As $\varkappa^\theta<\kappa$, by Proposition~\ref{bruteforce}, $\proj(\kappa,\varkappa,\theta,\theta)$ holds.
So, by Lemma~\ref{lemma15}(1), $\onto^{++}(\mathcal J^\kappa_{\kappa},\theta)$ holds, as well.
\end{proof}

In this section, $\kappa$ denotes a regular uncountable cardinal.

\begin{thm}\label{prop46}  Let $\theta<\kappa$.

\begin{enumerate}[(1)]
\item $\kappa\nrightarrow[\kappa]^2_\theta$ iff $\onto(\self, J^\bd[\kappa], \theta)$ iff $\onto^+(\self,\mathcal J^\kappa_\kappa, \theta)$;
\item $\kappa\nrightarrow[\kappa;\kappa]^2_\theta$ iff $\onto([\kappa]^\kappa, J^\bd[\kappa], \theta)$ iff $\onto^{++}([\kappa]^\kappa,\mathcal J^\kappa_\kappa, \theta)$.
\end{enumerate}
\end{thm}
\begin{proof} (1)
By the same proof of \cite[Proposition~6.4]{paper47},
$\kappa\nrightarrow[\kappa]^2_\theta$ implies $\onto(\self,\allowbreak J^{\bd}[\kappa],\allowbreak\theta)$.
By Theorem~\ref{thm42}(1), $\onto(\self, J^\bd[\kappa], \theta)$ implies $\onto^+(\self,\mathcal J^\kappa_\kappa, \theta)$.
Trivially, $\onto^+(\self,\mathcal J^\kappa_\kappa, \theta)$ implies $\kappa\nrightarrow[\kappa]^2_\theta$.

(2) By \cite[Proposition~6.6]{paper47},
$\kappa\nrightarrow[\kappa;\kappa]^2_\theta$ implies $\onto([\kappa]^\kappa,J^{\bd}[\kappa], \theta)$.
By Theorem~\ref{uparrowlemma}(1), $\onto([\kappa]^\kappa,J^{\bd}[\kappa], \theta)$ implies $\onto^{+}([\kappa]^\kappa,\mathcal J^\kappa_\kappa, \theta)$.
By Lemma~\ref{lemma39}(1),
$\onto^+([\kappa]^\kappa, J^\bd[\kappa], \theta)$ implies $\onto^{++}([\kappa]^\kappa,\mathcal J^\kappa_\kappa, \theta)$.
Trivially, $\onto^{++}([\kappa]^\kappa,\mathcal J^\kappa_\kappa, \theta)$ implies $\kappa\nrightarrow[\kappa;\kappa]^2_\theta$.
\end{proof}

We now arrive at a proof of Theorem~\ref{thmc}.
\begin{cor}
Suppose that $\kappa$ is a regular uncountable cardinal admitting a stationary set that does not reflect at regulars.
Then, for every cardinal $\theta<\kappa$, $\onto^{++}(\mathcal J^\kappa_\kappa,\theta)$ holds.
\end{cor}
\begin{proof} By the implication $(1)\implies(4)$ of Corollary~\ref{cor52},
$\onto^{++}(\mathcal J^\kappa_\kappa,\theta)$ holds for all $\theta\in\reg(\kappa)$.
Thus, the only thing left is establishing that $\onto^{++}(\mathcal J^\kappa_\kappa,\theta)$ holds
in the case that $\kappa=\theta^+$ where $\theta$ is a singular cardinal.
By Theorem~\ref{prop46}(3), it suffices to show that $\theta^+\nrightarrow[\theta^+;\theta^+]^2_\theta$ holds assuming that $\theta^+$ admits a nonreflecting stationary set,
and this follows, e.g., from the main result of \cite{paper15}.
\end{proof}

\begin{cor}\label{prop41}\label{thm54}  Suppose that $\kappa=\theta^+$ for an infinite cardinal $\theta$.
\begin{enumerate}[(1)]
\item If $\theta$ is regular, then $\onto^{++}([\kappa]^\kappa,\mathcal J^\kappa_\kappa,\theta)$ holds;
\item If $\theta$ is singular, then $\ubd^{++}([\kappa]^\kappa,\mathcal J^\kappa_{\kappa},\theta)$ holds.
\end{enumerate}
\end{cor}
\begin{proof} (1) By Theorem~\ref{prop46}(2), using \cite{paper14}.

(2) By \cite[Theorem~7.4]{paper47}, in particular, $\ubd([\kappa]^\kappa,J^{\bd}[\kappa],\theta)$ holds.
Now, appeal to Corollary~\ref{tripleupgrade}.
\end{proof}

\begin{cor}\label{cor8} Suppose that $\kappa$ is a regular uncountable cardinal, $\varkappa\le\kappa$,
and $\ubd(J^{\bd}[\kappa],\varkappa)$ holds. Then:
\begin{enumerate}[(1)]
\item For every $\theta\le\varkappa$, $\ubd^+(\mathcal J^\kappa_\kappa,\theta)$ holds;
\item For every regular $\theta<\varkappa$, $\onto^+(\mathcal J^\kappa_\kappa,\theta)$ holds;
\item For every $\theta\le\varkappa$ such that $2^\theta\le\kappa$, $\onto^+(\mathcal J^\kappa_\kappa,\theta)$ holds;
\item For every regular $\theta<\varkappa$ such that $\cf([\kappa]^\theta,{\s})=\kappa$, $\onto^{++}(\mathcal J^\kappa_\kappa,\theta)$ holds.
\end{enumerate}
\end{cor}
\begin{proof} By Theorem~\ref{uparrowlemma}(3), $\ubd^+(\mathcal J^\kappa_\kappa,\varkappa)$ holds.
Thus, Clauses (1)--(3) follow from Corollary~\ref{cmonotone},
Now what about Clause~(4)? The conclusion follows in the same way, but only under the additional assumption that $\kappa\ge\theta^{++}$.
The remaining case, when $\kappa=\theta^+$ for a regular cardinal $\theta$, is covered by Corollary~\ref{prop41}.
\end{proof}
\begin{remark} When put together with \cite[Corollary~3.10]{paper47},
this shows that if $\onto(J^{\bd}[\kappa],\theta)$ fails for a pair $\theta<\kappa$ of infinite regular cardinals, then $\kappa$ is greatly Mahlo.
This improves Clause~(1) of \cite[Theorem~D]{paper47},
which derived the same conclusion from the failure of $\onto(\ns_\kappa,\theta)$.
\end{remark}

Question~8.1.14 of \cite{TodWalks} asks whether a strong limit regular uncountable cardinal $\kappa$ is not weakly compact iff $\kappa\nrightarrow[\kappa]^2_\omega$.
By Theorem~\ref{prop46}(1), the latter is equivalent to $\onto^+(\self,\mathcal J^\kappa_\kappa,\omega)$. Hence, the next theorem
(which improves \cite[Corollary~10.4]{paper47}) is a step in the right direction.

\begin{thm}\label{treelemma}Suppose that $\kappa\ge\mathfrak d$ is a regular cardinal that is not weakly compact.

Then  $\onto^+(\mathcal J^\kappa_\kappa,\omega)$ holds.
\end{thm}
\begin{proof} $\br$ If $\kappa^{\aleph_0}>\kappa$, then since $\cf(\kappa)>\aleph_0$,
we may fix some $\lambda<\kappa$ such that $\lambda^{\aleph_0}\ge\kappa$,
and then the conclusion follows from Corollary~\ref{cor68}(2) below.

$\br$  If $\kappa^{\aleph_0}=\kappa$, then by \cite[Theorem~10.2]{paper47},
$\onto^+(J^{\bd}[\kappa],\omega)$ holds. Now, appeal to Theorem~\ref{uparrowlemma}(1).
\end{proof}

\begin{thm}\label{rectangular} Suppose that $\kappa=\kappa^{\aleph_0}$ and $\kappa\nrightarrow[\kappa;\kappa]^2_2$.
Then $\onto^{++}(\mathcal J^\kappa_\kappa,\omega)$ holds.
\end{thm}
\begin{proof} We follow the proof of Theorem~\ref{thm814}.
Fix an enumeration $\langle x_\eta \mid \eta< \kappa\rangle$ of all the elements in ${}^\omega \kappa$.
Fix a colouring $c:[\kappa]^2\rightarrow2$ witnessing $\kappa\nrightarrow[\kappa;\kappa]^2_2$.
Derive a colouring $d: [\kappa]^2 \rightarrow \omega$ by letting $d(\eta, \beta)$ be the least $n$ such that $c(x_\eta(n), \beta) =1$ if such an $n$ exists, and if not, $d(\eta, \beta) := 0$.
\begin{claim} Suppose that $\langle B_m \mid m< \omega\rangle$ is a sequence of sets in $J^+$,
for a given $J\in\mathcal J^\kappa_\kappa$.
Then there exists an $\eta < \kappa$ such that, for all $m< \omega$ and  $i < 2$,
the set	$\{\beta \in B_m \mid c(\eta, \beta) = i\}$  is in $J^+$.
\end{claim}
\begin{why}
Suppose not. Then for all $\eta < \kappa$ there are $m_\eta < \omega$, $i_\eta< 2$ and $E_\eta \in J^*$ such that $i_\eta \notin c[\{\eta\} \circledast (B_{m_\eta} \cap E_\eta)]$.
Fix $m^* < \omega$ and $i^* < 2$ for which  $A:=\{\eta<\kappa\mid m_\eta = m^*\text{ and }i_\eta = i^*\}$ is cofinal in $\kappa$.
As $J\in\mathcal J^\kappa_\kappa$, we may recursively construct two cofinal subsets $A' \s A$ and $B' \s B_{m^*}$ such that for every $(\eta, \beta) \in A' \circledast B'$, $\beta \in E_{\eta}$. In particular, $i^* \notin c[A' \circledast B']$. This contradicts the hypothesis on $c$.
\end{why}
The rest of the proof is now identical to that of Theorem~\ref{thm814}.
\end{proof}

\begin{defn}[{\cite[Definition~1.2]{paper34}}]\label{defUp}
$\U(\kappa, \mu, \theta, \chi)$ asserts the existence of a colouring $c:[\kappa]^2\rightarrow \theta$ such that for every $\sigma < \chi$, every pairwise disjoint subfamily
$\mathcal{A} \s [\kappa]^{\sigma}$ of size $\kappa$,
and every $i < \theta$, there exists $\mathcal{B} \in [\mathcal{A}]^\mu$
such that $\min(c[a \times b]) > i$ for all $a, b \in \mathcal{B}$ with $\sup(a)<\min(b)$.
\end{defn}

\begin{lemma}\label{lemma34} Let $\theta\in\reg(\kappa)$. Then:
\begin{enumerate}[(1)]
\item $\U(\kappa,2,\theta,2)$ iff $\ubd^+(\self,\allowbreak J^{\bd}[\kappa], \theta)$;
\item  $\U(\kappa,2,\theta,3)$ implies $\ubd^{++}([\kappa]^\kappa,\mathcal J^\kappa_\kappa, \theta)$.
\end{enumerate}
\end{lemma}
\begin{proof} (1) By the same proof of \cite[Proposition~6.4]{paper47},
$\U(\kappa,2,\theta,2)$ implies $\ubd(\self,\allowbreak J^{\bd}[\kappa], \theta)$.
By Theorem~\ref{thm42}(2), $\ubd(\self,\allowbreak J^{\bd}[\kappa], \theta)$ implies $\ubd^+(\self,\allowbreak J^{\bd}[\kappa], \theta)$.
It is trivial to see that $\ubd^+(\self,\allowbreak J^{\bd}[\kappa], \theta)$ implies $\U(\kappa,2,\theta,2)$.

(2) By Corollary~\ref{tripleupgrade}, it suffices to prove that $\U(\kappa,2,\theta,3)$ implies $\ubd([\kappa]^\kappa,\allowbreak J^{\bd}[\kappa],\theta)$.
To this end, suppose that $c:[\kappa]^2\rightarrow\theta$ is a colouring witnessing $\U(\kappa,2,\theta,3)$.
If $c$ fails to witness $\ubd([\kappa]^\kappa,\allowbreak J^{\bd}[\kappa], \theta)$, then we may fix $A,B\in[\kappa]^\kappa$
such that $\sup(c[\{\eta\}\circledast B])<\theta$ for all $\eta\in A$.
In particular, we may fix $A'\in[A]^\kappa$ and $\sigma<\theta$ such that
$\sup(c[\{\eta\}\circledast B])=\sigma$ for all $\eta\in A$. Let $\mathcal A$ be a pairwise disjoint subfamily of $[\kappa]^2$ of size $\kappa$ such that,
for every $x\in\mathcal A$, $\max(x)\in A'$ and $\min(x)\in B$.
As $c$ witnesses $\U(\kappa,2,\theta,3)$, we may now pick $a,b\in\mathcal A$ with $\max(a)<\min(b)$ such that $\min(c[a\times b])>\sigma$.
Set $\eta:=\max(a)$ and $\beta:=\min(b)$. Then $(\eta,\beta)\in A'\circledast B$ and $c(\eta,\beta)>\sigma$. This is a contradiction.
\end{proof}

By \cite[Corollary~5.21]{paper35}, for every $\theta\in\reg(\kappa)$, $\theta\in\cspec(\kappa)$ iff there is a closed witness to $\U(\kappa,2,\theta,\theta)$.
So, by the previous results of this section, $\ubd^{++}([\kappa]^\kappa,\mathcal J^\kappa_{\kappa},\theta)$ holds for every $\theta\in \cspec(\kappa)\cap \reg(\kappa)$.
Recalling Corollary~\ref{thm54}, the case remaining open is $\theta\in\cspec(\kappa)\cap\sing(\kappa)$ with $\theta^+<\kappa$.
Assuming an anti-large-cardinal hypothesis such as $\ssh$ (\emph{Shelah's Strong Hypothesis}; see \cite[\S8.1]{MR1112424}),
for every such singular cardinal $\theta$, $\cf([\theta]^{<\theta},{\s})<\kappa$.
Then the following proposition will help complete the picture.

\begin{prop}\label{csingular} Suppose that $\theta\in\cspec(\kappa)$ and $\cf([\theta]^{<\theta},{\s})<\kappa$.

Then $\ubd^{++}([\kappa]^\kappa,\mathcal J^\kappa_{\kappa},\theta)$ holds.
\end{prop}
\begin{proof} Fix a $C$-sequence $\vec C=\langle C_\gamma\mid\gamma<\kappa\rangle$ such that $\chi(\vec C)=\theta$.
Fix $\Delta\in[\kappa]^\kappa$ and $b:\kappa\rightarrow{}^\theta\kappa$ such that, for every $\beta<\kappa$, $\Delta\cap\beta\s\bigcup_{i<\theta}C_{b(\beta)(i)}$.
For every $\eta<\kappa$, let $\delta_\eta:=\min(\Delta\setminus\eta)$.
Fix an upper regressive colouring $c:[\kappa]^2\rightarrow\theta$ such that, for all $\eta<\beta<\kappa$, if $\delta_\eta<\beta$, then $c(\eta,\beta):=\min\{i<\theta\mid \delta_\eta\in C_{b(\beta)(i)}\}$.

We claim that $c$ witnesses  $\ubd^{++}([\kappa]^\kappa,\mathcal J^\kappa_\kappa,\theta)$.
Suppose not.
Fix $A\in[\kappa]^\kappa$,
$J\in\mathcal J^\kappa_\kappa$, and a sequence $\langle B_\tau\mid\tau<\theta\rangle$ of sets in $J^+$ such that, for every $\eta\in A$,
for some $\tau_\eta<\theta$,
$$T_\eta:=\{ i<\theta\mid \{ \beta\in B_{\tau_\eta}\setminus(\eta+1)\mid c(\eta,\beta)=i\}\in J^+\}$$ has order-type less than $\theta$.
As $\cf([\theta]^{<\theta},{\s})<\kappa$, there must exist $\tau<\theta$ and $T\in[\theta]^{<\theta}$ for which the following set is cofinal in $\kappa$:
$$A':=\{ \eta\in A\mid \tau_\eta=\tau\ \&\ T_\eta\s T\}.$$
Evidently, $\Delta':=\{ \delta_\eta\mid \eta\in A'\}$ is an element of $[\kappa]^\kappa$.
\begin{claim} Let $\epsilon<\kappa$. There exists $\Gamma\s\kappa$ with $|\Gamma|\le|T|$ such that $\Delta'\cap\epsilon\s\bigcup_{\gamma\in\Gamma}C_\gamma$.
\end{claim}
\begin{why} For all $\eta\in A'\cap\epsilon$ and $i\in\theta\setminus T$, since $i\notin T_\eta$,
there exists some $E_{\eta,i}\in J^*$ disjoint from $\{ \beta\in B_{\tau}\setminus(\eta+1)\mid c(\eta,\beta)=i\}$.
As $J$ is $\kappa$-complete, $B_\tau\cap\bigcap\{ E_{\eta,i}\mid \eta\in A'\cap\epsilon, i\in\theta\setminus T\}$  is in $J^+$
so we may pick an element $\beta$ in that intersection that is above $\epsilon$.
Set $\Gamma:=\{ b(\beta)(i)\mid i\in T\}$.

Now, given $\delta\in\Delta'\cap\epsilon$, find $\eta\in A'\cap\epsilon$ such that $\delta=\delta_\eta$,
and then notice that since $\beta\in B_\tau\cap\bigcap_{i\in\theta\setminus T}E_{\eta,i}$, it is the case that $c(\eta,\beta)\in T$.
So, since  $\beta>\epsilon>\delta$, this means that $\delta_\eta\in C_{b(\beta)(i)}$ for some $i\in T$. Altogether, $\Delta'\cap\epsilon\s\bigcup_{\gamma\in\Gamma}C_\gamma$.
\end{why}
It follows that $\chi(\vec C)\le|T|<\theta$.  This is a contradiction.
\end{proof}

\begin{cor} Assuming $\ssh$, $\ubd^{++}([\kappa]^\kappa,\mathcal J^\kappa_{\kappa},\theta)$ holds
for every $\theta\in\cspec(\kappa)$. \qed
\end{cor}

We conclude this section by providing a sufficient condition for $\onto^+(\ldots)$ to hold with the maximal possible number of colours.
For the definition of $\p^\bullet(\ldots)$, see \cite[Definition~5.9]{paper23}.

\begin{thm}\label{proxybullet} Suppose that $\p^\bullet(\kappa,\kappa^+,{\sq},1)$ holds.
Then:
\begin{enumerate}[(1)]
\item $\onto^+([\kappa]^\kappa,\mathcal J^\kappa_\kappa,\kappa)$ holds;
\item For every cardinal $\theta<\kappa$, $\onto^{++}([\kappa]^\kappa,\mathcal J^\kappa_\kappa,\theta)$ holds.
\end{enumerate}
\end{thm}
\begin{proof} (1)
According to \cite[\S5]{paper23}, $\p^\bullet(\kappa,\kappa^+,{\sq},1)$ provides us with a sequence $\langle\mathcal C_\beta\mid\beta<\kappa\rangle$ satisfying the following:
\begin{enumerate}
\item for every $\beta<\kappa$, $\mathcal C_\beta$ is a nonempty collection of functions $C:\mathring{C}\rightarrow  H_\kappa$ such that $\mathring{C}$ is a closed subset of $\beta$ with $\sup(\mathring C)=\sup(\beta)$;\footnote{$H_\kappa$ denotes the collection of all sets of hereditary cardinality less than $\kappa$.}
\item for all $\beta<\kappa$, $C\in\mathcal C_\beta$ and $\alpha\in\acc(\mathring C)$, $C\restriction \alpha\in\mathcal C_{\alpha}$;
\item for all $\Omega\s H_\kappa$ and $p\in H_{\kappa^+}$, there exists $\delta\in\acc(\kappa)$
such that, for all $C\in\mathcal C_\delta$ and $\epsilon<\delta$,
there exists an elementary submodel $\mathcal M\prec H_{\kappa^+}$ with $\{\epsilon,p\}\in\mathcal M$
such that $\tau:=\mathcal M\cap\kappa$ is in $\nacc(\mathring C)$ and $C(\tau)=\mathcal M\cap\Omega$.
\end{enumerate}
For our purposes, it suffices to consider the following special case of (iii):
\begin{enumerate}
\item[(iii')] for every function $g: \kappa\rightarrow \kappa$, there exists $\delta\in\acc(\kappa)$
such that, for all $C\in\mathcal C_\delta$,
$$\sup\{\tau\in \nacc(\mathring C)\mid g[\tau]\s\tau\ \&\ C(\tau)=g\restriction \tau\}=\delta.$$
\end{enumerate}

In particular,\footnote{Note that the proof of \cite[Claim~5.11.1]{paper23}
moreover shows that $H_\kappa=\bigcup_{\alpha\in\acc(\kappa)}\im(C_\alpha)$ for every transversal $\langle C_\alpha\mid\alpha<\kappa\rangle\in\prod_{\alpha<\kappa}\mathcal C_\alpha$.}
$\kappa^{<\kappa}=\kappa$.
Now, if $\kappa=\nu^+$ is a successor cardinal, then by \cite[Lemma~8.3(1)]{paper47},
$\onto^+([\kappa]^\nu,\allowbreak J^{\bd}[\kappa],\kappa)$ holds,
and then, by Proposition~\ref{prop614}(2), $\onto^+([\kappa]^\nu,\allowbreak \mathcal J^\kappa_\kappa,\kappa)$ holds. Thus, hereafter, assume that $\kappa$ is inaccessible.
By \cite[Lemma~8.5(2)]{paper47}, in this case, it suffices to prove that $\onto^+(\mathcal J^\kappa_\kappa, \kappa)$ holds.

To this end, fix a sequence $\vec C=\langle C_\beta\mid \beta<\kappa\rangle$ such that, $C_\beta\in\mathcal C_\beta$ for all $\beta<\kappa$.
We shall conduct walks on ordinals along
$\vec{\mathring C}:=\langle \mathring C_\beta\mid \beta<\kappa\rangle$ and all terms regarding walks on ordinals should be understood as pertaining to this fixed $C$-sequence.
Define an upper-regressive colouring $c:[\kappa]^2\rightarrow\kappa$ as follows. Given $\eta<\beta<\kappa$,
let $\gamma:=\min(\im(\tr(\eta,\beta)))$
and then let  $c(\eta,\beta):=C_\gamma(\min(\mathring C_\gamma\setminus(\eta+1)))(\eta)$
provided that the latter is a well-defined ordinal less than $\beta$; otherwise, let $c(\eta,\beta):=0$.

Towards a contradiction, suppose that $J\in\mathcal J^\kappa_\kappa$ and $B\in J^+$ are such that there exists a function $g:\kappa\rightarrow\kappa$ such that,
for every $\eta<\kappa$, $\{\beta\in B\setminus(\eta+1)\mid c(\eta,\beta)=g(\eta)\}$ is in $J$.
Fix $\delta\in\acc(\kappa)$ as in Clause~(iii').
\begin{claim} For every $\beta\in B$ above $\delta$,
there exists $\eta<\delta$ such that $c(\eta,\beta)=g(\eta)$.
\end{claim}
\begin{why} Let $\beta \in B$ be above $\delta$ and let $\gamma:= \last{\delta}{\beta}$ in the sense of \cite[Definition~2.10]{paper44} so that $\delta \le \gamma\le \beta$. Since $\delta \in \acc(\kappa)$, it is also the case that $\sup(\mathring C_\gamma\cap \delta)= \delta$. By Clause (ii) this implies that $C_\gamma \restriction \delta \in \mathcal C_\delta$. As $\delta$ was chosen as in Clause (iii') we have that
$$\sup\{\tau\in \nacc(\mathring C_\gamma)\mid g[\tau]\s\tau\ \&\ C_\gamma(\tau)=g\restriction \tau\}=\delta.$$
By \cite[Lemma~2.11]{paper44}, $\Lambda:= \lambda(\gamma, \beta)$ is less than $\delta$.
Therefore, for every $\eta\in \mathring C_\gamma$ larger than $\Lambda$, $\min(\im(\tr(\eta,\beta)))=\gamma$.
Now pick $\tau \in \nacc(\mathring C_\gamma)$ such that $g[\tau] \s \tau$ and $C_\gamma(\tau) = g\restriction \tau$ and such that $\eta:= \sup(\mathring C_\gamma \cap \tau)$ is larger than $\Lambda$. It follows that
$$c(\eta, \beta) = (C_\gamma(\tau))(\eta) = (g\restriction \tau)(\eta) = g(\eta),$$ as required.
\end{why}
For every $\beta\in B$ above $\delta$, let $\eta_\beta<\delta$ be given by the preceding claim.
As $J$ is $\kappa$-complete, there exists some $\eta<\delta$ such that $B':=\{ \beta\in B\setminus(\delta+1)\mid \eta_\beta=\eta\}$ is in $J^+$.
Then, $\{ \beta\in B\setminus(\eta+1)\mid c(\eta,\beta)=g(\eta)\}$ covers $B'$, contradicting the choice of $g(\eta)$.

(2) Given a cardinal $\theta<\kappa$, by Clause~(1), in particular, $\onto^+([\kappa]^\kappa,\mathcal J^\kappa_\kappa,\theta)$ holds.
Now appeal to Lemma~\ref{lemma39}(1).
\end{proof}

\section{Scales}\label{sectionscales}

\begin{lemma}\label{lemma712a}  Suppose that $\theta<\kappa$ are infinite cardinals.
Consider the following statements:
\begin{enumerate}[(1)]
\item $\ubd(\{\theta\},J^\bd[\kappa],\theta)$ holds;
\item There exists a sequence $\vec g=\langle g_\beta\mid \beta<\kappa\rangle$ of functions from $\theta$ to $\theta$,
such that, for every cofinal $B\s\kappa$, $\{ g_\beta\mid\beta \in B\}$ is unbounded in $({}^\theta\theta,<^*)$.
\end{enumerate}
If $\cf(\theta)\neq\cf(\kappa)$, then $(1)\implies(2)$. If $\cf(\theta)=\theta$, then $(2)\implies(1)$.
\end{lemma}
\begin{proof} $(1)\implies(2)$:
Suppose that $c:[\kappa]^2\rightarrow\theta$ is a colouring witnessing $\ubd(\{\theta\},\allowbreak J^\bd[\kappa],\theta)$.
Fix a surjection $\sigma:\theta\rightarrow\theta$ such that the preimage of any singleton is cofinal in $\theta$.
For every $\beta<\kappa$, derive $g_\beta:\theta\rightarrow\theta$ via $g_\beta(\tau):=c(\{\sigma(\tau),\beta\})$.
Towards a contradiction, suppose that there exists a cofinal $B\s\kappa$ such that $\{ g_\beta\mid \beta\in B\}$ is bounded in $({}^\theta\theta,<^*)$.
Pick a function $g:\theta\rightarrow\theta$ such that,
for every $\beta\in B$, for some $\epsilon_\beta<\theta$, $g_\beta(\tau)<g(\tau)$ whenever  $\epsilon_\beta\le \tau<\theta$.
Assuming that $\cf(\theta)\neq\cf(\kappa)$, we may pick $\epsilon<\theta$ for which $$B':=\{\beta\in B\mid \epsilon_\beta\le \epsilon\}$$
is cofinal in $\kappa$. Now, for every $\eta<\theta$, we may find $\tau\in\sigma^{-1}\{\eta\}$ above $\epsilon$,
and then $$c[\{\eta\}\circledast B']\s \{g_\beta(\tau)\mid \beta\in B'\}\s g(\tau),$$
so that $\otp(c[\{\eta\}\circledast B'])<\theta$.
This is a contradiction.

$(2)\implies(1)$: Given a sequence $\langle g_\beta\mid \beta<\kappa\rangle$ as above,
pick any upper-regressive colouring $c:[\kappa]^2\rightarrow\theta$ such that, for all $\eta\le\theta\le\beta<\kappa$,
$c(\eta,\beta)=g_\beta(\eta)$.
Towards a contradiction, suppose that $c$ fails to be a witness to $\ubd(\{\theta\},J^\bd[\kappa],\theta)$.
Then there exists a cofinal $B\s\kappa$ such that, for every $\eta<\theta$,
$$\otp(c[\{\eta\}\circledast B])<\theta.$$

Assuming that $\theta$ is regular,
we may define a function $g:\theta\rightarrow\theta$ via:
$$g(\eta):=\sup(c[\{\eta\}\circledast B]).$$
Then $g$ witnesses that $\{ g_\beta\mid \beta\in B\}$ is bounded.
\end{proof}

\begin{remark} For the implication $(1)\implies(2)$, the requirement ``$\cf(\theta)\neq\cf(\kappa)$'' cannot be waived.
For instance, if $\mathfrak b>\aleph_\omega$,
then for $\theta:=\aleph_0$ and $\kappa:=\aleph_\theta$,
we get from \cite[Propositions 6.1 and 7.8]{paper47} that $\ubd(\{\theta\},J^\bd[\kappa],\theta)$ holds,
but because $\mathfrak b>\kappa$, any $\kappa$-sized family of reals is bounded.
\end{remark}

\begin{cor}\label{lemma712} Suppose that $\theta$ is an infinite regular cardinal.
\begin{enumerate}[(1)]
\item $\ubd(\{\theta\},J^\bd[{\mathfrak b_\theta}],\theta)$ holds;
\item $\ubd(\{\theta\},J^\bd[{\mathfrak d_\theta}],\theta)$ holds;
\item If $\ubd(\{\theta\},J^{\bd}[\kappa],\theta)$ holds and $\cf(\kappa)\neq\theta$, then $\mathfrak b_\theta\le\cf(\kappa)\le \mathfrak d_\theta$.
In particular, $\ubd(\{\theta\},J^{\bd}[\theta^+],\theta)$ holds iff $\mathfrak b_\theta=\theta^+$.
\end{enumerate}
\end{cor}
\begin{proof} (1) Denote $\kappa:=\mathfrak b_\theta$ and note that $\kappa=\cf(\kappa) > \theta$.
Let $\vec f=\langle f_\beta\mid \beta<\kappa\rangle$ denote an enumeration of some unbounded family in $({}^\theta\theta,<^*)$.
For every $\beta<\kappa$, as $\beta<\mathfrak b_\theta$, let us fix $g_\beta:\theta\rightarrow\theta$ such that,
for every $\alpha\le\beta$, $f_\alpha<^* g_\beta$.
We claim that $\vec g=\langle g_\beta\mid\beta<\kappa\rangle$ is as in Lemma~\ref{lemma712a}.
Towards a contradiction, suppose that we may fix a cofinal $B\s\kappa$
such that $\{ g_\beta\mid\beta \in B\}$ is bounded in $({}^\theta\theta,<^*)$.
Pick a function $g:\theta\rightarrow\theta$ such that $g_\beta<^*g$ for every $\beta\in B$.
By the choice of $\vec f$, find $\alpha<\kappa$ such that $\neg(f_\alpha <^* g)$, and then fix $\beta\in B$ above $\alpha$.
As $\neg(f_\alpha <^* g)$ and  $f_\alpha<^* g_\beta$, we infer that $\neg(g_\beta <^* g)$.
This is a contradiction.

(2) 	Denote $\kappa:=\mathfrak d_\theta$ and note that $\cf(\kappa) > \theta$.
Let $\vec f=\langle f_\beta\mid \beta<\kappa\rangle$ denote an enumeration of some dominating family in $({}^\theta\theta,<^*)$.
For every $\beta<\kappa$, as $\beta<\mathfrak d_\theta$, let us fix $g_\beta:\theta\rightarrow\theta$ such that,
for every $\alpha\le\beta$, $\neg(g_\beta<^* f_\alpha)$.
We claim that $\vec g=\langle g_\beta\mid\beta<\kappa\rangle$ is as in Lemma~\ref{lemma712a}.
Towards a contradiction, suppose that we may fix a cofinal $B\s\kappa$
such that $\{ g_\beta\mid\beta \in B\}$ is bounded in $({}^\theta\theta,<^*)$.
Pick a function $g:\theta\rightarrow\theta$ such that 	$g_\beta<^*g$ for every $\beta\in B$.
By the choice of $\vec f$, find $\alpha<\kappa$ such that $g<^* f_\alpha $, and then fix $\beta\in B$ above $\alpha$.
As $g <^* f_\alpha$ and $\neg(g_\beta<^* f_\alpha)$, we infer that $\neg(g_\beta<^* g)$.
This is a contradiction.

(3) Suppose that $\ubd(\{\theta\},J^{\bd}[\kappa],\theta)$ holds, and $\cf(\kappa)\neq\theta$.
By Lemma~\ref{lemma712a}, fix a sequence $\vec g=\langle g_\beta\mid \beta<\kappa\rangle$ of functions from $\theta$ to $\theta$,
such that, for every cofinal $B\s\kappa$, $\im(\vec g\restriction B)$ is unbounded in $({}^\theta\theta,<^*)$.
Fix an injective map $\varphi:\cf(\kappa)\rightarrow\kappa$ whose image is cofinal in $\kappa$. For each $\beta<\cf(\kappa)$, set $f_\beta:=g_{\varphi(\beta)}$.
Then $\vec f:=\langle f_\beta\mid \beta<\cf(\kappa)\rangle$ is a sequence of functions from $\theta$ to $\theta$,
such that, for every cofinal $B\s\cf(\kappa)$, $\im(f\restriction B)$ is unbounded in $({}^\theta\theta,<^*)$.
In particular, $\mathfrak b_\theta\le\cf(\kappa)$.
Next, let $\mathcal D$ be a cofinal family in $({}^\theta\theta,<^*)$ of size $\mathfrak d_\theta$.
For each $\beta<\cf(\kappa)$, fix $h_\beta\in\mathcal D$ that dominates $f_\beta$.
Now, if $\mathfrak d_\theta<\cf(\kappa)$, then there exists $B\in[\cf(\kappa)]^{\cf(\kappa)}$ on which the map $\beta\mapsto h_\beta$ is constant over $B$,
and then $\im(\vec f\restriction B)$ is bounded, contradicting the choice of $\vec f$.
\end{proof}
In contrast with Clause~(3) of the preceding, we shall prove in Corollary~\ref{cor215} below that $\ubd(\{\theta\},J^{\bd}[\theta^+],\theta)$ holds for every singular cardinal $\theta$.

\begin{cor}\label{cor55} Suppose that $\theta$ is an infinite regular cardinal.
\begin{enumerate}[(1)]
\item If $\kappa=\mathfrak b_\theta$ and $\onto(J^{\bd}[\theta],\theta)$ holds, then so does
$\onto^+(\{\theta\},\mathcal J^\kappa_{\theta^+},\theta)$;
\item If $\kappa=\mathfrak b_\theta=\theta^+$ and $\onto(J^{\bd}[\theta],\theta)$ holds,
then so does $\onto^+(\{\theta\},\mathcal J^{\kappa}_{\kappa},\kappa)$;
\item If $\kappa=\mathfrak d_\theta$, then $\onto^+(\mathcal J^{\kappa}_{\theta^+},\theta)$ holds;
\item If $\kappa=\mathfrak d_\theta$ and $\mathfrak b_\theta>\theta^+$,
then $\ubd^+(\mathcal J^{\kappa}_{\theta^+},\theta^+)$ implies $\onto^+(\mathcal J^{\kappa}_{\theta^+},\theta^+)$.
\end{enumerate}
\end{cor}
\begin{proof}
Assuming $\kappa\in\{\mathfrak b_\theta,\mathfrak d_\theta\}$,
by Corollary~\ref{lemma712}, $\ubd(\{\theta\},J^{\bd}[\kappa],\theta)$ holds.
Then, by Proposition~\ref{prop614}(1), moreover $\ubd^+(\{\theta\},\mathcal J^{\kappa}_{\theta^+},\theta)$ holds.

(1) If $\onto(J^{\bd}[\theta],\theta)$ holds,
then so does $\proj(\theta,\theta,\theta,1)$.
So by Theorem~\ref{ubdtoonto}(2), using $\nu=\varkappa=\theta$ and the consequence of $\mathfrak b_\theta=\kappa$ that we noticed,
$\onto^+(\{\theta\},\mathcal J^\kappa_{\theta^+},\theta)$ holds.

(2) If $\kappa=\mathfrak b_\theta=\theta^+$ and $\onto(J^{\bd}[\theta],\allowbreak\theta)$ holds,
then, by Clause~(1), in particular $\onto(\{\theta\},J^{\bd}[\kappa],\allowbreak\theta)$ holds.
By \cite[Lemma~6.15(3)]{paper47}, then, $\onto(\{\theta\},J^{\bd}[\kappa],\kappa)$ holds.
Finally, by Theorem~\ref{uparrowlemma}(1),
$\onto^+(\{\theta\},\mathcal J^{\kappa}_{\kappa},\kappa)$ holds.

(3) If $\kappa=\mathfrak d_\theta$,
then since $\ubd^+(\mathcal J^{\kappa}_{\theta^+},\theta)$ holds,
Lemmas \ref{d1} and \ref{ubdtoonto} imply that so does $\onto^+(\mathcal J^{\kappa}_{\theta^+},\theta)$.

(4) Assuming $\kappa=\mathfrak d_\theta$,
fix a colouring $c:[\kappa]^2\rightarrow\theta$ witnessing $\ubd^+(\{\theta\},\mathcal J^\kappa_{\theta^+},\allowbreak\theta)$.
Suppose that $d:[\kappa]^2\rightarrow\theta^+$ is a colouring witnessing $\ubd^+(\mathcal J,\theta^+)$.
For every $\alpha<\theta^+$, fix a surjection $e_\alpha:\theta\rightarrow\alpha+1$.
Assuming $\mathfrak b_\theta>\theta^+$, appeal to Lemma~\ref{d1} to fix a map $p:\kappa\times\theta\rightarrow\theta$ witnessing $\proj(\kappa,\theta,\theta,\theta^+)$.
Also, fix a bijection $\pi:\kappa\leftrightarrow\kappa\times\kappa\times\theta$.
Finally, define a colouring $f:[\kappa]^2\rightarrow\theta^+$ as follows.
Given $\eta<\beta<\kappa$, let $(\eta_0,\eta_1,\eta_2):=\pi(\eta)$ and then set $f(\eta,\beta):=e_{d(\{\eta_0,\beta\})}(p(\eta_1,c(\{\eta_2,\beta\})))$.

To see this works, let $J\in\mathcal J^\kappa_{\theta^+}$ and $B\in J^+$.
For all $\eta,\alpha<\kappa$, denote $B^{\eta,\alpha}:=\{\beta\in B\setminus(\eta+1)\mid d(\eta,\beta)=\alpha\}$.
By the choice of $d$, we may fix an $\eta_0<\kappa$ such that $A_{\eta_0}(B):=\{\alpha<\theta^+\mid B^{\eta_0,\alpha}\in J^+\}$ has size $\theta^+$.
For each $\alpha\in A_{\eta_0}(B)$, as $B^{\eta_0,\alpha}$ is in $J^+$,
we may find an $\eta^\alpha<\theta$ such that the following set is in $[\theta]^\theta$:
$$X_\alpha:=\{\xi<\theta\mid \{\beta\in B^{\eta_0,\alpha}\setminus (\eta^\alpha+1)\mid c(\eta^\alpha,\beta)=\xi\}\in J^+\}.$$
Pick $\eta_2<\theta$ for which $A:=\{\alpha\in A_{\eta_0}(B)\mid \eta^\alpha=\eta_2\}$ has size $\theta^+$.
Finally, pick $\eta_1<\kappa$ such that $p[\{\eta_1\}\times X_\alpha]=\theta$ for all $\alpha\in A$.
Fix $\eta<\kappa$ such that $\pi(\eta)=(\eta_0,\eta_1,\eta_2)$.
Given any colour $\tau<\theta^+$, pick $\alpha\in A\setminus\tau$.
Then pick $i<\theta$ such that $e_\alpha(i)=\tau$.
Then pick $\xi\in X_\alpha$ such that $p(\eta_1,\xi)=i$.
Then pick $\beta\in B^{\eta_0,\alpha}$ above $\max\{\eta_0,\eta_2,\eta\}$ such that $c(\eta_2,\beta)=\xi$.
Then $f(\eta,\beta)=e_{d(\eta_0,\beta)}(p(\eta_1,c(\eta_2,\beta)))=e_\alpha(p(\eta_1,\xi))=e_\alpha(i)=\tau$, as sought.
\end{proof}

The proof of the implication $(2)\implies(1)$ of Lemma~\ref{lemma712a} makes it clear that the following holds, as well.

\begin{lemma}\label{lemma66} Suppose that $\kappa=\mathfrak b_\theta=\mathfrak d_\theta$
for an infinite regular cardinal $\theta$.

Then $\ubd([\theta]^\theta,J^{\bd}[\kappa],\theta)$ holds.\qed
\end{lemma}
\begin{remark} Analogous results may be obtained from the existence of a $\kappa$-Luzin subset of ${}^\theta\theta$.
In particular, for an infinite cardinal $\theta=\theta^{<\theta}$, if $\kappa=\cov(\mathcal M_\theta)=\cof(\mathcal M_\theta)$,
then $\onto([\theta]^\theta,J^{\bd}[\kappa],\theta)$ holds.
\end{remark}

\begin{cor}\label{ccor68} Suppose that $\kappa=\mathfrak b_{\theta^+}=\mathfrak d_{\theta^+}$
for an infinite cardinal $\theta$. Then:
\begin{enumerate}[(1)]
\item $\ubd^{++}(\{\theta^+\},\mathcal J^\kappa_{\theta^{++}},\theta)$ holds;
\item If $\theta$ is regular, then moreover $\onto^{++}(\{\theta^+\},\mathcal J^\kappa_{\theta^{++}},\theta)$ holds.
\end{enumerate}
\end{cor}
\begin{proof} By Lemma~\ref{lemma66}, pick a colouring $c:[\kappa]^2\rightarrow\theta^+$ witnessing $\ubd([\theta^+]^{\theta^+},\allowbreak J^{\bd}[\kappa],\theta^+)$.
By Proposition~\ref{prop614}(1), $c$ moreover witnesses $\ubd^+([\theta^+]^{\theta^+},\mathcal J^{\kappa}_{\theta^{++}},\allowbreak\theta^+)$.

(1) For every $\gamma<\theta^+$, fix an injection $e_\gamma:\gamma\rightarrow\theta$.
Fix a bijection $\pi:\theta^+\leftrightarrow\theta^+\times\theta^+$
and pick an upper-regressive colouring $d:[\kappa]^2\rightarrow\theta$ such that for all $\eta<\theta^+\le\beta<\kappa$,
if $\pi(\eta)=(\eta',\gamma)$, then $d(\eta,\beta)=e_\gamma(c(\{\eta',\beta\}))$.
To see this works, let $\langle B_\tau\mid \tau<\theta\rangle$ be a sequence of $J^+$-sets, for some $J\in\mathcal J^\kappa_{\theta^{++}}$.
By the choice of $c$, there is a large enough $\eta'<\theta^+$ such that, for every $\tau<\theta$,
the following set has size $\theta^+$:
$$X_\tau:=\{ \xi<\theta^+\mid \{\beta\in B_\tau\setminus(\eta'+1)\mid c(\eta,\beta)=\xi\}\in J^+\}.$$
Pick a large enough $\gamma<\theta^+$ such that $|X_\tau\cap\gamma|=\theta$ for all $\tau<\theta$.
Find $\eta<\theta^+$ such that $\pi(\eta)=(\eta',\gamma)$.
Then, for every $\tau<\theta$,
$$\{ e_\gamma(\xi)\mid \xi<\gamma\ \&\ \{\beta\in B_\tau\setminus(\eta'+1)\mid c(\eta,\beta)=\xi\}\in J^+\}$$
has order-type $\theta$.
Consequently, for every $\tau<\theta$,
$$\{ \zeta<\theta\mid \{\beta\in B_\tau\setminus(\eta'+1)\mid d(\eta,\beta)=\zeta\}\in J^+\}$$
has order-type $\theta$.

(2) Set $\nu:=\theta^+$ and $\varkappa:=\theta^+$. By Corollary~\ref{regularpairs}(2), $\proj(\nu,\varkappa,\theta,\theta)$ holds.
So, by Lemma~\ref{lemma15}(2) using $I:=[\nu]^\nu$, $\onto^{++}(\{\nu\},\mathcal J^\kappa_{\theta^{++}},\theta)$ holds.
\end{proof}

\begin{defn}[Shelah, \cite{Sh:g}] For a singular cardinal $\lambda$, $\PP(\lambda)$ stands for the set of all cardinals $\kappa$ such that
there exists an ultrafilter $\mathcal U$ over $\cf(\lambda)$ disjoint from $J^{\bd}[\cf(\lambda)]$ and a strictly increasing sequence of regular cardinals $\langle \lambda_i\mid i<\cf(\lambda)\rangle$ converging to $\lambda$
such that the linear order $(\prod_{i<\cf(\lambda)}\lambda_i,<_{\mathcal U})$ has cofinality $\kappa$.
\end{defn}
By a theorem of Shelah (see \cite[Theorem~5.3]{MR1112424}), $\PP(\lambda)$ is an interval of regular cardinals with $\min(\PP(\lambda))=\lambda^+$.

\begin{thm}\label{thm811} Suppose that $\kappa\in\PP(\lambda)$ for a singular cardinal $\lambda$.

Then $\onto^+(\{\lambda\},\mathcal J^\kappa_{\lambda^+},\theta)$ holds for every $\theta<\lambda$.
\end{thm}
\begin{proof}
Let $\theta<\lambda$ be arbitrary.
As $\lambda$ is a limit cardinal, we may increase $\theta$ and assume that it is a regular cardinal greater than $\cf(\lambda)$.
By Proposition~\ref{prop614}(1), it suffices to prove that $\onto(\{\lambda\},J^{\bd}[\kappa],\theta)$ holds.
Fix an ultrafilter $\mathcal U$ and a sequence $\langle \lambda_i\mid i<\cf(\lambda)\rangle$
witnessing together that $\kappa\in\PP(\lambda)$. As $\mathcal U$ is disjoint from $J^{\bd}[\cf(\lambda)]$, we may assume that $\lambda_0>\theta^+$.
Fix a strictly increasing and cofinal sequence $\langle f_\beta\mid \beta<\kappa\rangle$ in the linear order $(\prod_{i<\cf(\lambda)}\lambda_i,<_{\mathcal U})$.
For each $i<\cf(\lambda)$, fix a $\theta$-bounded club-guessing $C$-sequence $\langle C_\delta^i\mid \delta\in E^{\lambda_i}_\theta\rangle$.
Fix a bijection $\pi:\lambda\leftrightarrow\bigcup_{i<\cf(\lambda)}(\{i\}\times E^{\lambda_i}_\theta)$.
Finally, pick any colouring $c:[\kappa]^2\rightarrow\theta$ such that, for all $\eta<\lambda\le\beta<\kappa$,
if $\pi(\eta)=(i,\delta)$, then
$$c(\eta,\beta):=\sup(\otp(C_\delta^i\cap f_\beta(i))).$$

To see this works, suppose that we are given $B\in[\kappa]^\kappa$.

\begin{claim}\label{claim111} There are cofinally many $i<\cf(\lambda)$ such that
$$\sup\{ f_\beta(i)\mid \beta\in B\setminus\lambda\}=\lambda_i.$$
\end{claim}
\begin{why} Suppose not.
Then there exists $j<\cf(\lambda)$ and a function $g\in\prod_{i<\cf(\lambda)}\lambda_i$ such that, for every $i\in[j,\cf(\lambda))$,
$$g(i)=\sup\{ f_\beta(i) \mid \beta\in B\setminus\lambda\}.$$
Pick $\alpha<\lambda$ such that $g<_{\mathcal U}f_\alpha$. Then find $\beta\in B\setminus(\lambda\cup\alpha)$.
As $g<_{\mathcal U} f_\alpha<_{\mathcal U}f_\beta$, let us pick $i<\cf(\lambda)$ above $j$ such that $g(i)<f_\beta(i)$. This is a contradiction.
\end{why}
Fix one $i$ as in the claim. Then $D:=\acc^+(\{f_\beta(i)\mid \beta\in B\setminus\lambda\})$ is a club in $\lambda_i$,
and we may fix some $\delta\in E^{\lambda_i}_\theta$ such that $C_\delta^i\s D$.
Set $\eta:=\pi^{-1}(i,\delta)$. Clearly, $c[\{\eta\}\circledast B]=\theta$.
\end{proof}

It follows from the preceding that for every singular cardinal $\lambda$, $\onto(\{\lambda\},J^{\bd}[\lambda^+],\allowbreak\theta)$ holds for every cardinal $\theta<\lambda$.
This cannot be improved any further:
\begin{prop}\label{prop610} Assuming the consistency of a supercompact cardinal, it is consistent that $\onto(\{\lambda\},J^{\bd}[\lambda^+],\lambda)$
fails for some singular strong limit cardinal $\lambda$.
\end{prop}
\begin{proof} By \cite[Lemma~8.9(2)]{paper47},
for every strong limit cardinal $\lambda$, $\onto(\{\lambda\},\allowbreak J^{\bd}[\lambda^+],\allowbreak\lambda)$ implies $\lambda^+\nrightarrow[\lambda;\lambda^+]^2_{\lambda^+}$.
However, by the main result of \cite{Sh:949}, in a suitable forcing extension over a ground model with a supercompact cardinal,
there exists a singular strong limit cardinal $\lambda$ such that $\lambda^+\rightarrow[\lambda;\lambda^+]^2_2$ does hold.
\end{proof}

\begin{lemma} Suppose that $\kappa=\lambda^+$ for a singular cardinal $\lambda$. For every $\theta<\lambda$,
there exists a map $c:\lambda\times\kappa\rightarrow\theta\times\kappa$ such that,
for every $B\in[\kappa]^{\kappa}$ and every stationary $S\s\kappa$, there are $\eta<\lambda$
and a stationary $\Gamma\s S$ such that $c[\{\eta\}\circledast B]\supseteq\theta\times \Gamma$.
\end{lemma}
\begin{proof}
Given $\theta<\lambda$, use Theorem~\ref{thm811} to fix a map $d:\lambda\times\lambda^+\rightarrow\theta$ witnessing $\onto(\{\lambda\},J^{\bd}[\lambda^+],\theta)$.
For every $\beta<\kappa$, fix a surjection $e_\beta:\lambda\rightarrow\beta+1$.
Fix a bijection $\pi:\lambda\leftrightarrow\lambda\times\lambda$.
Define a colouring $c:\lambda\times\kappa\rightarrow\theta\times\kappa$ by letting $$c(\eta,\beta):=(d(j,e_\beta(i)),e_\beta(i))\text{ provided }\pi(\eta)=(i,j).$$

To see this works, let $B\in[\kappa]^\kappa$.
For all $i<\lambda$ and $\gamma<\kappa$, let $B^i_\gamma:=\{ \beta\in B\setminus\lambda\mid e_\beta(i)=\gamma\}$,
and then let $\Gamma^i(B):=\{ \gamma\in S\setminus\lambda\mid B^i_\gamma\in [\kappa]^\kappa\}$.
\begin{claim} There exists $i<\lambda$ such that $\Gamma^i(B)$ is stationary.
\end{claim}
\begin{why} For each $\gamma<\kappa$, there exists some $i<\lambda$ such that $B^i_\gamma\in[\kappa]^\kappa$.
Then, there exists some $i<\lambda$ such that
$\Gamma^i(B)$ is stationary.
\end{why}

Let $i$ be given by the claim. For each $\gamma\in\Gamma^i(B)$, since $B^i_\gamma\in[\kappa]^\kappa$,
we may find some $j_\gamma<\lambda$ such that $d[\{j_\gamma\}\circledast B^i_\gamma]=\theta$.
Pick $j<\lambda$ for which $\Gamma:=\{\gamma\in\Gamma^i(B)\mid j_\gamma=j\}$ is stationary.
Now, pick $\eta<\lambda$ such that $\pi(\eta)=(i,j)$.
\begin{claim} Let $(\tau,\gamma)\in \theta\times \Gamma$. There exists $\beta\in B\setminus\lambda$ such that $c(\eta,\beta)=(\tau,\gamma)$.
\end{claim}
\begin{why} As $\gamma\in \Gamma$, we infer that $j_\gamma=j$, and hence we may find $\beta\in B^i_\gamma$ such that $d(j,\beta)=\tau$.
As $\beta\in B^i_\gamma$, it is the case that $\beta\in B\setminus\lambda$ and $e_\beta(i)=\gamma$.
Recalling that $\pi(\eta)=(i,j)$ and the definition of $c$, we get that
$$c(\eta,\beta):=(d(j,e_\beta(i)),e_\beta(i))=(d(j,\gamma),\gamma)=(\tau,\gamma),$$
as sought.
\end{why}
This completes the proof.
\end{proof}

Recall that by Corollary~\ref{lemma712}(3), for a regular cardinal $\lambda$, $\ubd(\{\lambda\},J^{\bd}[\lambda^+],\allowbreak\lambda)$ holds
iff $\mathfrak b_\lambda=\lambda^+$. In contrast, for singular cardinals, we have the following unconditional result:

\begin{cor}\label{cor215} Suppose that $\kappa=\lambda^+$ for a singular cardinal $\lambda$.

Then $\ubd(\{\lambda\},J^{\bd}[\kappa],\lambda)$ holds.
\end{cor}
\begin{proof} Fix a colouring $c$ as in the preceding lemma, using $\theta:=\cf(\lambda)$.
Fix an increasing sequence of regular cardinals $\langle \lambda_i\mid i<\cf(\lambda)\rangle$ converging to $\lambda$
such that $\tcf(\prod_{i<\cf(\lambda)}\lambda_i,<^*)=\kappa$.
Fix a scale $\langle f_\gamma\mid \gamma<\kappa\rangle$ witnessing the preceding.
Define $d:\lambda\times\lambda^+\rightarrow\lambda$ by letting $d(\eta,\beta):=f_\gamma(i)$ iff $c(\eta,\beta)=(i,\gamma)$.

To see this works, let $B\in[\kappa]^\kappa$.
Pick $\eta<\lambda$ and $\Gamma\in[\lambda^+\setminus\lambda]^{\lambda^+}$ such that $c[\{\eta\}\circledast B]\supseteq\cf(\lambda)\times\Gamma$.
By Claim~\ref{claim111}, $I:=\{i<\cf(\lambda)\mid \sup\{ f_\gamma(i)\mid \gamma\in \Gamma\}=\lambda_i\}$ is cofinal in $\cf(\lambda)$.
In effect, the set $\{ f_\gamma(i)\mid i<\cf(\lambda), \gamma\in\Gamma\}$ has size $\lambda$.
As $d[\{\eta\}\circledast B]$ covers the above $\lambda$-sized set, we are done.
\end{proof}

\section{Independent and almost-disjoint families}\label{adsection}

To motivate the next definition, note that $\kappa\in\ad(\theta^+,\theta)$ iff there exists an almost-disjoint family in $[\theta]^\theta$ of size $\kappa$.

\begin{defn}\label{adf} $\ad(\mu,\theta)$ stands for the set of all cardinals $\kappa$ for which there exists a pair $(\mathcal X,\mathcal Y)$ satisfying all of the following:
\begin{enumerate}
\item $\mathcal X$ and $\mathcal Y$ consist of sets of ordinals with $|\mathcal X|=\kappa$ and $\otp(\mathcal Y,{\s})=\theta$;
\item For all $x\neq x'$ from $\mathcal X$, there exists $y\in\mathcal Y$ such that $x\cap x'\s y$;
\item For every $(x,y)\in\mathcal X\times\mathcal Y$,  $x\setminus y\neq\emptyset$;
\item For every $y\in \mathcal Y$, $\{ \min(x\setminus y)\mid x\in\mathcal X\}$ has size less than $\mu$.
\end{enumerate}
\end{defn}
\begin{remark}\label{admonotonicity} If $\kappa\in\ad(\mu,\theta)$, then $\varkappa\in\ad(\mu,\theta)$ for all cardinals $\varkappa<\kappa$.
If $\lambda$ is singular, then $\PP(\lambda)\s \ad(\lambda,\cf(\lambda))$.
\end{remark}

The following two corollaries summarise the main findings of this section.

\begin{cor}\label{cor73} Suppose that $\theta<\nu<\cf(\kappa)\le2^\nu$ are infinite cardinals.

Any of the following implies that $\onto^+(\{\nu\},\mathcal J^\kappa_{\nu^+},\theta)$ holds:
\begin{enumerate}[(1)]
\item $\nu=\nu^\theta$;
\item $\nu$ is a strong limit;
\item $\theta^+=\nu$, $\cf(\theta)=\theta$, and there is a $\nu$-Kurepa tree with $\cf(\kappa)$ many branches;
\item $\theta^+<\nu$, $\cf(\theta)=\theta$, and $\cf(\kappa)\in\ad(\nu^+,\nu)$;
\item $\theta^+<\nu$, $\mathcal C(\theta^+,\theta)\le\nu$, and $\cf(\kappa)\in\ad(\nu^+,\nu)$;
\item $\theta^{++}<\nu$ and $\cf(\kappa)\in\ad(\nu^+,\nu)$;
\item $\theta^{++}<\nu=\nu^{<\nu}$.
\end{enumerate}
\end{cor}
\begin{proof} (1) By Theorem~\ref{ek} below.

(2) By Corollary~\ref{sepapp} below.

(3) By Corollary~\ref{kurepacor} below.

(4) By Lemma~\ref{ad1} below, using $\mu:=\nu^+$, we get $\ubd^+(\{\nu\},\mathcal J^\kappa_{\nu^+},\theta^+)$.
So, by Lemma~\ref{ubdtoonto}(2), it suffices to prove that $\proj(\nu,\theta^+,\theta,1)$ holds.
Now appeal to Corollary~\ref{regularpairs}(1).

(5) As before, it suffices to prove that $\proj(\nu,\theta^+,\theta,1)$ holds.
Now appeal to Case~(1) of Lemma~\ref{bsicproj}.

(6) By Clause~(4), moreover $\onto^+(\{\nu\},\mathcal J^\kappa_{\nu^+},\theta^+)$ holds.

(7) By Clause~(6) together with Proposition~\ref{propD}(3) and the monotonicity property indicated in Remark~\ref{admonotonicity}.
\end{proof}

In particular, if there exists a Kurepa tree, then $\onto^+(\{\aleph_1\},\mathcal J^{\aleph_2}_{\aleph_2},\aleph_0)$ holds.\footnote{As for getting more than $\aleph_0$ many colours, note that $\onto(\{\aleph_1\},J^{\bd}[\aleph_2],\aleph_1)$ implies $\mathfrak b_{\aleph_1}=\aleph_2$ and $\onto(\{\aleph_1\},J^{\bd}[\aleph_2],\aleph_2)$ (See Corollary~\ref{lemma712}(3)  and \cite[Lemma~6.15(3)]{paper47}, respectively).}

\begin{cor} Suppose that $\theta<\nu<\kappa$ are infinite regular cardinals.
\begin{enumerate}[(1)]
\item If $\kappa\in\ad(\nu^+,\nu)$, then $\ubd^+(\{\nu\},\mathcal J^\kappa_{\nu^+},\theta)$ holds;
\item If $\kappa\in \ad(\nu,\theta)$ with $\theta$ regular, then $\ubd^+(\self, \mathcal J^\kappa_\nu, \theta)$ holds.
\end{enumerate}
\end{cor}
\begin{proof} (1) By Lemma~\ref{ad1} below.

(2) By Lemma~\ref{lemma65} below.
\end{proof}

\begin{thm}\label{ek} Suppose that $\nu=\nu^\theta<\cf(\kappa)\le 2^\nu$.

Then $\onto^{++}(\{\nu\},\mathcal J^\kappa_{\nu^+},\theta)$ holds.
\end{thm}
\begin{proof} By Proposition~\ref{prop614}(3), it suffices to prove that $\onto^{++}(\{\nu\},J^{\bd}[\kappa],\theta)$ holds.
Then, by the same trivial proof of Proposition 6.1 of \cite{paper47}, it suffices to prove that $\onto^{++}(\{\nu\},J^{\bd}[\varkappa],\theta)$ holds for $\varkappa:=\cf(\kappa)$.
Now, by the Engelking-Karlowicz theorem, we may fix a sequence of functions $\langle g_\eta\mid \eta<\nu\rangle$ such that:
\begin{itemize}
\item for every $\eta<\nu$, $g_\eta\in {}^\varkappa \theta$;
\item for every $X \in[\varkappa]^{\theta}$ and for every function $g:X\rightarrow\theta$, for some $\eta<\nu$, $g \s g_\eta$.
\end{itemize}

Let $c:[\varkappa]^2 \rightarrow \theta$ be any colouring that satisfies $c(\eta,\beta)=g_\eta(\beta)$ for all $\eta<\nu\le\beta<\varkappa$.
To see that $c$ witnesses $\onto^{++}(\{\nu\},J^{\bd}[\varkappa],\theta)$,
let $\langle B_\tau \mid \tau< \theta\rangle$ be any sequence of cofinal subsets of $\varkappa$.
For every $\epsilon<\varkappa$, since $\theta<\nu<\cf(\kappa)=\varkappa$, we may
fix an injection $f_\epsilon\in\prod_{\tau<\theta}B_\tau\setminus\epsilon$,
and then we find $\eta_\epsilon<\nu$ such that $f_\epsilon^{-1}\s g_{\eta_\epsilon}$.
As $\nu<\cf(\varkappa)$, find $\eta<\nu$ for which $\sup\{\epsilon<\varkappa\mid \eta_\epsilon=\eta\}=\varkappa$.
A moment's reflection makes it clear that $\sup\{ \beta\in B_\tau\mid c(\eta,\beta)=\tau\}=\varkappa$ for every $\tau<\theta$.
\end{proof}

\begin{cor}\label{alpehbeth} For every infinite cardinal $\theta$, $\onto^{++}(\{2^\theta\},J^{\bd}[2^{2^{\theta}}],\theta)$ holds.\qed
\end{cor}

By \cite[Proposition~7.7]{paper47}, if $\aleph_0<\kappa\le 2^{\aleph_0}$,
then $\onto(\{\aleph_0\},J^{\bd}[\kappa],n)$ holds for all $n<\omega$.
Using the ideas of the proof of Theorem~\ref{ek}, this is now improved as follows.

\begin{lemma}\label{ek2} Suppose that $\nu=\nu^\theta<\kappa\le 2^\nu$.

Then $\onto(\{\nu\},J^{\bd}[\kappa],\theta)$ holds iff $\cf(\kappa)\ge\theta$.
\end{lemma}
\begin{proof} The forward implication follows from the fact that $J^{\bd}[\kappa]$ contains a set of size $\cf(\kappa)$.
For the backward implication, fix a sequence of functions $\langle g_\eta\mid \eta<\nu\rangle$ such that:
\begin{itemize}
\item for every $\eta<\nu$, $g_\eta\in {}^\kappa \theta$;
\item for every $X \in[\kappa]^{\theta}$ and for every function $g:X\rightarrow\theta$, for some $\eta<\nu$, $g \s g_\eta$.
\end{itemize}

Let $c:[\kappa]^2 \rightarrow \theta$ be any colouring that satisfies $c(\eta,\beta)=g_\eta(\beta)$ for all $\eta<\nu\le\beta<\kappa$.
Now, let $B$ be any cofinal subset of $\kappa$.
Fix $X\in[B\setminus\nu]^\theta$ and some bijection $g:X\leftrightarrow\theta$.
Find $\eta<\nu$ such that $g \s g_\eta$. Then $c[\{\eta\}\circledast B]=\theta$.
\end{proof}

\begin{defn}[{\cite[Definition~1.9]{Sh:775}}] $\sep(\nu,\mu,\lambda,\theta,\varkappa)$ asserts the existence of a sequence $\langle g_\eta\mid \eta<\nu\rangle$
of functions from ${}^\mu\lambda$ to $\theta$, such that, for every function $g\in{}^\nu\theta$, the set
$$\mathcal F(g):=\{ f\in{}^\mu\lambda\mid \forall \eta<\nu\,(g_\eta(f)\neq g(\eta))\}$$ has size less than $\varkappa$.
\end{defn}

Sufficient conditions for $\sep(\ldots)$ to hold may be found at \cite[Claim~2.6]{Sh:898}.

\begin{lemma} \label{seplemma} Suppose that $\theta\le\nu<\cf(\kappa)\le\lambda^\mu$.

If $\sep(\nu,\mu,\lambda,\theta,\cf(\kappa))$ holds, then so does $\onto^+(\{\nu\},\mathcal J^\kappa_{\nu^+},\theta)$.
\end{lemma}
\begin{proof} Suppose that $\langle g_\eta\mid \eta<\nu\rangle$ is a sequence witnessing $\sep(\nu,\mu,\lambda,\theta,\varkappa)$
for $\varkappa:=\cf(\kappa)$.
By the same argument from the beginning of the proof of Theorem~\ref{ek},
in order to prove $\onto^+(\{\nu\},\mathcal J^\kappa_{\nu^+},\theta)$,
it suffices to prove $\onto(\{\nu\},J^{\bd}[\varkappa],\theta)$.
As $\varkappa\le\lambda^\mu$, fix an injective sequence $\langle f_\beta\mid \beta<\varkappa\rangle$ of functions from $\mu$ to $\lambda$.
Pick a function $c:[\varkappa]^2\rightarrow\theta$ that satisfies $c(\eta,\beta):=g_\eta(f_\beta)$ for all $\eta<\nu\le\beta<\varkappa$.

Towards a contradiction, suppose that $c$ fails to witness $\onto(\{\nu\},J^{\bd}[\varkappa],\theta)$
and pick a counterexample $B\in[\varkappa]^\varkappa$.
It follows that we may find a function $g:\nu\rightarrow\theta$ such that,
for every $\eta<\nu$, $g(\eta)\notin c[\{\eta\}\circledast B]$.
As $|\mathcal F(g)|<\varkappa=|B\setminus\nu|$,
we may pick $\beta\in B\setminus(\nu\cup\mathcal F(g))$.
As $\beta\notin\mathcal F(g)$, we may fix some $\eta<\nu$ such that $g_\eta(f_\beta)=g(\eta)$.
Altogether, $\eta<\nu\le\beta<\kappa$ and $c(\eta,\beta)=g_\eta(f_\beta)=g(\eta)$, contradicting the choice of $g$.
\end{proof}

Compared to Theorem~\ref{ek}, in the next result $\nu$ could have cofinality less than $\theta$,
and we settle for getting $\onto^+$ instead of $\onto^{++}$.

\begin{cor}\label{sepapp} Suppose $\theta<\nu<\cf(\kappa)\le2^\nu$ and $\nu$ is a strong limit, then
$\onto^+(\{\nu\},\mathcal J^\kappa_{\nu^+},\theta)$ holds.
\end{cor}
\begin{proof} By \cite[Claim~2.6(d)]{Sh:898}, if $\nu$ is a strong limit cardinal, $\theta<\nu$ and $\theta\neq\cf(\nu)$,
then $\sep(\nu,\nu,\theta,\theta,(2^\theta)^+)$ holds.
So, if $\nu$ is a strong limit, then for cofinally many cardinals $\theta<\nu$,
$\sep(\nu,\nu,\theta,\theta,\nu)$ holds.
Now, appeal to Lemma~\ref{seplemma} with $(\mu,\lambda):=(\nu,\theta)$, noting that $\theta^\nu=2^\nu$.
\end{proof}

\begin{prop}\label{propD}
Suppose that $\theta$ is an infinite regular cardinal.
\begin{enumerate}[(1)]
\item If there is a family of $\kappa$ many $\theta$-sized cofinal subsets of some cardinal $\lambda$ whose pairwise intersection has size less than $\theta$,
then $\kappa\in\ad(\lambda^+,\theta)$;
\item If there is a strictly $\s^*$-decreasing $\kappa$-sequence of sets in $[\theta]^\theta$, then $\kappa \in \ad(\theta^+, \theta)$.
In particular, if there is a strictly $<^*$-increasing $\kappa$-sequence of functions in ${}^\theta \theta$, then $\kappa \in \ad(\theta^+, \theta)$. So $\mathfrak b_\theta\in\ad(\theta^+,\theta)$;
\item If $\theta^{<\theta}=\theta$, then $2^\theta\in \ad(\theta^+,\theta)$.
\end{enumerate}
\end{prop}
\begin{proof}
(1) Given a subfamily $\mathcal X=\{ x_\alpha\mid\alpha<\kappa\}$ of $[\lambda]^{\theta}$ such that $\sup(x_\alpha)=\lambda$ and $|x_\alpha \cap x_\beta|<\theta$
for all $\alpha< \beta< \kappa$,
letting $\mathcal Y$ be a cofinal subset of $\lambda$ of order-type $\theta$,
it is easy to verify that $(\mathcal X,\mathcal Y)$ witnesses that $\kappa \in \ad(\lambda^+, \theta)$.

(2) Let $\langle Y_\alpha \mid \alpha< \kappa\rangle$ be a strictly $\s^*$-decreasing sequence of sets in $[\theta]^\theta$.
For every $\alpha< \kappa$, let $X_\alpha:= Y_\alpha \setminus Y_{\alpha+1}$. Then for every $\alpha< \beta< \kappa$, it is the case that $Y_\beta \s^* Y_{\alpha+1}$, and hence $|X_\alpha \cap X_\beta|< \theta$. It follows that $\langle X_\alpha \mid \alpha< \kappa\rangle$ is an almost disjoint family in $[\theta]^\theta$, so we finish by Clause (1). The rest is clear.

(3) That $\theta^{<\theta}= \theta$ gives rise to an almost disjoint family in $[\theta]^\theta$ of size $2^\theta$ is a standard fact. Namely, let $\langle f_\alpha \mid \alpha< \kappa\rangle$ be an injective sequence in ${}^\theta \theta$, and let $\pi: {}^{<\theta}\theta \leftrightarrow \theta$ be a bijection. Then for every $\alpha< \kappa$, let $X_\alpha:= \{\pi(f_\alpha \restriction \tau)\mid \tau< \theta\}$. It is easy to verify that $\langle X_\alpha \mid \alpha< \kappa\rangle$ is an almost disjoint family in $[\theta]^\theta$.
\end{proof}

\begin{prop}\label{prop63} Suppose that $\lambda^{<\theta}<\lambda^\theta$ with $\theta$ regular.

Then $\lambda^\theta\in\ad((\lambda^{<\theta})^+,\theta)$.
\end{prop}
\begin{proof} Fix an injection $f:{}^{<\theta}\lambda\rightarrow\lambda^{<\theta}$.
Set $\kappa:=\lambda^\theta$.
Let $\langle g_\alpha\mid\alpha<\kappa\rangle$ be an injective list of functions from $\theta$ to $\lambda$.
Put $\mathcal X:=\{ \{ f(g_\alpha\restriction\tau)\mid \tau<\theta\}\mid \alpha<\kappa\}$ and $\mathcal Y:=\{ f[{}^\tau\lambda]\mid \tau<\theta\}$.
Then $(\mathcal X,\mathcal Y)$ witnesses that $\kappa\in\ad((\lambda^{<\theta})^+,\theta)$.
\end{proof}
The preceding argument generalizes as follows.
\begin{lemma}
Suppose that $\theta$ is regular and there is a tree $(T,<_T)$ of height $\theta$ with at least $\kappa$ many $\theta$-branches
such that $|T_\gamma|<\mu$ for every $\gamma<\theta$. Then $\kappa\in \ad(\mu,\theta)$.
\end{lemma}
\begin{proof} Fix an injection $f:T\rightarrow\on$ satisfying that for all $\gamma<\delta<\theta$, and
$x\in T_\gamma$, and $y\in T_\delta$, it is the case that $f(x)<f(y)$.
Let $\langle \mathbf b_\alpha\mid\alpha<\kappa\rangle$ be an injective list of $\theta$-branches through the tree.
Set $\mathcal X:=\{ f[\mathbf b_\alpha]\mid \alpha<\kappa\}$ and
$\mathcal Y:=\{ f[T\restriction \tau]\mid \tau<\theta\}$.
Then $(\mathcal X,\mathcal Y)$ witnesses that $\kappa\in\ad(\mu,\theta)$.
\end{proof}

\begin{lemma}\label{ad1} Suppose that $\cf(\kappa)\in\ad(\mu,\theta)$ with $\theta$ regular.
If $\mu=\chi^+$ is a successor cardinal, then let $\nu:=\max\{\chi,\theta\}$. Otherwise, let $\nu:=\max\{\mu,\theta\}$.

If $\nu$ is less than $\kappa$,
then for every $\vartheta\in\reg(\theta)$,
$\ubd^+(\{\nu\},\allowbreak\mathcal J^\kappa_{\nu^+},\vartheta)$ holds.
\end{lemma}
\begin{proof} To avoid trivialities we can assume that $\cf(\kappa)\ge\nu^+$.
Pick a pair $(\mathcal X,\mathcal Y)$ witnessing that $\cf(\kappa)\in\ad(\mu,\theta)$.
Fix an injective enumeration $\langle x_\alpha \mid \alpha< \cf(\kappa)\rangle$ of $\mathcal X$.
Fix a $\s$-increasing enumeration $\langle y_\tau\mid\tau<\theta\rangle$ of $\mathcal Y$.
For every $\tau< \theta$, by Clauses (iii) and (iv) of Definition~\ref{adf}, we may fix an injection $f_\tau:\{ \min(x\setminus  y_\tau)\mid x\in\mathcal X\}\rightarrow\nu$.

Let $\vartheta\in\reg(\theta)$.  By Proposition~\ref{prop614}(1), it suffices to prove that $\ubd(\{\nu\},\allowbreak J^{\bd}[\kappa],\vartheta)$ holds.
Fix a $\vartheta$-bounded $C$-sequence $\langle C_\delta\mid \delta\in E^\nu_\vartheta\rangle$.
Fix a bijection $\pi:\nu\leftrightarrow\theta\times E^\nu_\vartheta$.
Fix a cofinal subset $U$ of $\kappa$ of order-type $\cf(\kappa)$. For every $\beta<\kappa$, let $\bar\beta:=\otp(\beta\cap U)$.
So $\beta\mapsto\bar\beta$ is a nondecreasing map from $\kappa$ to $\cf(\kappa)$ whose image is cofinal in $\cf(\kappa)$.

Finally, pick an upper-regressive colouring $c:[\kappa]^2\rightarrow\vartheta$ such that for all $\eta<\nu\le\beta<\kappa$,
if $\pi(\eta)=(\tau,\delta)$ and $f_\tau(\min(x_{\bar\beta}\setminus y_\tau))<\delta$,
then $$c(\eta,\beta):=\otp(C_\delta\cap f_\tau(\min(x_{\bar\beta}\setminus y_\tau))).$$

Now, given $B\in (J^{\bd}[\kappa])^+$, fix $B'\in[B\setminus\nu]^\vartheta$
on which the map $\beta\rightarrow\bar\beta$ is injective,
and then use Clause~(ii) of Definition~\ref{adf} to find a large enough $\tau<\theta$
such that for all $\alpha\neq\beta$ from $B'$,
$x_{\bar\alpha}\cap x_{\bar\beta}\s y_\tau$.
By Clause~(iii),
$\beta\mapsto \min(x_{\bar\beta}\setminus y_\tau)$ is a well-defined injection over $B'$.
So $H:=\{f_\tau(\min(x_{\bar\beta}\setminus y_\tau))\mid \beta\in B'\}$ has size $\vartheta$.
Pick $\delta\in E^\nu_\vartheta$ such that $\otp(H\cap\delta)=\vartheta$.
Fix $\eta<\nu$ such that $\pi(\eta)=(\tau,\delta)$.
Evidently, $\vartheta\ge\otp(c[\{\eta\}\circledast B])\ge \otp(c[\{\eta\}\times B'])=\vartheta$.
\end{proof}

\begin{lemma}\label{prop13} Suppose:
\begin{itemize}
\item $\vartheta<\theta\le\nu\le\kappa$;
\item $\theta$ and $\mu$ are regular cardinals;
\item $\varkappa:=\max\{\theta^+,\mu\}$ is less than $\cf(\kappa)$;
\item there is a tree $(T,<_T)$ of height $\theta$ with at least $\cf(\kappa)$ many $\theta$-branches
such that $|T_\gamma|<\mu$ and  $|T_\gamma|^\vartheta\le\nu$ for every $\gamma<\theta$.
\end{itemize}
Then $\onto^{++}(\{\nu\},\allowbreak\mathcal J^\kappa_\varkappa,\vartheta)$  holds.
\end{lemma}
\begin{proof}  Let $\langle \mathbf b_\beta \mid \beta<\cf(\kappa)\rangle$ be an injective enumeration of $\theta$-branches through the tree $(T,<_T)$.
We denote by $\Delta(\mathbf b_\alpha, \mathbf b_{\beta})$ the unique level in which $\mathbf b_\alpha$ splits from $\mathbf b_\beta$.
Fix a cofinal subset $U$ of $\kappa$ of order-type $\cf(\kappa)$. For every $\beta<\kappa$, let $\bar\beta:=\otp(\beta\cap U)$.
\begin{claim}\label{splitting} Let $J\in\mathcal J^\kappa_\varkappa$ and $B\in J^+$.
Then there is an $\alpha\in B$ such that the following set has order-type $\theta$:
$$B(\alpha) := \{\xi< \theta\mid \{\beta \in B \mid \Delta(\mathbf b_{\bar\alpha}, \mathbf b_{\bar\beta}) = \xi\}\in J^+\}.$$
\end{claim}
\begin{why} Suppose not. Then, for every $\alpha\in B$, $\otp(B(\alpha))<\theta$,
so that $\gamma_{\alpha}:= \ssup(B(\alpha))$ is less than $\theta$.
As $J$ is $\theta^+$-complete, we may find some $\gamma^* < \theta$ such that the set $B_0:=\{\alpha \in B \mid \gamma_{\alpha} = \gamma^*\}$ is in $J^+$.
Further, since $|T_{\gamma^*}|<\mu$ and $J$ is $\mu$-complete, we can also find a $t^*\in T_{\gamma^*}$ such that
the set $B_1:=\{\alpha \in B_0\mid \mathbf b_{\bar\alpha} \restriction \gamma^* = t^*\}$ is in $J^+$.
Let $\alpha:=\min(B_1)$ and $B_2:=B_1\setminus\{\alpha\}$.
Evidently, $B_2$ is in $J^+$ and for every $\beta\in B_2$, $\gamma^*\le\Delta(\mathbf b_{\bar\alpha}, \mathbf b_{\bar\beta})<\theta$.
Now, as $J$ is $\theta^+$-complete, we may pick $\xi<\theta$ for which $B_3:=\{\beta\in B_2\mid \Delta(\mathbf b_{\bar\alpha}, \mathbf b_{\bar\beta})=\xi\}$ is in $J^+$.
Then $B_3$ witnesses that $\xi\in B(\alpha)$, contradicting the fact that $\xi\ge\gamma^*=\ssup(B(\alpha))$.
\end{why}

Fix a surjection $f:\nu\rightarrow\bigcup_{\gamma<\theta}{}^\vartheta T_\gamma$.
Then fix a colouring $c:[\kappa]^2\rightarrow\vartheta$ satisfying that for every $\eta<\beta<\kappa$,
if $\eta<\nu$ and $\mathbf b_{\bar\beta}$ extends $f(\eta)(\tau)$ for some $\tau<\vartheta$,
then $c(\eta,\beta)$ is the least such $\tau$.

To see that $c$ is as sought, let $J\in\mathcal J^\kappa_\mu$,
and let $\langle B_\tau\mid \tau<\vartheta\rangle$ be a given sequence of $J^+$-sets.
Recursively define an injective transversal $\langle \alpha_\tau\mid \alpha<\vartheta\rangle\in\prod_{\tau<\vartheta}B_\tau$, as follows:

$\br$ Let $\alpha_0$ be the $\alpha$ given by the preceding claim with respect to $B:=B_0$.

$\br$ For every $\tau<\vartheta$ such that $\langle \alpha_\sigma\mid \sigma<\tau\rangle$ has already been defined,
let $\alpha_\tau$ be the $\alpha$ given by the preceding claim with respect to $B:=B_\tau\setminus\{\alpha_\sigma\mid \sigma<\tau\}$.
As $J$ is $\vartheta$-complete, $B$ is indeed in $J^+$.

As $\vartheta<\theta =\cf(\theta)$, $\gamma:=\ssup\{ \Delta(\mathbf b_{\bar\alpha_\sigma},\mathbf b_{\bar\alpha_\tau})\mid \sigma<\tau<\vartheta\}$
is less than $\theta$,
so we may find an ordinal $\eta<\nu$ such that, for every $\tau<\vartheta$, $f(\eta)(\tau)$ is the unique element of $T_\gamma$ that belongs to the branch $\mathbf b_{\bar\alpha_\tau}$.
\begin{claim} Let $\tau<\vartheta$. Then $\{\beta\in B_\tau\setminus(\eta+1)\mid c(\eta,\beta)=\tau\}$ is  in $J^+$.
\end{claim}
\begin{why} By the choice of $\alpha_\tau$,
we may pick $\xi\in B_{\tau}(\alpha_\tau)$ above $\gamma$.
As $J$ extends $J^{\bd}[\kappa]$, $B':=\{\beta \in B_\tau\setminus(\eta+1) \mid \Delta(\mathbf b_{\bar\alpha_\tau}, \mathbf b_{\bar\beta}) = \xi\}$ is in $J^+$.
Evidently, $c(\eta,\beta)=\tau$ for every $\beta\in B'$.
\end{why}
This completes the proof.
\end{proof}
By taking $\mu$ and $\nu$ to be $\kappa$ in Lemma~\ref{prop13}, we get:
\begin{cor} Suppose that $\vartheta<\cf(\theta)=\theta<\cf(\kappa)=\kappa$ are given cardinals.
If there is a tree $(T,<_T)$ of height $\theta$ with at least $\kappa$ many $\theta$-branches
such that, for every $\gamma<\theta$, $|T_\gamma|<\kappa$ and $|T_\gamma|^\vartheta\le\kappa$,
then $\onto^{++}(\mathcal J^\kappa_\kappa,\vartheta)$ holds.\qed
\end{cor}

\begin{lemma}\label{kurepalemma} Assume that $\nu$ is a regular uncountable cardinal and there is a $\nu$-Kurepa tree with $\kappa$-many branches. Let $\theta<\nu$.
\begin{enumerate}[(1)]
\item $\ubd^+(\{\nu\},\mathcal J^\kappa_{\nu^+},\theta)$ holds;
\item If $\proj(\nu,\nu,\theta,1)$ holds, then moreover $\onto^+(\{\nu\},\mathcal J^\kappa_{\nu^+},\theta)$ holds.
\end{enumerate}
\end{lemma}
\begin{proof} (1) This follows from Lemma~\ref{ad1}.

(2) Fix a $\nu$-Kurepa tree $T\s{}^{<\nu}2$.
Let $\langle \mathbf b_\beta \mid \beta< \kappa\rangle$ be an injective enumeration of $\nu$-branches through $T$.
Fix a bijection $\pi:\nu\leftrightarrow T\times\nu$.
Fix a map $p:\nu\times\nu\rightarrow\theta$ witnessing $\proj(\nu,\nu,\theta,1)$.
Finally define a colouring $c:[\kappa]^2\rightarrow\theta$, as follows. For all $\eta<\nu\le\beta<\kappa$,
if $\pi(\eta)=(t,\zeta)$ and $\Delta(t,\mathbf b_\beta)$ is defined, then let $c(\eta,\beta):=p(\zeta,\Delta(t,\mathbf b_\beta))$. Otherwise, let $c(\eta,\beta):=0$.

To see this works, let $B\in J^+$, for a given $J\in\mathcal J^\kappa_{\nu^+}$.
For all $\alpha<\nu$ and $\xi<\nu$, denote
$$B_\alpha^\xi:=\{\beta \in B \mid \Delta(\mathbf b_\alpha, \mathbf b_{\beta}) = \xi\}.$$
Using Claim~\ref{splitting}, pick $\alpha\in B$ such that the following set has order-type $\nu$:
$$X := \{\xi< \nu\mid B_\alpha^\xi\in J^+\}.$$
By the choice of $p$, pick $\zeta<\nu$ such that $p[\{\zeta\}\times X]=\theta$.
Now, find $X'\in[X]^\theta$ such that $p[\{\zeta\}\times X']=\theta$. Set $\gamma:=\ssup(X')$ and $t:=\mathbf b_\alpha\restriction\gamma$.
Pick $\eta<\nu$ such that $\pi(\eta)=(t,\zeta)$.
Then for every $\tau<\theta$, we may find $\xi\in X'$ such that $p(\zeta,\xi)=\tau$,
and then for every $\beta\in B_\alpha^\xi$, $c(\eta,\beta)=p(\zeta,\Delta(t,\mathsf b_\beta))=p(\zeta,\Delta(\mathbf b_\alpha,\mathbf b_\beta))=p(\zeta,\xi)=\tau$, as sought.
\end{proof}
\begin{cor}\label{kurepacor}  If $\theta$ is an infinite regular cardinal, and there exists a $\theta^+$-Kurepa tree with $\cf(\kappa)$ many branches,
then $\onto^+(\{\theta^+\},\mathcal J^\kappa_{\theta^{++}},\theta)$ holds.
\end{cor}
\begin{proof} By Lemma~\ref{kurepalemma}(2) together with Corollary~\ref{regularpairs}(1), the hypothesis implies that $\onto^+(\{\theta^+\},\mathcal J^{\cf(\kappa)}_{\theta^{++}},\theta)$ holds,
hence, so does $\onto^+(\{\theta^+\},\mathcal J^\kappa_{\theta^{++}},\theta)$.
\end{proof}

\begin{lemma}\label{lemma65} Suppose that $\kappa\in \ad(\mu,\theta)$ with $\theta$ regular. Set $\varkappa:=\max\{\mu,\theta^+\}$.

Then $\ubd^+(\self, \mathcal J^\kappa_\varkappa, \theta)$ holds.
\end{lemma}
\begin{proof} Pick a pair $(\mathcal X,\mathcal Y)$ witnessing that $\cf(\kappa)\in\ad(\mu,\theta)$.
Fix an injective enumeration $\langle x_\alpha \mid \alpha< \kappa\rangle$ of $\mathcal X$ and a $\s$-increasing enumeration $\langle y_\tau\mid\tau<\theta\rangle$ of $\mathcal Y$.
By Clause~(ii) of Definition~\ref{adf}, we may define a colouring $c:[\kappa]^2\rightarrow \theta$ via $$c(\alpha, \beta): = \min\{\tau<\theta\mid x_\alpha\cap x_\beta\s y_\tau\}.$$
Towards a contradiction, suppose that $c$ does not witness $\ubd^+(\self, \mathcal J^\kappa_\varkappa, \theta)$.
Fix $J \in\mathcal J^\kappa_{\varkappa}$ and $B \in J^+$ such that
for every $\eta \in B$ there is an $\epsilon_\eta< \theta$ such that
$$\{\tau<\theta\mid \{\beta\in B\setminus(\eta+1)\mid c(\eta,\beta)=\tau\}\in J^+\}\s \epsilon_\eta.$$
As $J$ is $\theta^+$-complete, we can then find an $\epsilon< \theta$ and a subset $B'\s B$ in $J^+$  such that for every $\eta \in B'$, $\epsilon_\eta= \epsilon$.
As $J$ is $\theta^+$-complete,
it follows that for every $\eta\in B'$, $N_\eta:=\{ \beta\in B\setminus(\eta+1)\mid c(\eta,\beta)\ge\epsilon\}$ is in $J$.
Now, as $J$ is $\mu$-complete, we may define an injection $f:\mu\rightarrow B'$ by recursion, letting $f(0):=\min(B')$,
and, for every nonzero $i<\mu$, $$f(i):=\min(B'\setminus\bigcup\nolimits_{\eta\in\im(f\restriction i)}N_\eta).$$
By the Dushnik-Miller theorem, fix $I\in[\mu]^\mu$ on which $f$ is order-preserving.
By Clause~(iii), we may define a function $g: I\rightarrow\on$ via $$g(i):=\min(x_{f(i)}\setminus y_\epsilon).$$
By Clause~(iv), we may now pick $(i,j)\in[I]^2$ such that $g(i)=g(j)$.
Denote $\eta:=f(i)$, $\beta:=f(j)$ and $\xi:=g(i)$.
Then $\eta<\beta$ and, by the definition of $f$, $\beta\in B'\setminus N_\eta$,
so that $x_\eta\cap x_\beta\s y_\tau$ where $\tau:=c(\eta,\beta)$ is less than $\epsilon$.
As $g(i)=\xi=g(j)$,
we infer that $\xi\in (x_\eta\cap x_\beta)\setminus y_\epsilon$, contradicting the fact that $y_\tau\s y_\epsilon$.
\end{proof}

\begin{cor}\label{prop52}
Suppose that $\theta\in\reg(\cf(\kappa))$
and that there is a tree of height $\theta$ with at least $\kappa$ many $\theta$-branches and all of whose levels have size less than $\cf(\kappa)$.
Then $\kappa\in\ad(\cf(\kappa),\theta)$ and hence $\ubd^+(\self,\mathcal J^\kappa_{\cf(\kappa)},\theta)$ holds.\qed
\end{cor}

\begin{cor}\label{cor68} Suppose that $\lambda^{<\theta}=\nu<\cf(\kappa)\le\lambda^\theta$ with $\theta$ regular.
Then:
\begin{enumerate}[(1)]
\item $\ubd^+(\self, \mathcal J^\kappa_{\nu^+}, \theta)$ holds;
\item $\onto^+(\mathcal J^\kappa_{\nu^+}, \theta)$ holds, provided that $\mathfrak d_\theta\le\kappa$.
\end{enumerate}
\end{cor}
\begin{proof}
By Proposition~\ref{prop63}, $\cf(\kappa)\in\ad(\nu^+,\theta)$.

(1)  By Lemma~\ref{lemma65},
$\ubd^+(\self, \mathcal J^{\cf(\kappa)}_{\nu^+}, \theta)$ holds,
and then $\ubd^+(\self,\allowbreak\mathcal J^\kappa_{\nu^+}, \theta)$ holds, as well.

(2) By Clause~(1) together with Lemmas \ref{d1} and \ref{ubdtoonto}(2).
\end{proof}

\begin{cor}\label{gs} Suppose that $\kappa=2^\theta$ for an infinite cardinal $\theta=\theta^{<\theta}$.

Then $\onto^+(\mathcal J^\kappa_{\theta^+},\theta)$ holds.\qed
\end{cor}
\begin{remark} The hypothesis ``$\theta=\theta^{<\theta}$'' cannot be waived. By \cite[Theorem~9.9]{paper47},
it is consistent that $\kappa=2^\theta=\theta^{<\theta}$ is weakly inaccessible,
and $\ubd(\ns_\kappa,\allowbreak\theta)$ fails.
\end{remark}

In contrast to Corollary~\ref{cmonotone},
it turns out that narrow colourings lack the feature of monotonicity.
\begin{cor} Suppose that $\kappa=\nu^+=2^\nu$ for a cardinal $\nu=\nu^{<\nu}$.
In some cofinality-preserving forcing extension, we have:
\begin{itemize}
\item $\ubd^+(\{\nu\},\allowbreak J^{\bd}[\kappa],\theta)$ holds for $\theta=\nu^+$;
\item $\ubd(\{\nu\},\allowbreak J^{\bd}[\kappa],\theta)$ fails for $\theta=\nu$;
\item $\onto^+(\{\nu\},\allowbreak J^{\bd}[\kappa],\theta)$ holds for every $\theta<\nu$.
\end{itemize}
\end{cor}
\begin{proof}
Use the forcing of \cite[\S3]{MR1355135} to make $\mathfrak b_\nu=\nu^{++}$, while preserving the cardinal structure and $\nu^{<\nu}=\nu$.
By Lemma~\ref{lemma712}(3), in the extension, $\ubd(\{\nu\},J^{\bd}[\nu^+],\allowbreak\nu)$ fails.
By Ulam's theorem, $\ubd^+(\{\nu\},\allowbreak  J^{\bd}[\nu^+],\nu^+)$ holds.
Finally, since $\nu^{<\nu}=\nu$, by Corollary~\ref{cor73},
$\onto^+(\{\nu\},\allowbreak J^{\bd}[\kappa],\theta)$ holds for every $\theta<\nu$.
\end{proof}

\section*{Acknowledgments}
The first author is supported by the Israel Science Foundation (grant agreement 2066/18).
The second author is partially supported by the European Research Council (grant agreement ERC-2018-StG 802756) and by the Israel Science Foundation (grant agreement 203/22).

Some of the results of this paper were presented by the first author at the seventh meeting of the \emph{Set Theory in the UK} on World Logic Day, 14 January 2022,
and by the second author at the first \emph{Gda\'nsk Logic Conference}, 05 May 2023.
We thank the organisers for the opportunity to speak.

\section*{Appendix: Index of results}\label{summary}

In all of the following tables, $\kappa$ stands for an uncountable cardinal.
\begin{table}[H]
$$\begin{array}{l|c|c|c|c}
\multicolumn{1}{c|}{\text{Hypotheses}}&\mathcal A&\mathcal J&\theta&\text{Reference}\\
\hline
\kappa=\theta^+&[\kappa]^\kappa&\mathcal J^\kappa_\kappa&\text{regular}&\ref{prop41}\\
\hdashline
\p^\bullet(\kappa,\kappa^+,{\sq},1)&[\kappa]^\kappa&\mathcal J^\kappa_\kappa&\text{less than }\kappa&\ref{proxybullet}\\
\hdashline
\kappa\nrightarrow[\kappa;\kappa]^2_\theta&[\kappa]^\kappa&\mathcal J^\kappa_\kappa&\text{less than }\kappa&\ref{prop46}\\
\hdashline
\kappa\nrightarrow[\kappa;\kappa]^2_\omega\text{ and }\chi(\kappa)>1&\{\kappa\}&\mathcal J^\kappa_\kappa&\text{in }\reg(\kappa)&\ref{cor52}\\
\hdashline
\exists T\in(\ns_\kappa)^+\,\Tr(T)\cap\reg(\kappa)=\emptyset&\{\kappa\}&\mathcal J^\kappa_\kappa&\text{in }\reg(\kappa)&\ref{cor52}\\
\hdashline
\kappa=\mathfrak b_{\theta^+}=\mathfrak d_{\theta^+}&\{\theta^+\}&\mathcal J^\kappa_{\theta^{++}}&\text{regular}&\ref{ccor68}\\
\hdashline
\nu=\nu^\theta<\cf(\kappa)\le 2^\nu&\{\nu\}&\mathcal J^\kappa_{\nu^+}&\text{less than }\nu&\ref{ek}\\
\hdashline
\kappa^{\aleph_0}=\cf(\kappa)\text{ and }\kappa\nrightarrow[\kappa;\kappa]^2_2&\{\kappa\}&\mathcal J^\kappa_\kappa&\aleph_0&\ref{rectangular}\\
\hdashline
\kappa^{\aleph_0}=\cf(\kappa)\text{ and }\kappa\nrightarrow[\kappa;\kappa]^2_2&\{\kappa\}&\mathcal S^\kappa_{\omega_1}&\aleph_0&\ref{thm814}
\end{array}$$
\caption{Sufficient conditions for $\onto^{++}(\mathcal A,\mathcal J,\theta)$ to hold.}
\end{table}

\begin{table}[H]
$$\begin{array}{l|c|c|c|c}
\multicolumn{1}{c|}{\text{Hypotheses}}&\mathcal A&\mathcal J&\theta&\text{Reference}\\
\hline
\kappa=\theta^+&[\kappa]^\kappa&\mathcal J^\kappa_\kappa&\text{singular}&\ref{thm54}\\
\hdashline
\cf([\theta]^{<\theta},{\s})<\kappa&[\kappa]^\kappa&\mathcal J^\kappa_\kappa&\text{in }\cspec(\kappa)&\ref{csingular}\\
\hdashline
\kappa=\mathfrak b_{\theta^+}=\mathfrak d_{\theta^+}&\{\theta^+\}&\mathcal J^\kappa_{\theta^{++}}&\text{singular}&\ref{ccor68}
\end{array}$$
\caption{Sufficient conditions for $\ubd^{++}(\mathcal A,\mathcal J,\theta)$ to hold.}
\end{table}

\begin{table}[H]
$$\begin{array}{l|c|c|c|c}
\multicolumn{1}{c|}{\text{Hypotheses}}&\mathcal A&\mathcal J&\theta&\text{Reference}\\
\hline
\kappa=\nu^+=2^\nu&[\kappa]^\nu&\mathcal J^\kappa_\kappa&\kappa&\ref{ulamext}\\
\hdashline
\p^\bullet(\kappa,\kappa^+,{\sq},1)&[\kappa]^\kappa&\mathcal J^\kappa_\kappa&\kappa&\ref{proxybullet}\\
\hdashline
\kappa=\cf(\kappa)\text{ and }\kappa\nrightarrow[\kappa]^2_\theta&\self&\mathcal J^\kappa_\kappa&\text{less than }\kappa&\ref{prop46}\\
\hdashline
\exists\theta^+\text{-Kurepa tree with }\cf(\kappa)\text{ branches}&\{\theta^+\}&\mathcal J^\kappa_{\theta^{++}}&\text{regular}&\ref{kurepacor}\\
\hdashline
\kappa\in\PP(\nu)\text{ and }\nu\text{ is singular}&\{\nu\}&\mathcal J^\kappa_{\nu^+}&\text{less than }\nu&\ref{thm811}\\
\hdashline
\nu<\cf(\kappa)\le2^\nu\text{ and }\nu\text{ is strong limit}&\{\nu\}&\mathcal J^\kappa_{\nu^+}&\text{less than }\nu&\ref{sepapp}\\
\hdashline
\lambda^{<\theta}=\nu<\cf(\kappa)\le\lambda^\theta\ \&\ \mathfrak d_\theta\le\kappa&\self&\mathcal J^\kappa_{\nu^+}&\text{regular}&\ref{cor68}\\
\hdashline
\kappa=\mathfrak d_\theta&\{\kappa\}&\mathcal J^\kappa_{\theta^+}&\text{regular}&\ref{cor55}\\
\hdashline
\kappa=2^\theta\text{ and }\theta=\theta^{<\theta}&\{\kappa\}&\mathcal J^\kappa_{\theta^+}&\text{regular}&\ref{gs}\\
\hdashline
\kappa=\cf(\kappa)\ge\mathfrak d\text{ and }\kappa\nrightarrow[\kappa]^2_2&\{\kappa\}&\mathcal J^\kappa_\kappa&\aleph_0&\ref{treelemma}
\end{array}$$
\caption{Sufficient conditions for $\onto^{+}(\mathcal A,\mathcal J,\theta)$ to hold.}
\end{table}

\begin{table}[H]
$$\begin{array}{l|c|c|c|c}
\multicolumn{1}{c|}{\text{Hypotheses}}&\mathcal A&\mathcal J&\theta&\text{Reference}\\
\hline
\exists T\in(\ns_\kappa)^+\,\Tr(T)\cap\reg(\kappa)=\emptyset&[\kappa]^\kappa&\mathcal J^\kappa_\kappa&\kappa&\ref{cor52}\\
\hdashline
\chi(\kappa)>1\text{ and }\kappa\nrightarrow[\kappa;\kappa]^2_\omega&[\kappa]^\kappa&\mathcal J^\kappa_\kappa&\kappa&\ref{cor52}\\
\hdashline
\chi(\kappa)>1&\self&\mathcal J^\kappa_\kappa&\kappa&\text{Thm}~\ref{thmb}\\
\hdashline
\exists\nu\text{-Kurepa tree with }\kappa\text{ branches}&\{\nu\}&\mathcal J^\kappa_{\nu^+}&\text{less than }\nu&\ref{kurepalemma}\\
\hdashline
\lambda^{<\theta}=\nu<\cf(\kappa)\le\lambda^\theta&\self&\mathcal J^\kappa_{\nu^+}&\text{regular}&\ref{cor68}\\
\hdashline
\kappa=\mathfrak b_\theta=\mathfrak d_\theta&[\theta]^\theta&\mathcal J^\kappa_{\theta^+}&\text{regular}&\ref{lemma66}\\
\hdashline
\kappa=\mathfrak b_\theta\text{ or }\kappa=\mathfrak d_\theta&\{\theta\}&\mathcal J^\kappa_{\theta^+}&\text{regular}&\ref{lemma712}\\
\hdashline
\kappa=\theta^+&\{\theta\}&\mathcal J^\kappa_{\kappa}&\text{singular}&\ref{cor215}\\
\end{array}$$
\caption{Sufficient conditions for $\ubd^{+}(\mathcal A,\mathcal J,\theta)$ to hold.}
\end{table}
\end{document}